\documentclass[11pt,reqno]{amsart}
\usepackage{amssymb,mathrsfs,graphicx,subfigure, enumerate}
\usepackage{amsmath,amsfonts,amssymb,amscd,amsthm,bbm}
\usepackage{extpfeil}
\usepackage{graphicx,colortbl}
\usepackage{float}
\usepackage{epsfig}
\usepackage{caption}
\usepackage{bm}
\graphicspath{{Figure/}}
\definecolor{ForestGreen}{rgb}{0.13,0.55,0.13}
\usepackage{comment}

\usepackage[margin=1in]{geometry}
\makeatletter
\@namedef{subjclassname@2020}{\textup{2020} Mathematics Subject Classification}
\makeatother

\title[Navier--Stokes--Fourier--Korteweg equations]{Time-asymptotic stability of generic Riemann solutions for viscous-dispersive heat-conducting compressible fluids}

\author[Han]{Sungho Han$^{*}$}
\address[Sungho Han]{\newline Department of Mathematical Sciences \newline Korea Advanced Institute of Science and Technology, Daejeon  34141, Republic of Korea}
\email{sungho\_han@kaist.ac.kr}
\urladdr{https://orcid.org/0009-0003-3117-2542}

\author[Kim]{Jeongho Kim}
\address[Jeongho Kim]{\newline Department of Applied Mathematics, \newline Kyung Hee University, 1732, Deogyeong-daero, Giheung-gu, Yongin-si, Gyeonggi-do 17104, Republic of Korea}
\email{jeonghokim@khu.ac.kr}
\urladdr{https://orcid.org/0000-0002-5220-3139}

\begin{document}
\newtheorem{theorem}{Theorem}[section]
\newtheorem{lemma}{Lemma}[section]
\newtheorem{corollary}{Corollary}[section]
\newtheorem{proposition}{Proposition}[section]
\newtheorem{remark}{Remark}[section]
\newtheorem{definition}{Definition}[section]
\renewcommand{\theequation}{\thesection.\arabic{equation}}
\renewcommand{\thetheorem}{\thesection.\arabic{theorem}}
\renewcommand{\thelemma}{\thesection.\arabic{lemma}}
\newcommand{\R}{\mathbb R}
\newcommand{\bbr}{\mathbb{R}}
\newcommand{\e}{\varepsilon}
\newcommand{\pa}{\partial}

\newcommand{\pv}{p(v)}

\newcommand{\norm}[1]{\left\lVert#1\right\rVert}

\newcommand{\tm}{\widetilde{\mu}}
\newcommand{\tk}{\widetilde{\kappa}}



\thanks{$^{*}$Corresponding author.}

\begin{abstract}
We establish the time-asymptotic stability of generic Riemann solutions for the one-dimensional Navier--Stokes--Fourier--Korteweg (NSFK) system, which describes the dynamics of viscous, heat-conducting, compressible fluids with internal capillarity. We prove that, for the adiabatic exponent $\gamma\in(1,2]$, the solution to the NSFK system converges uniformly in space, as time tends to infinity, to a composite wave consisting of a rarefaction wave, a viscous contact wave, and a viscous-dispersive shock with a time-dependent shift, provided that the initial perturbation and the strengths of the contact wave and the shock are sufficiently small. In particular, we do not impose any smallness condition on the strength of the rarefaction wave. Because of the capillarity, the shock is viscous-dispersive and its profile equation is of third order. We treat the shock by the method of $a$-contraction with shifts introduced by Kang and Vasseur. This is the first result on the time-asymptotic stability of a viscous-dispersive shock, and of a generic Riemann solution, for the NSFK system.
\end{abstract}

\maketitle

\section{Introduction}

We consider a one-dimensional compressible viscous, heat conducting, and capillary gas dynamics model written in the Lagrangian mass coordinate:
\begin{align}
\begin{aligned}\label{eq:NSK}
    &v_t - u_x = 0,\quad t>0,\quad x\in\R,\\
    &u_t + p_x = \mu\left(\frac{u_x}{v}\right)_x + \kappa\left( -\frac{v_{xx}}{v^5} + \frac{5v_x^2}{2v^6} \right)_x,\\
    &E_t +(pu)_x = \alpha\left(\frac{\theta_x}{v}\right)_x+\mu\left(\frac{uu_x}{v}\right)_x+\kappa\left(\left(-\frac{v_{xx}}{v^5}+\frac{5v_x^2}{2v^6}\right)u+\frac{v_xu_x}{v^5}\right)_x.
\end{aligned}
\end{align}
Here, $v$, $u$, and $E$ denote the specific volume, the velocity, and the specific total energy of the fluid, respectively. The positive constants $\mu$, $\alpha$, and $\kappa$ are the viscosity, heat conductivity, and internal capillarity coefficients, and $p$ represents the pressure. The total energy of the non-isothermal capillary flow is given by
\begin{equation}\label{eq:total_energy}
	E=e+\frac{u^2}{2}=\frac{R\theta}{\gamma-1}+\frac{\kappa v_x^2}{2v^5}+\frac{u^2}{2},
\end{equation}
where $e$ is the internal energy of the fluid that incorporates the gradient of the density (or equivalently, the gradient of the specific volume), $\theta$ is the absolute temperature, $R$ is the gas constant, and $\gamma>1$ is the adiabatic exponent. Substituting \eqref{eq:total_energy} into the energy equation \eqref{eq:NSK}$_3$, one can reformulate \eqref{eq:NSK} in terms of $(v,u,\theta)$ as
\begin{align}
	\begin{aligned}\label{eq:NSK-theta}
		&v_t - u_x = 0,\quad t>0,\quad x\in\R,\\
		&u_t + p_x = \mu\left(\frac{u_x}{v}\right)_x + \kappa\left( -\frac{v_{xx}}{v^5} + \frac{5v_x^2}{2v^6} \right)_x,\\
		&\frac{R}{\gamma-1}\theta_t +pu_x=\alpha\left(\frac{\theta_x}{v}\right)_x+\frac{\mu u_x^2}{v},
	\end{aligned}
\end{align}
where the pressure is given by $p(v,\theta) = \frac{R\theta}{v}$. We refer to the system \eqref{eq:NSK-theta} as the Navier--Stokes--Fourier--Korteweg (NSFK) system. 

The fluid with internal capillarity was first introduced by van der Waals \cite{W94} and then by Korteweg \cite{K01}, by considering an internal energy that may depend on the gradient of the density. Later on, Dunn and Serrin rigorously derived the modern form of the system \eqref{eq:NSK} based on the thermodynamics of interstitial working \cite{DS85}. From the mathematical point of view, there have been numerous works on the well-posedness \cite{CH13,DD01,GL16,H09,H11,H17,H23,HL96} and, for the barotropic system, on the weak-strong uniqueness \cite{BGL19,GLT17} and the vanishing viscosity and/or capillarity limit \cite{CH13,BH22}. In particular, when the capillarity is given as $\kappa(v) = v$, it corresponds to the celebrated quantum Navier--Stokes equations \cite{AM09,AS22,BGL19,J10,LV18}.

We consider the initial data $(v,u,\theta)(0,x)=(v_0,u_0,\theta_0)(x)$ that converge to two different constant states at the far-field, that is
\[\lim_{x\to\pm\infty}(v_0,u_0,\theta_0)(x)=(v_\pm,u_\pm,\theta_\pm).\]
Then, the long-time behavior of the solution to \eqref{eq:NSK-theta} is determined by the Riemann problem of the non-isothermal Euler system ($\mu=\kappa=\alpha=0$ in \eqref{eq:NSK})
\begin{align}
\begin{aligned}\label{eq:Euler}
	&v_t - u_x = 0,\quad t>0,\quad x\in\R,\\
	&u_t + p_x = 0,\\
	&\Bigl(\frac{R\theta}{\gamma-1}+\frac{u^2}{2}\Bigr)_t+(pu)_x=0,
\end{aligned}
\end{align}
subject to the Riemann initial data
\begin{equation}\label{eq:Riemann_initial}
	(v,u,\theta)(0,x)=\begin{cases}
	(v_-,u_-,\theta_-),\quad &\mbox{if $x<0$},\\
	(v_+,u_+,\theta_+),\quad &\mbox{if $x>0$}.
\end{cases}
\end{equation}

The long-time behavior of the compressible fluid model has been extensively investigated in the literature. For the barotropic Navier--Stokes equation, Matsumura and Nishihara investigated the stability of a single viscous shock wave \cite{MN85} and of a rarefaction wave \cite{MN86}, under smallness conditions on both the initial perturbation and the wave strength. Later in \cite{MN92}, the stability of the superposition of two rarefaction waves was shown without assuming smallness of the initial perturbation or of the rarefaction strength. The stability of viscous shock waves was further investigated by Liu \cite{L85}, Szepessy--Xin \cite{SX93}, and Liu--Zeng \cite{LZ15}, where the zero-mass condition was removed by introducing constant shifts and diffusion waves. Mascia and Zumbrun \cite{MZ04} established the nonlinear orbital stability of small-amplitude viscous shock profiles for general symmetric hyperbolic--parabolic systems. The stability of the superposition of two viscous shock waves was first considered in \cite{HM09} for the full Navier--Stokes system, under the condition that the two shock strengths are comparable, and this restriction was removed in \cite{HKK23} for the barotropic system. Finally, the stability of the superposition of a rarefaction wave and a viscous shock wave was established in \cite{KVW23}, which completes the classification of all possible asymptotic patterns for the barotropic Navier--Stokes equations. For the degenerate viscosity case, the stability of shock waves was first studied in \cite{MW10}, and it was further generalized in \cite{VY16}. 

For the full Navier--Stokes system, the stability of rarefaction waves was proved in \cite{KMN86} and \cite{LX88} under the smallness assumption on the initial perturbation and strength of rarefaction. Later, the smallness condition on the strength of the rarefaction wave was removed in \cite{NYZ04}. The stability of the superposition of shock waves for the full Navier--Stokes system was obtained in \cite{HM09} for comparable shock strengths, as mentioned above. Finally, the stability of the contact discontinuity was proved in \cite{HXY08}, and the stability of the superposition of rarefaction and contact discontinuity was proved in \cite{HLM10}. The long-standing open question of stability of rarefaction, contact discontinuity, and shock wave was recently resolved by Kang--Vasseur--Wang \cite{KVW25}.

On the other hand, the long-time behavior of the viscous-capillary fluid is a relatively new subject. The first result on this topic is \cite{C12}, where the stability of strong rarefaction waves for the Navier--Stokes--Korteweg system was proved. Then, the stability of the viscous-dispersive shock wave \cite{CHZ15,HKKL25} and of the superposition of a rarefaction wave and a shock wave \cite{HK26} were established subsequently. We also mention that the spectral stability of viscous-dispersive shocks was investigated in the recent works \cite{FPZ23,FLP26}. Finally, for the NSFK system, the stability of the viscous contact wave \cite{CX13} and of the superposition of rarefaction waves and a viscous contact wave \cite{CS19,CXM14} have been studied. To the best of our knowledge, the long-time behavior of generic Riemann solutions to the NSFK system, given by a superposition of a rarefaction wave, a contact discontinuity, and a viscous-dispersive shock wave, has not been considered and remains an open problem. In the present paper, we study this problem.

To set the stage, we first review each of the elementary waves for the Euler system \eqref{eq:Euler}. In general, the Riemann solution of the Euler system \eqref{eq:Euler} consists of a 1-wave, a 2-contact discontinuity, and a 3-wave, where each of the 1- and 3-waves is a rarefaction wave or a shock wave. The Riemann solution is called generic when it consists of a rarefaction wave, a contact discontinuity, and a shock wave. For a given state $(v_L,u_L,\theta_L)\in \R_+\times\R\times\R_+$ with $\R_+:=(0,\infty)$, the 1-rarefaction curve $R_1(v_L,u_L,\theta_L)$ is defined as
\[R_1(v_L,u_L,\theta_L):=\left\{(v,u,\theta)~\Big|~v>v_L,\,s(v,\theta)=s(v_L,\theta_L)=:s_L,\,u=u_L-\int_{v_L}^v\lambda_1(s_L,v')\,dv'\right\},\]
where $s(v,\theta)=\frac{R}{\gamma-1}\ln(\theta v^{\gamma-1}) + \text{constant}$ is the specific entropy and $\lambda_1(s,v)=-\sqrt{\frac{\gamma p(v,s)}{v}}$ is the first eigenvalue of the Jacobian matrix of the flux function for the Euler system \eqref{eq:Euler}.
Then, for any $(v_R,u_R,\theta_R)\in R_1(v_L,u_L,\theta_L)$, the Riemann solution to the Euler system with the left end state $(v_L,u_L,\theta_L)$ and right end state $(v_R,u_R,\theta_R)$ is given by the 1-rarefaction wave $(v^r,u^r,\theta^r)(x/t)$, which is a self-similar continuous weak solution to \eqref{eq:Euler}.
For a state $(v_*,u_*,\theta_*)\in R_1(v_-,u_-,\theta_-)$, we also introduce a smooth approximation of the 1-rarefaction wave connecting $(v_-,u_-,\theta_-)$ to $(v_*,u_*,\theta_*)$ as in \cite{MN92}, on which the initial condition of the main theorem is formulated. Set $w_-:=\lambda_1(v_-,\theta_-)$ and $w_*:=\lambda_1(v_*,\theta_*)$, where $\lambda_1(v,\theta):=-\sqrt{\gamma R\theta}/v$ is $\lambda_1(s,v)$ written in the variables $(v,\theta)$, so that $w_-<w_*$. We pose the inviscid Burgers equation with the smooth initial datum $w_0$:
\begin{equation}\label{eq: inviscid Burgers equation}
	\begin{cases}
		w_t+w w_x=0,\\
		w(0,x)= w_0(x)= \frac{w_*+w_-}{2}+\frac{w_*-w_-}{2}K_q\int_0^{\e_R x}\frac{1}{(1+y^2)^q}\,dy,
	\end{cases}
\end{equation} 
where $w_-<w_*$ are the far-field states as $x\to\pm\infty$, $q \ge 14$ is a fixed constant, $K_q$ is the normalizing constant satisfying $K_q\int_0^\infty (1+y^2)^{-q}\,dy=1$, and $0<\e_R\le1$ is a small smoothing parameter measuring the initial slope of the wave. The condition $q\ge14$ gives the rate $O(\delta_0^{1/6})$ in Proposition~\ref{prop:apriori}, since $\frac{q-8}{3(q-2)}\ge\frac16$ exactly when $q\ge14$ (see Section~\ref{sec:L2}). Then, for a given $\e_R\in(0,1]$, we define a smooth approximation $(v^R,u^R,\theta^R)(t,x)$ of the 1-rarefaction wave $(v^r,u^r,\theta^r)$ as
    \begin{equation} \label{smooth rarefaction}
        \begin{aligned}
            &\lambda_1(v^R,\theta^R) = w(1+t,x),\\
            &z_1(v^R,u^R)(t,x) = z_1(v_-,u_-)=z_1(v_*,u_*),\\
            &s(v^R,\theta^R)(t,x) = s(v_-,\theta_-)=s(v_*,\theta_*),
        \end{aligned}
    \end{equation}
where $w(t,x)$ is the smooth solution to \eqref{eq: inviscid Burgers equation} with parameter $\e_R$ and $z_1(v,u)=u+\int^v \lambda_1(s,v')\,dv'$. {Throughout the paper, the smoothing parameter is fixed as $\e_R=\delta_0/2$, where $\delta_0$ is the constant in Theorem~\ref{thm:main}.} The strength of the rarefaction wave is $\delta_R:=|v_*-v_-|\sim|u_*-u_-|\sim|\theta_*-\theta_-|$. The properties of $w$ and $(v^R,u^R,\theta^R)$ are summarized in Lemmas~\ref{lem:Burgers} and~\ref{lem:rarefaction_property}.

Next, for any $(v_L,u_L,\theta_L)\in \R_+\times\R\times\R_+$, the 2-contact discontinuity curve $CD_2(v_L,u_L,\theta_L)$ is given by
\begin{equation*}
    CD_2(v_L,u_L,\theta_L):=\left\{(v,u,\theta) \left| \right. u=u_L,\,p(v,\theta)=p(v_L,\theta_L)\right\}.
\end{equation*}
Then, for any $(v_R,u_R,\theta_R)\in CD_2(v_L,u_L,\theta_L)$, the Riemann solution with the left end state $(v_L,u_L,\theta_L)$ and right end state $(v_R,u_R,\theta_R)$ is given by the 2-contact discontinuity $(v^c,u^c,\theta^c)$ defined by
\begin{equation*}
    (v^c,u^c,\theta^c)(t,x)=\begin{cases}
	(v_L,u_L,\theta_L),\quad & x<0,\\
	(v_R,u_R,\theta_R),\quad & x>0.
\end{cases}
\end{equation*}
Finally, for any $(v_L,u_L,\theta_L)\in \R_+\times\R\times\R_+$, consider the Rankine--Hugoniot condition
\begin{align}
	\begin{aligned}\label{RH-condition}
		&-\sigma(v-v_L)-(u-u_L)=0,\\
		&-\sigma(u-u_L)+(p(v,\theta)-p(v_L,\theta_L))=0,\\
		&-\sigma(E-E_L)+(p(v,\theta)u-p(v_L,\theta_L)u_L)=0,
	\end{aligned}
\end{align}
where $E_L=\frac{R\theta_L}{\gamma-1}+\frac{u_L^2}{2}$. Among the two possible values of $\sigma$ satisfying \eqref{RH-condition}, we consider the positive one $\sigma=\sqrt{-\frac{p(v,\theta)-p(v_L,\theta_L)}{v-v_L}}$. We define the 3-shock curve $S_3(v_L,u_L,\theta_L)$ as
\begin{equation*}
    S_3(v_L,u_L,\theta_L):=\left\{(v,u,\theta)~\Big|~v>v_L,\,\sigma=\sqrt{-\frac{p(v,\theta)-p(v_L,\theta_L)}{v-v_L}},\, \mbox{and \eqref{RH-condition} is satisfied}\right\}.
\end{equation*}

Then, for any $(v_R,u_R,\theta_R)\in S_3(v_L,u_L,\theta_L)$, the Riemann solution with the left end state $(v_L,u_L,\theta_L)$ and right end state $(v_R,u_R,\theta_R)$ is given by 3-shock wave $(v^s,u^s,\theta^s)(t,x)$ given as
\[(v^s,u^s,\theta^s)(t,x)=\begin{cases}
	(v_L,u_L,\theta_L),\quad & x<\sigma t,\\
	(v_R,u_R,\theta_R),\quad & x>\sigma t.
\end{cases}\]

As mentioned, the purpose of this paper is to investigate the time-asymptotic behavior of the solution $(v,u,\theta)$ to \eqref{eq:NSK-theta} with the far-field states $(v_\pm,u_\pm,\theta_\pm)$, where the solution to the Riemann problem \eqref{eq:Euler}--\eqref{eq:Riemann_initial} is given by a composite wave of 1-rarefaction, 2-contact discontinuity, and 3-shock. Specifically, we consider the end-states $(v_\pm,u_\pm,\theta_\pm)$ for which there exist two intermediate states $(v_*,u_*,\theta_*)$ and $(v^*,u^*,\theta^*)$ such that
\[(v_*,u_*,\theta_*)\in R_1(v_-,u_-,\theta_-),\quad (v^*,u^*,\theta^*)\in CD_2(v_*,u_*,\theta_*),\quad (v_+,u_+,\theta_+)\in S_3(v^*,u^*,\theta^*).\]
Due to the viscous, heat-conducting, and capillary nature of the NSFK system, the time-asymptotic behavior of \eqref{eq:NSK-theta} is expected to be given as a superposition of the rarefaction wave, the viscous contact wave, and the viscous-dispersive shock wave. To define the viscous contact wave, we consider a nonlinear diffusion equation as in \cite{HXY08}:

\begin{equation}\label{eq:nonlinear_diffusion}
	\Theta_t  = \frac{(\gamma-1)\alpha p_*}{R^2\gamma} \left(\frac{\Theta_x}{\Theta}\right)_x,\quad \Theta(t,-\infty)=\theta_*,\quad \Theta(t,+\infty) = \theta^*,
\end{equation}
where $p_*:=p(v_*,\theta_*)=p(v^*,\theta^*)$ denotes the common pressure of the two states adjacent to the contact discontinuity. Let $\Theta(t,x)=\Theta(x/\sqrt{1+t})$ be the self-similar solution of \eqref{eq:nonlinear_diffusion} constructed in \cite{HXY08}. Then, the viscous contact wave $(v^C,u^C,\theta^C)$ is defined as
\begin{equation}\label{def:viscous-contact}
	v^C:=\frac{R}{p_*}\Theta,\quad u^C:=u_*+\frac{(\gamma-1)\alpha}{R\gamma}\frac{\Theta_x}{\Theta},\quad \theta^C:=\Theta.
\end{equation}
The strength of the contact wave is $\delta_C:=|v^*-v_*|=\frac{R}{p_*}|\theta^*-\theta_*|$. The properties of the viscous contact wave are summarized in Lemma~\ref{lem:viscous_contact}.
On the other hand, the profile of the viscous-dispersive shock wave $(v^S,u^S,\theta^S)(\xi)$ with $\xi=x-\sigma t$ is given by the traveling wave solution to \eqref{eq:NSK}. Substituting the traveling wave ansatz $(v,u,E)(t,x)=(v^S,u^S,E^S)(x-\sigma t)$ into \eqref{eq:NSK}, the profile satisfies the following system of ODEs: 
\begin{align}
	\begin{aligned}\label{eq:viscous_shock}
		-\sigma (v^S)' - (u^S)' &= 0, \quad '=\frac{d}{d\xi},\\
		-\sigma (u^S)' + (p^S)' &= \left(\mu \frac{(u^S)'}{v^S} \right)' + \kappa\left( - \frac{(v^S)''}{(v^S)^5} + \frac{5((v^S)')^2}{2(v^S)^6} \right)', \\
		-\sigma (E^S)' + (p^S u^S)' &= \alpha \left( \frac{(\theta^S)'}{v^S} \right)' + \mu \left( \frac{u^S (u^S)'}{v^S} \right)' + \kappa \left(  \left( - \frac{(v^S)''}{(v^S)^5} + \frac{5((v^S)')^2}{2(v^S)^6}\right)u^S + \frac{(v^S)' (u^S)'}{(v^S)^5}  \right)',
	\end{aligned}
\end{align}
together with the limit $\lim_{\xi\to -\infty} (v^S,u^S,\theta^S)=(v^*,u^*,\theta^*)$ and $\lim_{\xi\to\infty}(v^S,u^S,\theta^S)=(v_+,u_+,\theta_+)$, where $p^S:=p(v^S,\theta^S)$, and $E^S:=\frac{R\theta^S}{\gamma-1}+\frac{(u^S)^2}{2}+\frac{\kappa((v^S)')^2}{2(v^S)^5}$. The strength of the shock is $\delta_S:=|v_+-v^*|\sim|u_+-u^*|\sim|\theta_+-\theta^*|$. The existence and uniqueness, up to translation, of a monotone viscous-dispersive shock profile and its quantitative properties are established in the companion paper \cite{H26} and are summarized in Lemma \ref{lem:shock-property}.

We are now ready to present the main result of the present paper, which is the time-asymptotic stability of the composite wave of the 1-rarefaction wave, the 2-viscous contact wave, and the 3-viscous-dispersive shock wave.

\begin{theorem}\label{thm:main}

	Let $\gamma\in (1,2]$ and $q\ge14$. For given constant states $(v_-,u_-,\theta_-)\in\R_+\times\R\times\R_+$ and $(v_*,u_*,\theta_*)\in R_1(v_-,u_-,\theta_-)$, there exist constants $\delta_0,\e_0>0$ such that the following holds.
	For any $(v^*,u^*,\theta^*)\in CD_2(v_*,u_*,\theta_*)$ and $(v_+,u_+,\theta_+)\in S_3(v^*,u^*,\theta^*)$ satisfying
	\[|v^*-v_*|+|v_+-v^*|<\delta_0,\]
	we denote by $(v^r,u^r,\theta^r)(\frac xt)$, $(v^C,u^C,\theta^C)(t,x)$, and $(v^S,u^S,\theta^S)(x-\sigma t)$ the 1-rarefaction wave, the viscous contact wave \eqref{def:viscous-contact}, and the viscous-dispersive shock wave of Lemma~\ref{lem:shock-property} connecting these end states, and by $(v^R,u^R,\theta^R)(t,x)$ the approximate rarefaction wave \eqref{smooth rarefaction} with $\e_R=\delta_0/2$. Let $(v_0,u_0,\theta_0)$ be any initial data such that
	\begin{equation}\label{initial_data}
	\begin{aligned}
		&\bigl\|(v_0,u_0,\theta_0)-(v^R,u^R,\theta^R)(0,\cdot)\bigr\|_{L^2(\R_-)}\\
		&\quad+\bigl\|(v_0,u_0,\theta_0)-(v^R,u^R,\theta^R)(0,\cdot)-(v_+-v_*,u_+-u_*,\theta_+-\theta_*)\bigr\|_{L^2(\R_+)}\\
		&\quad+\bigl\|(v_0-v^R(0,\cdot))_x\bigr\|_{H^1(\R)}
		+\bigl\|\bigl((u_0-u^R(0,\cdot))_x,(\theta_0-\theta^R(0,\cdot))_x\bigr)\bigr\|_{L^2(\R)}<\e_0 ,
	\end{aligned}
	\end{equation}
	where $\R_-:=(-\infty,0)$.
	Then the NSFK system \eqref{eq:NSK-theta} admits a unique global-in-time solution $(v,u,\theta)(t,x)$ for all $t\ge0$. Moreover, there exists a Lipschitz shift $X(t)$ with $X(0)=0$ such that the superposition wave
	\begin{equation*}
    \begin{aligned}(\bar v,\bar u,\bar\theta)(t,x):={}&(v^R,u^R,\theta^R)(t,x)+(v^C,u^C,\theta^C)(t,x)+(v^S,u^S,\theta^S)(x-\sigma t-X(t))\\&-(v_*,u_*,\theta_*)-(v^*,u^*,\theta^*)
    \end{aligned}
    \end{equation*}
	satisfies
	\begin{equation}\label{eq:regularity}
	\begin{aligned}
		&v-\bar v\in C([0,\infty);H^2(\R)),\qquad (v-\bar v)_x\in L^2(0,\infty;H^2(\R)),\\
		&(u-\bar u,\theta-\bar\theta)\in C([0,\infty);H^1(\R)),\qquad (u-\bar u,\theta-\bar\theta)_x\in L^2(0,\infty;H^1(\R)),
	\end{aligned}
	\end{equation}
	and
	\begin{align}
		\begin{aligned}\label{est:large_time}
		\lim_{t\to\infty}\sup_{x\in \R} \Big|(v,u,\theta)(t,x)-\Big(&v^r\left(\frac{x}{t}\right)+v^C(t,x)+v^S(x-\sigma t-X(t))-v_*-v^*,\\
		&u^r\left(\frac{x}{t}\right)+u^C(t,x)+u^S(x-\sigma t-X(t))-u_*-u^*,\\
		&\theta^r\left(\frac{x}{t}\right)+\theta^C(t,x)+\theta^S(x-\sigma t-X(t))-\theta_*-\theta^*\Big)\Big|=0 .
		\end{aligned}
	\end{align}
	In addition, the shift satisfies
	\begin{equation}\label{est:shift}
		\lim_{t\to\infty}|\dot{X}(t)|=0.
	\end{equation}

\end{theorem}

\begin{remark}
	\begin{enumerate}
		\item In Theorem \ref{thm:main}, we do not assume the smallness condition on the rarefaction strength $|v_--v_*|\sim|u_--u_*|\sim|\theta_--\theta_*|$. Therefore, we cover the case of a large rarefaction wave. Instead, we restrict the adiabatic exponent to $\gamma\in (1,2]$, whereas the full range $\gamma\in (1,\infty)$ was covered in \cite{KVW25} for the Navier--Stokes--Fourier system. The restriction $\gamma\le2$ is used only in \eqref{thetaRvR}, where it makes $v^R\theta^R$ nondecreasing in $v^R$ and gives the sign $v^R\theta^R\le v_*\theta_*$ in the lower bound of the diffusion term for the temperature perturbation, as in \cite{HKLO26}.
	\item Since the rarefaction wave may be large while the initial perturbation is small, the initial condition \eqref{initial_data} measures the initial data from the approximate rarefaction wave with the smoothing parameter $\e_R=\delta_0/2$ at $t=0$. The approximate wave converges to the rarefaction wave as $t\to\infty$ by Lemma~\ref{lem:rarefaction_property}~(5).
	\item Thanks to the shift estimate \eqref{est:shift}, we have $|X(t)|=o(t)$ as $t\to\infty$. Therefore, the location of the shock wave is asymptotically $\sigma t$.
	\item The change of variables $(x,u)\mapsto(-x,-u)$ leaves \eqref{eq:NSK-theta} invariant and exchanges the 1-waves and the 3-waves. Hence Theorem~\ref{thm:main} also covers the generic Riemann solutions consisting of a 1-shock, a 2-contact discontinuity, and a 3-rarefaction wave.
	\end{enumerate}
\end{remark}

\subsection*{Novelty of the paper}
The nonlinear stability of generic Riemann solutions had remained open due to the incompatibility between the traditional methodologies for the stability of rarefaction waves and those for shock waves. Precisely, the strategy for attaining stability of the rarefaction waves has been based on the direct $L^2$-estimate on the physical variables. On the other hand, one of the successful strategies for the stability of shock waves is the anti-derivative method, which considers the anti-derivatives of the perturbations of the physical variables. Although the anti-derivative method was very successful for a single shock wave, this mismatch obstructs the stability analysis of the superposition of rarefaction and shock waves. This problem was resolved in \cite{KVW23,KVW25} by using another approach called the {\it method of $a$-contraction with shifts} introduced by Kang and Vasseur \cite{KV17,KV21-JEMS} to obtain stability of shock waves. This methodology relies on the weighted relative entropy between physical variables, which is essentially an $L^2$-estimate on them. Therefore, it is natural to combine it with an $L^2$-estimate for the rarefaction wave, obtaining the stability of composite waves of rarefaction and shock. In the present paper, we extend the method of $a$-contraction with shifts to the viscous, heat-conducting model with capillarity and investigate the stability of generic Riemann solutions to the NSFK system. To the best of our knowledge, this is the first result on the time-asymptotic stability of a viscous-dispersive shock, and of a generic Riemann solution, for the NSFK system.

Furthermore, we remove the smallness assumption on the strength of the rarefaction wave, which was assumed in \cite{KVW23,KVW25} for the Navier--Stokes and Navier--Stokes--Fourier systems. To deal with the large rarefaction wave, a special form of approximate rarefaction wave has been considered, motivated by \cite{MN92}. Instead of assuming the smallness of rarefaction strength, this approximation assumes that the initial slope of the wave is small enough, so that the wave is initially rarefied (see Section \ref{sec:rarefaction}). As a consequence, in Lemma \ref{lem: interaction estimates}, we carefully estimate the interactions between the elementary waves without using the smallness of the rarefaction strength.

The capillarity and the viscous-dispersive shock require two further ingredients. First, the Korteweg term is of third order, and the specific volume has no dissipation of its own. As in \cite{BDD06,HKKL25}, we rewrite the system with an auxiliary variable for the derivative of the specific volume. Then the Korteweg terms cancel at the highest order in the relative entropy estimate, and mixed energy estimates give the dissipation of this variable (Section~\ref{sec:H1}). Second, the $a$-contraction argument near the shock uses a precise relation between the shock profile and its derivative, up to higher order in the shock strength. For the NSFK system, the profile equation is of third order, and the companion paper \cite{H26} proves this relation uniformly in the left end state of the shock, which varies with the strength of the contact wave.

The rest of the paper is organized as follows. In Section \ref{sec:preliminaries}, we present basic estimates on the elementary waves, and introduce an auxiliary system which helps our analysis later. Then, we present a key a priori estimate in Section \ref{sec:proof}, and using it, we prove Theorem \ref{thm:main}.
Sections \ref{sec:L2} and \ref{sec:H1} are devoted to the proof of this a priori estimate. In Appendix \ref{app:proof_heat}, we provide the proof of a technical estimate used in Section \ref{sec:L2}.

\section{Preliminaries}\label{sec:preliminaries}
\setcounter{equation}{0}

In this section, we present several technical lemmas and properties of the elementary waves. Moreover, we introduce the extended system for the NSFK system \eqref{eq:NSK-theta}, which is more convenient for the later analysis.

\subsection{Approximate rarefaction wave}\label{sec:rarefaction}
{We collect the properties of the approximate rarefaction wave \eqref{smooth rarefaction}.}

\begin{lemma}\label{lem:Burgers}
	Suppose $w_-<w_*$ and set $\tilde{w}:=w_*-w_-$. Then the unique global smooth solution $w(t,x)$ to \eqref{eq: inviscid Burgers equation} satisfies the following properties:
	\begin{enumerate}
		\item $w_-< w(t,x)<w_*$ and $w_x(t,x)>0$ for all $t>0$ and $x\in\R$.
		\item For any $1\le p\le \infty$, there exists a constant $C_{p,q}>0$ such that, for all $t >0$,
		\begin{align*}
			&\|w_x(t,\cdot)\|_{L^p}\le C_{p,q}\min\left\{\tilde{w}\e_R^{1-1/p},\frac{\tilde{w}^{1/p}}{t^{1-1/p}}\right\}, \\
			&\|\pa_x^j w(t,\cdot)\|_{L^p}\le C_{p,q}\min\left\{\tilde{w}\e_R^{j-1/p},\frac{\tilde{w}^{(j-1)/q}\e_R^{(j-1)-1/p+1/q}}{t^{1-(j-1)/q}}\right\},\quad j=2,3,4 .
		\end{align*}
		Moreover, there exists a constant $C_q>0$ such that
		\[
			|\pa_x^j w(t,x)|\le C_q\e_R^{\,j-1}w_x(t,x),\qquad j=2,3,4,\quad t\ge0,\ x\in\R .
		\]

\item (Left tail) There exists a constant $C_q>0$ such that, for all $t\ge0$ and $x\le w_- t$,
\begin{align*}
  |w(t,x)-w_-|
  &\le \frac{C_q\,\tilde w}{\bigl(1+(\e_R|x-w_-t|)^2\bigr)^{q/2}},\\[3pt]
  |w_x(t,x)|
  &\le \frac{C_q\,\tilde w\,\e_R}{\bigl(1+(\e_R|x-w_-t|)^2\bigr)^{q}}.
\end{align*}
  \item (Right tail) There exists a constant $C_q>0$ such that, for all $t\ge0$ and $x\ge w_* t$,
\begin{align*}
	|w(t,x)-w_*|
	&\le \frac{C_q\,\tilde w}{\bigl(1+(\e_R|x-w_*t|)^2\bigr)^{q/2}},\\[3pt]
	|w_x(t,x)|
	&\le \frac{C_q\,\tilde w\,\e_R}{\bigl(1+(\e_R|x-w_*t|)^2\bigr)^{q}}.
\end{align*}
\item Let $w^r(x/t)$ be the rarefaction wave solution to the Burgers equation connecting $w_-$ and $w_*$. Then, $\displaystyle\lim_{t\to\infty}\sup_{x \in \mathbb{R}}|w(t,x)-w^r(x/t)|=0.$
	\end{enumerate}
\end{lemma}
\begin{proof}

Although the proof follows the argument of \cite{MN92} and \cite[Lemma~2.1]{EHK26}, we provide it for the reader's convenience.
\noindent (1), (5) Since $w_0'>0$, the characteristics do not intersect, and the solution is
\[
    w(t,x)=w_0(x_0),\qquad x=x_0+w_0(x_0)t .
\]
The representation gives (1), and (5) is \cite[Lemma~2.1~(v)]{MN92}.

\noindent (2) Differentiating in $x$ with $\frac{dx}{dx_0}=1+w_0'(x_0)t$, and writing $w_0^{(k)}=w_0^{(k)}(x_0)$,
\begin{equation}\label{eq:w-derivatives}
\begin{aligned}
    w_x &= \frac{w_0'}{1+w_0't},\qquad
    w_{xx} = \frac{w_0''}{(1+w_0't)^3},\qquad
    w_{xxx} = \frac{w_0'''}{(1+w_0't)^4}-\frac{3t(w_0'')^2}{(1+w_0't)^5},\\
    w_{xxxx} &= \frac{w_0''''}{(1+w_0't)^5}-\frac{10t\,w_0''w_0'''}{(1+w_0't)^6}+\frac{15t^2(w_0'')^3}{(1+w_0't)^7}.
\end{aligned}
\end{equation}

From \eqref{eq: inviscid Burgers equation},
\[
    w_0'(x_0)=\frac{K_q\tilde w\,\e_R}{2\bigl(1+(\e_Rx_0)^2\bigr)^q}.
\]
Differentiating the formula for $w_0'$, we obtain, for $j=1,2,3$,
\begin{equation}\label{eq:w0-ratio}
\begin{aligned}
    |w_0^{(j+1)}(x_0)|&\le C_q\,\e_R^{\,j}\,w_0'(x_0)\bigl(1+(\e_Rx_0)^2\bigr)^{-1/2},\\
    |w_0^{(j+1)}(x_0)|&\le C_q\,|w_0'(x_0)|^{1-\frac jq}\,\e_R^{\,j}(\e_R\tilde w)^{\frac jq}\bigl(1+(\e_Rx_0)^2\bigr)^{-1}.
\end{aligned}
\end{equation}
By \eqref{eq:w-derivatives} and \eqref{eq:w0-ratio}$_1$, with $\frac{w_0't}{1+w_0't}\le1$ for the terms carrying powers of $t$,
\begin{equation}\label{eq:wj-pointwise}
    |\pa_x^j w(t,x)|\le C_q\,\e_R^{\,j-1}\,w_x(t,x),\qquad j=2,3,4 .
\end{equation}

For $j=1$, by $dx=(1+w_0't)\,dx_0$, $w_0'>0$, and $\int_\R w_0'\,dx_0=\tilde w$,
\[
    \|w_x\|_{L^p}^p=\int_\R\frac{(w_0')^p}{(1+w_0't)^{p-1}}\,dx_0
    \le\min\Bigl\{\int_\R(w_0')^p\,dx_0,\ t^{1-p}\int_\R w_0'\,dx_0\Bigr\}
    \le C\min\bigl\{\tilde w^p\e_R^{\,p-1},\ \tilde w\,t^{1-p}\bigr\}.
\]
For $j=2,3,4$, the bound in $\e_R$ follows from \eqref{eq:wj-pointwise} and the case $j=1$, and the bound in $t$ comes from \eqref{eq:w-derivatives} and \eqref{eq:w0-ratio}$_2$. For $j=2$, using $\frac{w_0'}{1+w_0't}\le\frac1t$ and $(1+w_0't)^{-(2p-1+\frac pq)}\le1$,
\[
\begin{aligned}
    \|w_{xx}\|_{L^p}^p&=\int_\R\frac{|w_0''|^p}{(1+w_0't)^{3p-1}}\,dx_0
    \le C\e_R^{\,p}(\e_R\tilde w)^{\frac pq}\int_\R\Bigl(\frac{w_0'}{1+w_0't}\Bigr)^{p-\frac pq}\frac{dx_0}{(1+(\e_Rx_0)^2)^{p}}\\
    &\le C\,t^{-p+\frac pq}\,\e_R^{\,p-1}(\e_R\tilde w)^{\frac pq} .
\end{aligned}
\]
For $j=3,4$, the terms $t(w_0'')^2$, $t\,w_0''w_0'''$, and $t^2(w_0'')^3$ in \eqref{eq:w-derivatives} are treated in the same way, with $\bigl(\frac{w_0't}{1+w_0't}\bigr)^p\le1$ absorbing the powers of $t$, and \eqref{eq:w0-ratio} with $j=1,2,3$ and $\e_R\le1$ give
\[
    \|\pa_x^jw\|_{L^p}\le C_{p,q}\,\tilde w^{\frac{j-1}q}\,\e_R^{\,(j-1)-\frac1p+\frac1q}\,t^{-1+\frac{j-1}q},\qquad j=3,4 .
\]
For $p=\infty$, the bound in $\e_R$ is \eqref{eq:wj-pointwise} with $\|w_x\|_{L^\infty}\le C\min\{\tilde w\e_R,t^{-1}\}$, and the bound in $t$ follows from \eqref{eq:w-derivatives}, \eqref{eq:w0-ratio}$_2$, and $\frac{w_0'}{1+w_0't}\le\frac1t$.

\noindent (3), (4) For $x\le w_-t$, we have $x_0=x-w_0(x_0)t\le x-w_-t\le0$. Hence, by $\int_z^\infty(1+y^2)^{-q}\,dy\le C_q(1+z^2)^{-q/2}$ for $z\ge0$,
\[
\begin{aligned}
    |w(t,x)-w_-|&=\frac{K_q\tilde w}{2}\int_{-\infty}^{\e_Rx_0}\frac{dy}{(1+y^2)^q}
    \le \frac{C_q\tilde w}{\bigl(1+(\e_R|x-w_-t|)^2\bigr)^{q/2}},\\
    |w_x(t,x)|&\le w_0'(x_0)\le\frac{C_q\tilde w\,\e_R}{\bigl(1+(\e_R|x-w_-t|)^2\bigr)^{q}} .
\end{aligned}
\]
The bound (4) follows in the same way from $x_0\ge x-w_*t\ge0$ for $x\ge w_*t$.
\end{proof}

{The approximate rarefaction wave $(v^R,u^R,\theta^R)$ defined in \eqref{smooth rarefaction} solves the Euler equations} 
    \begin{equation}\label{eq:smooth rarefaction}
        \begin{aligned}
        &v^R_t-u^R_x=0, \\
        &u^R_t+p(v^R,\theta^R)_x=0,\\
        &\frac{R}{\gamma-1}\theta^R_t +p(v^R,\theta^R)u^R_x=0. 
        \end{aligned}
    \end{equation}
    
    Using Lemma \ref{lem:Burgers}, one can obtain the following properties on $(v^R,u^R,\theta^R)$.
\begin{lemma}\label{lem:rarefaction_property} Let $\lambda_{1-}:=\lambda_1(v_-,\theta_-)=w_-$ and $\lambda_{1*}:=\lambda_1(v_*,\theta_*)=w_*$. Then, the smooth approximate 1-rarefaction wave $(v^R, u^R,\theta^R)(t, x)$ defined in \eqref{smooth rarefaction} satisfies the following properties:
\begin{enumerate}
\item For all $t\ge0$ and $x\in\mathbb{R}$,
\[
u^R_x(t,x)=\frac{2v^R}{\gamma+1}\,w_x(1+t,x)>0,\quad
v^R_x=\frac{v^Ru^R_x}{\sqrt{R\gamma\theta^R}}>0,\quad
\theta^R_x=-\frac{(\gamma-1)\theta^Rv^R_x}{v^R}<0.
\]

\item For any $1\le p\le \infty$, there exists a constant $C_{p,q}>0$ such that, for all $t\ge0$, 
\begin{align*}
&\|(v^R_x,u^R_x,\theta^R_x)(t,\cdot)\|_{L^p}
\le C_{p,q}\,\min\!\left\{\delta_R\,\e_R^{\,1-1/p},\ \frac{\delta_R^{1/p}}{(1+t)^{\,1-1/p}}\right\},\\
&\|\partial_x^j(v^R,u^R,\theta^R)(t,\cdot)\|_{L^p}
\le C_{p,q}\,\min\!\left\{\delta_R\,\e_R^{\,j-1/p},\
\left( \delta_R^{1/p} + \delta_R^{(j-1)/q} \right) \frac{1}{(1+t)^{\,1-(j-1)/q}}\right\},
\quad j=2,3,4.
\end{align*}
Moreover, there exists a constant $C_q>0$ such that
\[
|\partial_x^j(v^R,u^R,\theta^R)(t,x)|\le C_q\e_R^{\,j-1}|(v^R_x,u^R_x,\theta^R_x)(t,x)|,\qquad j=2,3,4,\quad t\ge0,\ x\in\R .
\]

\item (Left tail) There exists a constant $C_q>0$ such that, for all $t\ge0$ and $x\le \lambda_{1-}(1+t)$,
\begin{align*}
|(v^R,u^R,\theta^R)(t,x)-(v_-,u_-,\theta_-)|
&\le \frac{C_q\,\delta_R}{\bigl(1+(\e_R|x-\lambda_{1-}(1+t)|)^2\bigr)^{q/2}},\\
|(v^R_x,u^R_x,\theta^R_x)(t,x)|
&\le \frac{C_q\,\delta_R\,\e_R}{\bigl(1+(\e_R|x-\lambda_{1-}(1+t)|)^2\bigr)^{q}}.
\end{align*}
\item (Right tail) There exists a constant $C_q>0$ such that, for all $t\ge0$ and $x\ge \lambda_{1*}(1+t)$,
\begin{align*}
	|(v^R,u^R,\theta^R)(t,x)-(v_*,u_*,\theta_*)|
	&\le \frac{C_q\,\delta_R}{\bigl(1+(\e_R|x-\lambda_{1*}(1+t)|)^2\bigr)^{q/2}},\\
	|(v^R_x,u^R_x,\theta^R_x)(t,x)|
	&\le \frac{C_q\,\delta_R\,\e_R}{\bigl(1+(\e_R|x-\lambda_{1*}(1+t)|)^2\bigr)^{q}}.
\end{align*}

\item Let $(v^r,u^r,\theta^r)(x/t)$ be the 1-rarefaction wave connecting $(v_-,u_-,\theta_-)$ and $(v_*,u_*,\theta_*)$. Then, $\displaystyle \lim_{t\to\infty}\sup_{x\in\mathbb{R}}|(v^R,u^R,\theta^R)(t,x)-(v^r,u^r,\theta^r)(x/t)|=0.$
\end{enumerate}
\end{lemma}

\begin{proof}

\noindent (1) We use $s(v,\theta)=\frac{R}{\gamma-1}\ln\theta+R\ln v+\mathrm{const}$ and $\lambda_1(v,\theta)=-\frac{\sqrt{\gamma R\theta}}{v}$. Differentiating $s(v^R,\theta^R)=s(v_-,\theta_-)$ and $z_1(v^R,u^R)=z_1(v_-,u_-)$ in $x$, we obtain 
\begin{equation}\label{eq:sz_diff}
    \theta^R_x=-\frac{(\gamma-1)\theta^R}{v^R}\,v^R_x ,\qquad
    u^R_x=-\lambda_1(v^R,\theta^R)\,v^R_x=\frac{\sqrt{R\gamma\theta^R}}{v^R}\,v^R_x .
\end{equation}
Along the constant entropy curve, $\lambda_1(v^R,\theta^R)=-c_0(v^R)^{-\frac{\gamma+1}2}$ for a constant $c_0>0$, and the $x$-derivative of $\lambda_1(v^R,\theta^R)=w(1+t,x)$ is 
\begin{equation}\label{eq:lambda_diff}
    w_x=-\frac{\gamma+1}{2v^R}\,\lambda_1(v^R,\theta^R)\,v^R_x=\frac{\gamma+1}{2v^R}\,u^R_x .
\end{equation}
The signs follow from \eqref{eq:sz_diff}, \eqref{eq:lambda_diff}, and $w_x>0$ in Lemma~\ref{lem:Burgers}.

\noindent (2) By \eqref{smooth rarefaction}, $U^R:=(v^R,u^R,\theta^R)=\Lambda(w(1+t,x))$ for a smooth function $\Lambda$ on $[w_-,w_*]$, and \eqref{eq:lambda_diff} gives $\Lambda'\ne0$ there. Also $\tilde w=w_*-w_-\sim\delta_R$. The chain rule gives
\[
\begin{aligned}
    \pa_x U^R &= \Lambda'(w)w_x, \\
    \pa_x^2 U^R &= \Lambda''(w)(w_x)^2 + \Lambda'(w)w_{xx}, \\
    \pa_x^3 U^R &= \Lambda'''(w)(w_x)^3 + 3\Lambda''(w)w_x w_{xx} + \Lambda'(w)w_{xxx}, \\
    \pa_x^4 U^R &= \Lambda''''(w)(w_x)^4 + 6\Lambda'''(w)(w_x)^2 w_{xx} + 3\Lambda''(w)(w_{xx})^2 + 4\Lambda''(w)w_x w_{xxx} + \Lambda'(w)w_{xxxx},
\end{aligned}
\]
where $w$ and its derivatives are evaluated at $(1+t,x)$. By Lemma~\ref{lem:Burgers}~(2), the case $j=1$ holds and
\[
    |w_xw_{xx}|\le C\e_R(w_x)^2,\qquad |(w_{xx})^2|+|w_xw_{xxx}|\le C\e_R^2(w_x)^2,\qquad |(w_x)^2w_{xx}|\le C\e_R(w_x)^3,
\]
so every term of $\pa_x^jU^R$ other than $\Lambda'(w)\pa_x^jw$ is bounded by $C\e_R^{\,j-r}(w_x)^r$ for some $2\le r\le j$. Lemma~\ref{lem:Burgers}~(2), with $rp$ in place of $p$ for these terms, $\e_R\le1$, and $r-\frac1p\ge1>1-\frac{j-1}q$ give
\[
\begin{aligned}
    \|\pa_x^jU^R\|_{L^p}&\le C\|\pa_x^jw\|_{L^p}+C\sum_{r=2}^{j}\e_R^{\,j-r}\|w_x\|_{L^{rp}}^{r}\\
    &\le C\min\Bigl\{\delta_R\,\e_R^{\,j-\frac1p},\ \bigl(\delta_R^{\frac1p}+\delta_R^{\frac{j-1}q}\bigr)(1+t)^{-1+\frac{j-1}q}\Bigr\},\qquad j=2,3,4 .
\end{aligned}
\]
Since $\|w_x\|_{L^\infty}\le C\delta_R\e_R$ and $|U^R_x|=|\Lambda'(w)|\,w_x\ge c\,w_x$, the same term-by-term bound gives
\[
    |\pa_x^jU^R|\le C\e_R^{\,j-1}w_x\le C\e_R^{\,j-1}|U^R_x|,\qquad j=2,3,4 .
\]

\noindent (3), (4) By the mean value theorem,
\[
    \begin{gathered}|U^R(t,x)-(v_-,u_-,\theta_-)|\le\sup_{[w_-,w_*]}|\Lambda'|\,|w(1+t,x)-w_-|,\\ |U^R(t,x)-(v_*,u_*,\theta_*)|\le\sup_{[w_-,w_*]}|\Lambda'|\,|w(1+t,x)-w_*|,\\ |U^R_x(t,x)|\le C\,|w_x(1+t,x)|,\end{gathered} 
\]
so the bounds follow from Lemma~\ref{lem:Burgers}~(3),(4) with $t$ replaced by $1+t$.

\noindent (5) Lemma~\ref{lem:Burgers}~(5) at time $1+t$, the uniform continuity of $\Lambda$ on $[w_-,w_*]$, and $\sup_{x\in\R}|w^r(x/(1+t))-w^r(x/t)|\le C/t$, which holds since $w^r$ is Lipschitz and constant outside the fan, give the convergence.

\end{proof}

\subsection{Viscous contact wave}
It is proved in \cite{HXY08} that the viscous contact wave $(v^C,u^C,\theta^C)$ defined in \eqref{def:viscous-contact} satisfies the following decay estimates.
\begin{lemma}[\cite{HXY08}]\label{lem:viscous_contact}
	There exists a positive constant $C_1$ such that the viscous contact wave $(v^C,u^C,\theta^C)$ satisfies
	\begin{align*}
		&(v^C-v_*,u^C-u_*,\theta^C-\theta_*)(t,x)=O(1)\delta_Ce^{-\frac{C_1x^2}{1+t}},\quad x<0,\\
		&(v^C-v^*,u^C-u^*,\theta^C-\theta^*)(t,x)=O(1)\delta_Ce^{-\frac{C_1x^2}{1+t}},\quad x>0,\\
		&(\pa_x^n v^C,\pa_x^n \theta^C)(t,x)=\frac{O(1)\delta_C}{(1+t)^{n/2}}e^{-\frac{C_1x^2}{1+t}},\quad x\in \R,\quad n=1,2,\ldots,\\
		&\pa^n_x u^C(t,x) = \frac{O(1)\delta_C}{(1+t)^{(n+1)/2}}e^{-\frac{C_1x^2}{1+t}},\quad x\in \R,\quad n=1,2,\ldots.
	\end{align*}
\end{lemma}
To match the equation for the viscous contact wave with the NSFK system \eqref{eq:NSK-theta}, we note that $(v^C,u^C,\theta^C)(t,x)$ satisfies the following equations
\begin{align}
	\begin{aligned}\label{eq:viscous_contact}
		& v_t^C -u_x^C=0,\\
		& u_t^C +p^C_x=\mu\left(\frac{u^C_x}{v^C}\right)_x+\kappa\left(-\frac{v^C_{xx}}{(v^C)^5}+\frac{5(v^C_x)^2}{2(v^C)^6}\right)_x+Q_1^C,\\
		& \frac{R}{\gamma-1}\theta_t^C +p^C u_x^C = \alpha \left(\frac{\theta_x^C}{v^C}\right)_x +\mu \frac{(u^C_x)^2}{v^C}+Q_2^C,
	\end{aligned}
\end{align}
where $p^C :=p(v^C,\theta^C)$ and
\begin{equation}\label{def_QC}		Q_1^C:=u_t^C-\mu\left(\frac{u_x^C}{v^C}\right)_x-\kappa\left(-\frac{v^C_{xx}}{(v^C)^5}+\frac{5(v^C_x)^2}{2(v^C)^6}\right)_x,\quad Q_2^C:=-\mu \frac{(u_x^C)^2}{v^C}.
\end{equation}
Here we used the fact that $p^C=p(v^C,\theta^C)\equiv p_*$, so that $p^C_x=0$.
Then, by Lemma~\ref{lem:viscous_contact} and \eqref{eq:nonlinear_diffusion}--\eqref{def:viscous-contact}, we have
\begin{equation}\label{error_contact}
	Q_1^C=\frac{O(1)\delta_C}{(1+t)^{3/2}}e^{-\frac{C_1x^2}{1+t}},\quad (Q_1^C)_x=\frac{O(1)\delta_C}{(1+t)^{2}}e^{-\frac{C_1x^2}{1+t}}, \quad Q_2^C=\frac{O(1)\delta_C}{(1+t)^2}e^{-\frac{2C_1x^2}{1+t}},\quad t>0,\ x\in\R.
\end{equation}

\subsection{Viscous-dispersive shock wave}
 A solution of \eqref{eq:viscous_shock} is called a monotone profile if each of $v^S$, $u^S$, and $\theta^S$ is monotone on $\R$. The companion paper \cite{H26} constructs monotone profiles by Fenichel's geometric singular perturbation theory. 

\begin{lemma}[{\cite[Theorem 1.1 and Corollary 1.3]{H26}}] \label{lem:shock-property}

	For a given left end state $(v^*,u^*,\theta^*)\in\R_+\times\R\times\R_+$, there exist positive constants $\delta_0$, $c$, and $C$ such that the following holds for every right end state $(v_+,u_+,\theta_+)\in S_3(v^*,u^*,\theta^*)$ with shock strength $\delta_S:=|v_+-v^*|<\delta_0$. There is a traveling wave solution $(v^S,u^S,\theta^S):\R\to\R_+\times\R\times\R_+$ to \eqref{eq:viscous_shock} with $\lim_{\xi\to\pm\infty}(v^S,u^S,\theta^S)=(v^*,u^*,\theta^*)$ and $(v_+,u_+,\theta_+)$, which is a monotone profile and is unique up to translation within the monotone class. Let $(v^S,u^S,\theta^S)$ be normalized by $v^S(0)=\frac{v^*+v_+}{2}$. Then
	\begin{equation}\label{shock_est}
	\begin{aligned}
	&v^S_\xi>0,\quad u^S_\xi<0,\quad \theta^S_\xi<0,\qquad \xi\in\R,\\
	&|(v^S(\xi)-v^*,u^S(\xi)-u^*,\theta^S(\xi)-\theta^*)|\le C\delta_S e^{-c\delta_S|\xi|},\qquad \xi\le0,\\
	&|(v^S(\xi)-v_+,u^S(\xi)-u_+,\theta^S(\xi)-\theta_+)|\le C\delta_S e^{-c\delta_S|\xi|},\qquad \xi\ge 0,\\
	&|(v^S_\xi,u^S_\xi,\theta^S_\xi)|\le C\delta_S^2 e^{-c\delta_S |\xi|},\qquad \xi\in\R,\\
	&|\pa_\xi^k(v^S,u^S,\theta^S)|\le C\delta_S^{k-1}|v^S_\xi|,\qquad k=2,3,4,\quad \xi\in\R .
	\end{aligned}
	\end{equation}
	Set $p^*:=p(v^*,\theta^*)$, $p_+:=p(v_+,\theta_+)$, and $\sigma^*:=\sqrt{\frac{\gamma p^*}{v^*}}$. Then the shock speed satisfies
	\begin{equation}\label{sigma-est}
		|\sigma-\sigma^*|\le C\delta_S ,
	\end{equation}
	the derivatives satisfy \[ C^{-1}|v^S_\xi|\le|u^S_\xi|\le C|v^S_\xi|,\qquad C^{-1}|v^S_\xi|\le|\theta^S_\xi|\le C|v^S_\xi|,\qquad \xi\in\R, \] and, more precisely,
	\begin{equation} \label{u_xi,theta_xi-est}
	|u^S_\xi+\sigma^*v^S_\xi|\le C\delta_S|v^S_\xi|,\qquad \left|\theta^S_\xi+\frac{(\gamma-1)p^*}{R}v^S_\xi\right|\le C\delta_S|v^S_\xi|,\qquad \xi\in\R .
	\end{equation}
	Moreover,
	\begin{equation}\label{logistic-est}
	C^{-1} (v_+-v^S)(v^S-v^*) \le v^S_\xi \le C (v_+-v^S)(v^S-v^*),\qquad \xi\in\R .
	\end{equation}
	Set $A^*:=\frac{\gamma(\gamma+1)p^*}{2(v^*)^2\sigma^*}$. Then
	\begin{equation}\label{secant-est}
	\left|\frac{p^S-p_+}{v^S-v_+}-\frac{p^S-p^*}{v^S-v^*}-\sigma^*A^*\frac{\mu R\gamma}{\mu R\gamma+\alpha(\gamma-1)^2}\,(v_+-v^*)\right|\le C\delta_S^2,\qquad \xi\in\R .
	\end{equation}
	Finally, if $\Omega\subset\R_+\times\R_+$ is compact, then $\delta_0$, $c$, and $C$ can be chosen uniformly for all left end states with $(v^*,\theta^*)\in\Omega$.
\end{lemma}

We apply the last statement of Lemma~\ref{lem:shock-property} with a compact set $\Omega\subset\R_+\times\R_+$ whose interior contains $(v_*,\theta_*)$, fixed once $(v_*,u_*,\theta_*)$ is given. Since $|v_*-v^*|+|\theta_*-\theta^*|\le C\delta_0$, the left end state of the shock satisfies $(v^*,\theta^*)\in\Omega$ for $\delta_0$ small, so the constants of Lemma~\ref{lem:shock-property} do not depend on $(v^*,u^*,\theta^*)$.

\subsection{The superposition wave}
Now, we introduce the superposition wave of the approximate rarefaction wave, the viscous contact discontinuity, and the viscous-dispersive shock wave. For a given absolutely continuous shift function $X(t)$, we define a superposition wave $(\bar{v},\bar{u},\bar{\theta})$ as
\begin{align}
\begin{aligned}\label{def:superposition}
	(\bar{v},\bar{u},\bar{\theta})(t,x)=\Big(&v^R(t,x)+v^C(t,x)+v^S(x-\sigma t-X(t))-v_*-v^*,\\
	&\,u^R(t,x)+u^C(t,x)+u^S(x-\sigma t-X(t))-u_*-u^*,\\
	&\theta^R(t,x)+\theta^C(t,x)+\theta^S(x-\sigma t-X(t))-\theta_*-\theta^*\Big).
\end{aligned}
\end{align}
Then, using the equations \eqref{eq:smooth rarefaction}, \eqref{eq:viscous_contact}, and \eqref{eq:viscous_shock} for the elementary waves, we derive the following governing equations for the superposition wave:
\begin{align}
	\begin{aligned}\label{eq:superposition}
	&\bar{v}_t +\dot{X}(t)v^S_x- \bar{u}_x  = 0,\\
	&\bar{u}_t  +\dot{X}(t)u^S_x+ \bar{p}_x = \mu\left(\frac{\bar{u}_x}{\bar{v}}\right)_x+\kappa\left(-\frac{\bar{v}_{xx}}{\bar{v}^5}+\frac{5(\bar{v}_x^2)}{2\bar{v}^6}\right)_x+Q_1,\\
	&\frac{R}{\gamma-1}\bar{\theta}_t  + \frac{R}{\gamma-1}\dot{X}(t)\theta^S_x+\bar{p}\bar{u}_x=\alpha\left(\frac{\bar{\theta}_x}{\bar{v}}\right)_x +\mu \frac{\bar{u}_x^2}{\bar{v}}+Q_2,
	\end{aligned}
\end{align}
where $\bar{p}:=\frac{R\bar{\theta}}{\bar{v}}$. Here the error terms $Q_1$ and $Q_2$ are decomposed as
\[Q_i = Q_i^I+Q_i^R+Q_i^C,\quad i=1,2,\]
where $Q^I_i$ are the terms from the interaction between elementary waves:
\begin{equation*}
    \begin{aligned}
        Q_1^I&=\left(\bar{p}-p^R-p^C-p^S\right)_x-\mu \left(\frac{\bar{u}_x}{\bar{v}}-\frac{u^R_x}{v^R}-\frac{u^C_x}{v^C}-\frac{u^S_x}{v^S}\right)_x\\
	&\quad -\kappa\left(-\frac{\bar{v}_{xx}}{\bar{v}^5}+\frac{v^R_{xx}}{(v^R)^5}+\frac{v^C_{xx}}{(v^C)^5}+\frac{v^S_{xx}}{(v^S)^5}+\frac{5(\bar{v}_x)^2}{2\bar{v}^6}-\frac{5(v^R_x)^2}{2(v^R)^6}-\frac{5(v^C_x)^2}{2(v^C)^6}-\frac{5(v^S_x)^2}{2(v^S)^6}\right)_x,\\
	Q_2^I&=(\bar{p}\bar{u}_x-p^Ru^R_x-p^Cu^C_x-p^Su^S_x)-\alpha\left(\frac{\bar{\theta}_x}{\bar{v}}-\frac{\theta^R_x}{v^R}-\frac{\theta^C_x}{v^C}-\frac{\theta^S_x}{v^S}\right)_x\\
	&\quad -\mu\left(\frac{\bar{u}^2_x}{\bar{v}}-\frac{(u^R_x)^2}{v^R}-\frac{(u^C_x)^2}{v^C}-\frac{(u^S_x)^2}{v^S}\right),
    \end{aligned}
\end{equation*}
and $Q^R_i$ are the terms from the approximation of the rarefaction wave:
\[Q_1^R :=-\mu\left(\frac{u^R_x}{v^R}\right)_x+\kappa\left(\frac{v^R_{xx}}{(v^R)^5}-\frac{5(v^R_x)^2}{2(v^R)^6}\right)_x,\quad Q^R_2:=-\alpha\left(\frac{\theta^R_x}{v^R}\right)_x-\mu\frac{(u^R_x)^2}{v^R}.\]
Finally, the errors $Q^C_i$ from the viscous contact wave are defined in \eqref{def_QC}.

\subsection{Extended system}
As in \cite{BDD06,HKKL25}, we introduce the auxiliary variable
\[w:=-\sqrt{\kappa}\frac{v_x}{v^{5/2}},\]
which is a natural variable considering the energy structure of the NSFK system. Note that $w$ is equivalent to $v_x$ when $v$ is bounded below and above. Then, \eqref{eq:NSK-theta}$_2$ can be written as
\[u_t +p_x = \mu\left(\frac{u_x}{v}\right)_x + \sqrt{\kappa}\left(\frac{w_x}{v^{5/2}}\right)_x,\]
and $w$ satisfies
\[w_t = -\sqrt{\kappa}\left(\frac{v_x}{v^{5/2}}\right)_t =-\sqrt{\kappa}\left(\frac{v_t}{v^{5/2}}\right)_x =  -\sqrt{\kappa}\left(\frac{u_x}{v^{5/2}}\right)_x.\]
Therefore, the system \eqref{eq:NSK-theta} can be written in the following extended form:
\begin{align}
	\begin{aligned}\label{eq:NSFK-extended}
		&v_t -u_x = 0,\\
		&u_t +p_x = \mu\left(\frac{u_x}{v}\right)_x+\sqrt{\kappa}\left(\frac{w_x}{v^{5/2}}\right)_x,\\
		&\frac{R}{\gamma-1}\theta_t +pu_x = \alpha\left(\frac{\theta_x}{v}\right)_x+\frac{\mu u_x^2}{v},\\
		&w_t = -\sqrt{\kappa}\left(\frac{u_x}{v^{5/2}}\right)_x.
	\end{aligned}
\end{align}
This extended system makes the later analysis simpler. Similarly, the system \eqref{eq:superposition} for the superposition wave can be extended by introducing $\bar{w}:=-\sqrt{\kappa}\frac{\bar{v}_x}{\bar{v}^{5/2}}$. Then, it satisfies
\begin{align*}
	\bar{w}_t &= -\sqrt{\kappa}\left(\frac{\bar{v}_t}{\bar{v}^{5/2}}\right)_x= -\sqrt{\kappa}\left(\frac{\bar{u}_x}{\bar{v}^{5/2}}\right)_x-\dot{X}(t)w^S_x+Q_3,
\end{align*}
where
\[w^S:=-\sqrt{\kappa}\frac{v^S_x}{(v^S)^{5/2}},\quad Q_3 := -\sqrt{\kappa}\dot{X}\left(v^S_x\left(\frac{1}{(v^S)^{5/2}}-\frac{1}{\bar{v}^{5/2}}\right)\right)_x.\]
Therefore, the extended superposition wave $(\bar{v},\bar{u},\bar{\theta},\bar{w})$ satisfies
\begin{align}
	\begin{aligned}\label{eq:composite_extended}
		&\bar{v}_t +\dot{X}(t)v^S_x - \bar{u}_x=0,\\
		&\bar{u}_t +\dot{X}(t)u^S_x + \bar{p}_x = \mu \left(\frac{\bar{u}_x}{\bar{v}}\right)_x + \sqrt{\kappa}\left(\frac{\bar{w}_x}{\bar{v}^{5/2}}\right)_x+ Q_1,\\
		&\frac{R}{\gamma-1}\bar{\theta}_t + \frac{R}{\gamma-1}\dot{X}(t)\theta^S_x+ \bar{p}\bar{u}_x = \alpha\left(\frac{\bar{\theta}_x}{\bar{v}}\right)_x +\frac{\mu\bar{u}_x^2}{\bar{v}}+Q_2,\\
		&\bar{w}_t +\dot{X}(t)w^S_x = -\sqrt{\kappa}\left(\frac{\bar{u}_x}{\bar{v}^{5/2}}\right)_x+Q_3.
	\end{aligned}
\end{align}

From \eqref{eq:composite_extended} and \eqref{eq:NSFK-extended}, the perturbation
\begin{equation}\label{def:perturbation}
    (\phi,\psi,\vartheta,\omega):=(v-\bar v,\,u-\bar u,\,\theta-\bar\theta,\,w-\bar w),
    \qquad (\phi_0,\psi_0,\vartheta_0,\omega_0):=(\phi,\psi,\vartheta,\omega)(0,\cdot),
\end{equation}
which we use in Sections~\ref{sec:proof} and~\ref{sec:H1} and in Appendix~\ref{app:proof_heat}, satisfies
\begin{align}
	\begin{aligned}\label{eq:perturbation}
		&\phi_t -\dot{X}(t)v^S_x - \psi_x=0,\\
		&\psi_t -\dot{X}(t)u^S_x + (p-\bar{p})_x = \mu \left(\frac{u_x}{v}-\frac{\bar{u}_x}{\bar{v}}\right)_x + \sqrt{\kappa}\left(\frac{w_x}{v^{5/2}}-\frac{\bar{w}_x}{\bar{v}^{5/2}}\right)_x- Q_1,\\
		&\frac{R}{\gamma-1}\vartheta_t - \frac{R}{\gamma-1}\dot{X}(t)\theta^S_x+ (pu_x-\bar{p}\bar{u}_x) = \alpha\left(\frac{\theta_x}{v}-\frac{\bar{\theta}_x}{\bar{v}}\right)_x +\mu\left(\frac{u_x^2}{v}-\frac{\bar{u}_x^2}{\bar{v}}\right)-Q_2,\\
		&\omega_t -\dot{X}(t)w^S_x = -\sqrt{\kappa}\left(\frac{u_x}{v^{5/2}}-\frac{\bar{u}_x}{\bar{v}^{5/2}}\right)_x-Q_3.
	\end{aligned}
\end{align}

The following lemma collects the pointwise bounds on the derivatives of the superposition wave, which are used in Sections~\ref{sec:proof}, \ref{sec:L2}, and~\ref{sec:H1} and in Appendix~\ref{app:proof_heat}.
\begin{lemma}\label{lem:superposition-derivatives}
Under the assumptions of Theorem~\ref{thm:main}, let $X$ be an absolutely continuous shift. Then, there exists a constant $C>0$, independent of $X$, $\delta_0$, and $\e_R$, such that for all $t\ge0$ and $x\in\R$, the extended superposition wave $(\bar v,\bar u,\bar\theta,\bar w)$ in \eqref{eq:composite_extended} satisfies
\begin{equation}\label{eq:vbar-derivatives}
\begin{aligned}
    &\sum_{j=1}^{3}|\partial_x^j\bar v(t,x)|+\sum_{j=1}^{2}|\partial_x^j(\bar u,\bar\theta)(t,x)|+\sum_{j=0}^{2}|\partial_x^j\bar w(t,x)|\\
    &\qquad\le C\bigl(|v^R_x(t,x)|+|v^S_x(x-\sigma t-X(t))|\bigr)+C\delta_C(1+t)^{-1/2}e^{-\frac{C_1|x|^2}{1+t}},
\end{aligned}
\end{equation}
where $C_1$ is the constant in Lemma~\ref{lem:viscous_contact}.
\end{lemma}
\begin{proof}[Proof of Lemma~\ref{lem:superposition-derivatives}]
By \eqref{def:superposition}, for $j=1,2,3$,
\[
    \partial_x^j(\bar v,\bar u,\bar\theta)(t,x)=\partial_x^j(v^R,u^R,\theta^R)(t,x)+\partial_x^j(v^C,u^C,\theta^C)(t,x)+\partial_x^j(v^S,u^S,\theta^S)(x-\sigma t-X(t)).
\]
It follows from Lemma~\ref{lem:rarefaction_property}~(1) and~(2), Lemma~\ref{lem:shock-property}, and Lemma~\ref{lem:viscous_contact}, respectively, that, for $j=1,2,3$,
\[
\begin{aligned}
    &|\partial_x^j(v^R,u^R,\theta^R)|\le C|v^R_x|,\qquad |\partial_x^j(v^S,u^S,\theta^S)|\le C|v^S_x|,\\
    &|\partial_x^j(v^C,u^C,\theta^C)|\le C\delta_C(1+t)^{-\frac12}e^{-\frac{C_1|x|^2}{1+t}}.
\end{aligned}
\]
This proves \eqref{eq:vbar-derivatives} for $\bar v$, $\bar u$, and $\bar\theta$. Since $\bar w=-\sqrt\kappa\,\bar v_x/\bar v^{5/2}$, $\bar v\ge c>0$, and $\|\bar v_x\|_{L^\infty}\le C$, we have
\[
    |\bar w|\le C|\bar v_x|,\qquad |\bar w_x|\le C\bigl(|\bar v_{xx}|+|\bar v_x|\bigr),\qquad |\bar w_{xx}|\le C\bigl(|\bar v_{xxx}|+|\bar v_{xx}|+|\bar v_x|\bigr),
\]
and the bound for $\bar w$ follows from the bound for $\bar v$. This completes the proof of Lemma~\ref{lem:superposition-derivatives}.
\end{proof}

\section{A priori estimate and the proof of Theorem \ref{thm:main}}\label{sec:proof}
\setcounter{equation}{0}

In this section, we prove Theorem \ref{thm:main}, where an a priori estimate on the $H^1$-norms of the perturbations is crucially used. We first explain how to construct the shift function $X(t)$, and then present the a priori estimate for $H^1$-perturbations. Using the a priori estimate, we complete the proof of Theorem \ref{thm:main}.

\subsection{Construction of the shift and weight function}
To obtain a stability result for the superposition wave, we need to shift the viscous-dispersive shock wave $(v^S,u^S,\theta^S)$ by an appropriate shift function $X(t)$. The shift is defined as a solution to the following ODE:
\begin{align}
	\begin{aligned}\label{def:shift}
	\dot{X}(t) = -\frac{M}{\delta_S}\int_{\R}\Bigg[a(t,x)&\Bigg(u^S_x(x-\sigma t-X(t))(u-\bar{u})+\frac{\bar{p}v^S_x(x-\sigma t-X(t))}{\bar{v}}(v-\bar{v})\\
	&\quad +\frac{R\theta^S_x(x-\sigma t-X(t))}{(\gamma-1)\bar{\theta}}(\theta-\bar{\theta})\Bigg)\Bigg]\,dx,\quad X(0)=0,
	\end{aligned}
\end{align}
where the constant $M$ and weight function $a$ are defined as
\begin{equation}\label{choice_M}
	M=\frac{3A^*}{2(\sigma^*)^2}\frac{\mu R \gamma}{\mu R \gamma+\alpha(\gamma-1)^2}\left(1+\frac{2\alpha(\gamma-1)^2}{\mu R\gamma}\right)
\end{equation}
{with $A^*$ as in Lemma~\ref{lem:shock-property},} and
\begin{equation}\label{def:weight}
	a(t,x):=1+\frac{v^S(x-\sigma t-X(t))-v^*}{\sqrt{\delta_S}}.
\end{equation}
Then, using $v_+-v^*=\delta_S$, it is straightforward to observe that
\begin{equation*}\label{property_weight}
	1<a<1+\sqrt{\delta_S},\quad a_x=\frac{v^S_x}{\sqrt{\delta_S}}>0,\quad |a_x|=\frac{|v^S_x|}{\sqrt{\delta_S}}\sim\frac{|u^S_x|}{\sqrt{\delta_S}}\sim \frac{|\theta^S_x|}{\sqrt{\delta_S}}.
\end{equation*}
The following lemma ensures that \eqref{def:shift} has a unique absolutely continuous solution on any time interval on which the solution $(v,u,\theta)$ stays bounded.
\begin{lemma}\label{lem:shift}
	For any $c_1,c_2>0$, there exists a constant $C>0$ such that the following holds. For any $T>0$ and any $(v,u,\theta)\in L^\infty((0,T)\times\R)$ with
	\[
	    c_1\le v(t,x),\ \theta(t,x)\le c_2,\quad |u(t,x)|\le c_2,\qquad (t,x)\in[0,T]\times\R,
	\]
	the ODE \eqref{def:shift} has a unique absolutely continuous solution $X$ on $[0,T]$. Moreover, $|X(t)|\le Ct$ for all $t\in[0,T]$.
\end{lemma}
The proof follows from the Cauchy--Lipschitz theorem in the form of \cite[Lemma A.1]{CKKV}, as in \cite[Lemma 3.1]{KVW25}, and we omit it.
The motivation of the definition of the shift is to control the terms $Y_1$, $Y_2$, and $Y_3$ in \eqref{def:Yi}. Using the definitions of $Y_i$, \eqref{def:shift} can be written as $\dot{X}=-\frac{M}{\delta_S}(Y_1+Y_2+Y_3)$, from which we can extract a good term $-\frac{\delta_S}{M}|\dot{X}|^2$. This enables us to control the leading terms when we estimate the weighted relative entropy. We refer to Section \ref{sec:decompose} for the detailed motivation of the shift. 
\subsection{Local existence}

We recall the local existence of strong solutions to \eqref{eq:NSK-theta}; see \cite{HL96,CS19}.
\begin{proposition}\label{prop:local}
	Let $\underline v,\underline u,\underline\theta$ be smooth monotone functions such that $(\underline v,\underline u,\underline\theta)(x)=(v_\pm,u_\pm,\theta_\pm)$ for $\pm x\ge1$. For any constants $M_0,M_1,\underline k_0,\overline k_0,\underline k_1,\overline k_1$ with $M_1>M_0>0$ and $\overline k_1>\overline k_0>\underline k_0>\underline k_1>0$, there exists a constant $T_0>0$ such that if
	\begin{align*}
	    &\|v_0-\underline v\|_{H^2(\R)}+\|(u_0-\underline u,\theta_0-\underline\theta)\|_{H^1(\R)}\le M_0,\\
	    &0<\underline k_0\le v_0(x),\ \theta_0(x)\le\overline k_0,\qquad \forall x\in\R,
	\end{align*}
	then \eqref{eq:NSK-theta} has a unique solution $(v,u,\theta)$ on $[0,T_0]$ such that
	\begin{align*}
	    &v-\underline v\in C([0,T_0];H^2(\R))\cap L^2(0,T_0;H^3(\R)),\\
	    &(u-\underline u,\theta-\underline\theta)\in C([0,T_0];H^1(\R))\cap L^2(0,T_0;H^2(\R)),
	\end{align*}
	and
	\[
	    \|v-\underline v\|_{L^\infty(0,T_0;H^2(\R))}+\|(u-\underline u,\theta-\underline\theta)\|_{L^\infty(0,T_0;H^1(\R))}\le M_1.
	\]
	Moreover,
	\[
	    \underline k_1\le v(t,x),\ \theta(t,x)\le\overline k_1,\qquad \forall (t,x)\in[0,T_0]\times\R .
	\]
\end{proposition}

\subsection{A priori estimate}

The key estimate for the proof of Theorem \ref{thm:main} is the following a priori estimate.

\begin{proposition}\label{prop:apriori}
	For given constant states $(v_-,u_-,\theta_-)\in \R_+\times\R\times\R_+$ and $(v_*,u_*,\theta_*)\in R_1(v_-,u_-,\theta_-)$, there exist positive constants $\delta_0$, $\varepsilon_1$, and $C_0$ such that the following holds. Suppose that $(v,u,\theta)$ is the solution to \eqref{eq:NSK-theta} on $[0,T]$ for some $T>0$ and $(\bar{v},\bar{u},\bar{\theta})$ is the superposition wave defined in \eqref{def:superposition} with the shift $X$ given by \eqref{def:shift}. Assume that $\delta_C,\delta_S<\delta_0$, $\e_R=\delta_0/2$,
	\[
	    v-\bar v\in C([0,T];H^2(\R))\cap L^2(0,T;H^3(\R)),\qquad (u-\bar u,\theta-\bar\theta)\in C([0,T];H^1(\R))\cap L^2(0,T;H^2(\R)),
	\]
	and
	\begin{equation}\label{smallness}
		\|v-\bar{v}\|_{C([0,T];H^2(\R))}+\|(u-\bar{u},\theta-\bar{\theta})\|_{C([0,T];H^1(\R))}\le \varepsilon_1.
	\end{equation}
	Then, for all $t\in[0,T]$, 
	\begin{align}
		\begin{aligned}\label{est:apriori}
			\sup_{s \in [0,t]}&\|(v-\bar{v},u-\bar{u},\theta-\bar{\theta},w-\bar{w})(s,\cdot)\|_{H^1}^2+\delta_S\int_0^t |\dot{X}(s)|^2\,ds\\
			&+\int_0^t(G^R+G^S)\,ds+\int_0^t(\|(u-\bar{u},\theta-\bar{\theta})_x\|_{H^1}^2+\|w-\bar{w}\|_{H^2}^2)\,ds\\
			&\le C_0\|(v-\bar v,u-\bar u,\theta-\bar\theta,w-\bar w)(0,\cdot)\|_{H^1}^2+C_0\delta_0^{1/6},
		\end{aligned}
	\end{align}
	where $G^R$ and $G^S$ are defined as
	\begin{equation}\label{GSGR}
		G^R:=\int_{\R} |v^R_x||(v-\bar{v},\theta-\bar{\theta})|^2\,dx,\quad G^S:=\int_{\R} |v^S_x||(v-\bar{v},u-\bar{u},\theta-\bar{\theta})|^2\,dx.
	\end{equation}
	In addition,
	\begin{equation}\label{est:shift-bound}
		|\dot X(t)|\le C_0\|(v-\bar v,u-\bar u,\theta-\bar\theta)(t,\cdot)\|_{L^\infty(\R)},\qquad t\in[0,T].
	\end{equation}
\end{proposition}

The proof of Proposition \ref{prop:apriori} is given in Sections \ref{sec:L2} and \ref{sec:H1}.

\subsection{Proof of Theorem \ref{thm:main}}

We take the constants $\delta_0$, $\varepsilon_1$, and $C_0$ of Proposition~\ref{prop:apriori}, and let $(\phi,\psi,\vartheta,\omega)$ be the perturbation \eqref{def:perturbation}. The constants $C$ in Sections~\ref{sec:L2} and~\ref{sec:H1} and Appendix~\ref{app:proof_heat} are independent of $\delta_0$, $\e_R$, $\varepsilon_1$, and $T$. Hence Proposition~\ref{prop:apriori} holds with the same $\varepsilon_1$ and $C_0$ for every smaller $\delta_0$, and we fix $\delta_0$ in \eqref{eq:choice-eps0}.
Since $\omega=-\sqrt\kappa\bigl(\phi_x/v^{5/2}+\bar v_x(v^{-5/2}-\bar v^{-5/2})\bigr)$, if $\|\phi\|_{H^2}\le\varepsilon_1$, then we have
\begin{equation}\label{eq:H2-equiv}
\begin{aligned}
    &C_*^{-1}\|\phi\|_{H^2}\le \|\phi\|_{H^1}+\|\omega\|_{H^1}\le C_*\|\phi\|_{H^2},\\
    &\|\phi_x\|_{H^2}\le C_*\bigl(\|\omega\|_{H^2}+\|(\bar v_x\phi,\bar v_{xx}\phi,\bar v_{xxx}\phi)\|_{L^2}\bigr)
\end{aligned}
\end{equation}
for some constant $C_*\ge1$ independent of $\delta_0$.

By \eqref{def:superposition}, Lemma~\ref{lem:viscous_contact} with $t=0$, and \eqref{shock_est} in Lemma~\ref{lem:shock-property}, we have
\begin{align}\label{eq:initial-bridge}
\begin{aligned}
    &\bigl\|(\bar v,\bar u,\bar\theta)(0,\cdot)-(v^R,u^R,\theta^R)(0,\cdot)\bigr\|_{L^2(\R_-)}\\
    &\quad+\bigl\|(\bar v,\bar u,\bar\theta)(0,\cdot)-(v^R,u^R,\theta^R)(0,\cdot)-(v_+-v_*,u_+-u_*,\theta_+-\theta_*)\bigr\|_{L^2(\R_+)}\\
    &\quad+\bigl\|(\bar v-v^R)_x(0,\cdot)\bigr\|_{H^1(\R)}+\bigl\|\bigl((\bar u-u^R)_x,(\bar\theta-\theta^R)_x\bigr)(0,\cdot)\bigr\|_{L^2(\R)}\le\hat C\delta_0^{1/2}
\end{aligned}
\end{align}
for some constant $\hat C>0$ independent of $\delta_0$. We define $\varepsilon_0:=\varepsilon_1/\bigl(4\sqrt3\,C_*^2(C_0+1)^{1/2}\bigr)$ and choose $\delta_0$ such that
\begin{equation}\label{eq:choice-eps0}
    \varepsilon_0+\hat C\delta_0^{1/2}\le\frac{\varepsilon_1}{2},\qquad 6C_*^4C_0\hat C^2\delta_0+3C_*^2C_0\delta_0^{1/6}\le\frac{\varepsilon_1^2}{8}.
\end{equation}

Let $(v_0,u_0,\theta_0)$ satisfy \eqref{initial_data}, which we restate as
\begin{equation}\label{eq:initial-restated}
\begin{aligned}
    &\bigl\|(v_0,u_0,\theta_0)-(v^R,u^R,\theta^R)(0,\cdot)\bigr\|_{L^2(\R_-)}\\
    &\quad+\bigl\|(v_0,u_0,\theta_0)-(v^R,u^R,\theta^R)(0,\cdot)-(v_+-v_*,u_+-u_*,\theta_+-\theta_*)\bigr\|_{L^2(\R_+)}\\
    &\quad+\bigl\|(v_0-v^R(0,\cdot))_x\bigr\|_{H^1(\R)}+\bigl\|\bigl((u_0-u^R(0,\cdot))_x,(\theta_0-\theta^R(0,\cdot))_x\bigr)\bigr\|_{L^2(\R)}<\varepsilon_0 .
\end{aligned}
\end{equation}
By the triangle inequality, \eqref{eq:initial-restated}, \eqref{eq:initial-bridge}, and \eqref{eq:choice-eps0}, we have
\begin{align}\label{eq:initial-small}
\begin{aligned}
    &\|\phi_0\|_{H^2}+\|(\psi_0,\vartheta_0)\|_{H^1}\\
    &\quad\le\bigl\|(v_0,u_0,\theta_0)-(v^R,u^R,\theta^R)(0,\cdot)\bigr\|_{L^2(\R_-)}+\bigl\|(\bar v,\bar u,\bar\theta)(0,\cdot)-(v^R,u^R,\theta^R)(0,\cdot)\bigr\|_{L^2(\R_-)}\\
    &\qquad+\bigl\|(v_0,u_0,\theta_0)-(v^R,u^R,\theta^R)(0,\cdot)-(v_+-v_*,u_+-u_*,\theta_+-\theta_*)\bigr\|_{L^2(\R_+)}\\
    &\qquad+\bigl\|(\bar v,\bar u,\bar\theta)(0,\cdot)-(v^R,u^R,\theta^R)(0,\cdot)-(v_+-v_*,u_+-u_*,\theta_+-\theta_*)\bigr\|_{L^2(\R_+)}\\
    &\qquad+\bigl\|(v_0-v^R(0,\cdot))_x\bigr\|_{H^1}+\bigl\|(\bar v-v^R)_x(0,\cdot)\bigr\|_{H^1}\\
    &\qquad+\bigl\|\bigl((u_0-u^R(0,\cdot))_x,(\theta_0-\theta^R(0,\cdot))_x\bigr)\bigr\|_{L^2}+\bigl\|\bigl((\bar u-u^R)_x,(\bar\theta-\theta^R)_x\bigr)(0,\cdot)\bigr\|_{L^2}\\
    &\quad<\varepsilon_0+\hat C\delta_0^{1/2}\le\frac{\varepsilon_1}{2}.
\end{aligned}
\end{align}
By the Sobolev embedding, \eqref{eq:initial-small} implies that $v_0$ and $\theta_0$ are bounded above and below by positive constants. Since $X(0)=0$, Lemma~\ref{lem:rarefaction_property}~(2)--(4), Lemma~\ref{lem:viscous_contact}, and \eqref{shock_est} in Lemma~\ref{lem:shock-property} give
\begin{equation}\label{eq:bar-under-H3}
    \|(\bar v-\underline v,\bar u-\underline u,\bar\theta-\underline\theta)(0,\cdot)\|_{H^2}<\infty .
\end{equation}
We set $M_0:=\|v_0-\underline v\|_{H^2}+\|(u_0-\underline u,\theta_0-\underline\theta)\|_{H^1}$. By the triangle inequality, \eqref{eq:initial-small}, and \eqref{eq:bar-under-H3},
\begin{align*}
    M_0&\le\|\phi_0\|_{H^2}+\|(\psi_0,\vartheta_0)\|_{H^1}+\|(\bar v-\underline v)(0,\cdot)\|_{H^2}+\|(\bar u-\underline u,\bar\theta-\underline\theta)(0,\cdot)\|_{H^1}\\
    &\le\frac{\varepsilon_1}{2}+\|(\bar v-\underline v)(0,\cdot)\|_{H^2}+\|(\bar u-\underline u,\bar\theta-\underline\theta)(0,\cdot)\|_{H^1}<\infty .
\end{align*}
Hence Proposition~\ref{prop:local} with $M_1:=M_0+1$ gives a unique solution $(v,u,\theta)$ on $[0,T_0]$ in the class of Proposition~\ref{prop:local}. By Lemma~\ref{lem:shift}, \eqref{def:shift} has a unique absolutely continuous solution $X$ on $[0,T_0]$, and $|\dot X|\le C$ on $[0,T_0]$ since the right-hand side of \eqref{def:shift} is bounded.

Next, on every interval $[0,T]$ on which $(v,u,\theta)$ and $X$ exist, we apply Lemma~\ref{lem:rarefaction_property}~(2), Lemma~\ref{lem:viscous_contact}, \eqref{shock_est} in Lemma~\ref{lem:shock-property}, and $|\dot X|\le C$ to obtain
\begin{equation}\label{eq:bar-lip}
    \|(\bar v,\bar u,\bar\theta)(t,\cdot)-(\bar v,\bar u,\bar\theta)(s,\cdot)\|_{H^2}\le C|t-s|,\qquad \|\partial_x^3\bar v(t,\cdot)\|_{L^2}\le C,\qquad s,t\in[0,T].
\end{equation}
By Proposition~\ref{prop:local}, \eqref{eq:bar-under-H3}, and \eqref{eq:bar-lip}, $(v,u,\theta)$ is in the class of Proposition~\ref{prop:apriori} on $[0,T_0]$, and $t\mapsto\|\phi(t,\cdot)\|_{H^2}+\|(\psi,\vartheta)(t,\cdot)\|_{H^1}$ is continuous on $[0,T_0]$. Therefore, by \eqref{eq:initial-small}, there exists $T_1\in(0,T_0]$ such that
\begin{equation}\label{eq:cont-hyp}
    \sup_{t\in[0,T_1]}\Bigl(\|\phi(t,\cdot)\|_{H^2}+\|(\psi,\vartheta)(t,\cdot)\|_{H^1}\Bigr)\le\varepsilon_1 .
\end{equation}

We now extend the solution by the continuation argument, and let
\begin{equation}\label{eq:TM}
\begin{aligned}
    T_M:=\sup\Bigl\{T>0:\ &(v,u,\theta)\text{ exists on }[0,T]\text{ in the class of Proposition~\ref{prop:apriori}, with $X$}\\
    &\text{given by Lemma~\ref{lem:shift}, and }\sup_{t\in[0,T]}\bigl(\|\phi(t,\cdot)\|_{H^2}+\|(\psi,\vartheta)(t,\cdot)\|_{H^1}\bigr)\le\varepsilon_1\Bigr\}.
\end{aligned}
\end{equation}
By \eqref{eq:cont-hyp}, $T_M\ge T_1>0$.
Suppose, for contradiction, that $T_M<\infty$. For $0\le t<T_M$, \eqref{eq:TM}, \eqref{eq:bar-under-H3}, and \eqref{eq:bar-lip} give
\[
\begin{aligned}
    &\|(v-\underline v)(t,\cdot)\|_{H^2}+\|(u-\underline u,\theta-\underline\theta)(t,\cdot)\|_{H^1}\\
    &\quad\le\varepsilon_1+\|(\bar v-\underline v)(0,\cdot)\|_{H^2}+\|(\bar u-\underline u,\bar\theta-\underline\theta)(0,\cdot)\|_{H^1}+CT_M=:M_*.
\end{aligned}
\]
By \eqref{eq:TM} and the Sobolev embedding, $v$ and $\theta$ are bounded above and below on $[0,T_M)\times\R$ by positive constants independent of $t$. Proposition~\ref{prop:local} with $M_0=M_*$, $M_1=M_*+1$, and these bounds gives an existence time independent of the initial time, so applying it at a time close to $T_M$ extends $(v,u,\theta)$ beyond $T_M$. Since the extended solution has the two properties shown before \eqref{eq:cont-hyp}, the maximality of $T_M$ implies
\begin{equation}\label{eq:saturation}
    \sup_{t\in[0,T_M]}\Bigl(\|\phi\|_{H^2}+\|(\psi,\vartheta)\|_{H^1}\Bigr)=\varepsilon_1 .
\end{equation}
On the other hand, for every $T<T_M$, the solution on $[0,T]$ satisfies the hypotheses of Proposition~\ref{prop:apriori} by \eqref{eq:TM}. Letting $T\uparrow T_M$ and using \eqref{eq:H2-equiv}, \eqref{eq:initial-small}, and \eqref{eq:choice-eps0}, we find that
\[
\begin{aligned}
    \sup_{t\in[0,T_M]}\Bigl(\|\phi\|_{H^2}+\|(\psi,\vartheta)\|_{H^1}\Bigr)^2
    &\le 3C_*^2\sup_{t\in[0,T_M]}\|(\phi,\psi,\vartheta,\omega)\|_{H^1}^2\\
    &\le 3C_*^2C_0\bigl(C_*^2(\varepsilon_0+\hat C\delta_0^{1/2})^2+\delta_0^{1/6}\bigr)\le\frac{\varepsilon_1^2}{4},
\end{aligned}
\]
which contradicts \eqref{eq:saturation}. Hence $T_M=\infty$, and Proposition~\ref{prop:apriori} applies on every $[0,T]$ by \eqref{eq:TM}. Moreover, \eqref{est:apriori} in Proposition~\ref{prop:apriori}, \eqref{eq:H2-equiv}, and \eqref{eq:initial-small} give
\begin{align}
	\begin{aligned}\label{est:apriori_global}
		&\sup_{t\ge 0}\|(\phi,\psi,\vartheta,\omega)\|_{H^1}^2+\delta_S\int_0^\infty |\dot{X}|^2\,dt+\int_0^\infty(G^R+G^S)\,dt\\
		&\quad+\int_0^\infty\bigl(\|(\psi,\vartheta)_x\|_{H^1}^2+\|\omega\|_{H^2}^2\bigr)\,dt
		\le C\varepsilon_0^2+C\delta_0^{1/6}.
	\end{aligned}
\end{align}

By Lemma~\ref{lem:superposition-derivatives}, \eqref{GSGR} in Proposition~\ref{prop:apriori}, \eqref{est:apriori_global}, and \eqref{heat-int} and \eqref{eq:phi_x_int} in the proof of Proposition~\ref{prop:apriori},
\begin{equation}\label{eq:phix-L2t}
\begin{aligned}
    I&:=\int_0^\infty\int_\R|(\bar v_x,\bar v_{xx},\bar v_{xxx},\bar u_x,\bar u_{xx},\bar\theta_x,\bar\theta_{xx},\bar w_x,\bar w_{xx})|^2|(\phi,\vartheta)|^2\,dx\,dt\\
    &\le C\int_0^\infty(G^R+G^S)\,dt+C\delta_C^2\int_0^\infty\frac{1}{1+t}\int_\R e^{-\frac{C_1|x|^2}{1+t}}|(\phi,\vartheta)|^2\,dx\,dt<\infty .
\end{aligned}
\end{equation}
By \eqref{eq:H2-equiv}, \eqref{est:apriori_global}, and \eqref{eq:phix-L2t},
\begin{equation}\label{eq:J-def}
    \int_0^\infty\|\phi_x\|_{H^2}^2\,dt\le C\int_0^\infty\|\omega\|_{H^2}^2\,dt+CI<\infty .
\end{equation}
Therefore \eqref{eq:regularity} follows from \eqref{eq:TM}, \eqref{est:apriori_global}, and \eqref{eq:J-def}.

Finally, let $g(t):=\|\phi_x(t)\|_{H^1}^2+\|\psi_x(t)\|_{L^2}^2+\|\vartheta_x(t)\|_{L^2}^2$. We claim that
\begin{equation}\label{eq:g-W11}
    \int_0^\infty\bigl(|g(t)|+|g'(t)|\bigr)\,dt<\infty .
\end{equation}
By \eqref{est:apriori_global} and \eqref{eq:J-def}, $g\in L^1(0,\infty)$. To estimate $g'$, we use the perturbation equations \eqref{eq:perturbation} and \eqref{eq:vbar-derivatives} in Lemma~\ref{lem:superposition-derivatives} as in \cite{KVW23}, which give
\[
\begin{aligned}
    |g'(t)|&\le C\Bigl(\|(\phi_x,\phi_{xx},\phi_{xxx},\psi_x,\psi_{xx},\vartheta_x,\vartheta_{xx},\omega_x,\omega_{xx})(t)\|_{L^2}^2+|\dot X(t)|^2+\|(Q_1,Q_2)(t)\|_{L^2}^2\Bigr)\\
    &\quad+C\int_\R|(\bar v_x,\bar v_{xx},\bar v_{xxx},\bar u_x,\bar u_{xx},\bar\theta_x,\bar\theta_{xx},\bar w_x,\bar w_{xx})|^2|(\phi,\vartheta)|^2(t,x)\,dx .
\end{aligned}
\]
The right-hand side is integrable on $(0,\infty)$ by \eqref{est:apriori_global}, \eqref{eq:J-def}, \eqref{eq:phix-L2t}, Lemma~\ref{lem: interaction estimates}, Lemma~\ref{lem:rarefaction_property}~(2), and \eqref{error_contact}. This proves \eqref{eq:g-W11}, and hence $\lim_{t\to\infty}g(t)=0$. The Gagliardo--Nirenberg inequality and \eqref{est:apriori_global} then give
\[
    \|(\phi,\psi,\vartheta)(t,\cdot)\|_{L^\infty}\le C\|(\phi,\psi,\vartheta)(t,\cdot)\|_{L^2}^{1/2}\|(\phi,\psi,\vartheta)_x(t,\cdot)\|_{L^2}^{1/2}\le Cg(t)^{1/4}\longrightarrow0,
\]
and \eqref{est:large_time} follows by Lemma~\ref{lem:rarefaction_property}~(5). By \eqref{est:shift-bound} in Proposition~\ref{prop:apriori}, \eqref{eq:TM} with $T_M=\infty$, and the Sobolev embedding,
\[
    |\dot X(t)|\le C_0\|(\phi,\psi,\vartheta)(t,\cdot)\|_{L^\infty}\le C\varepsilon_1,\qquad t\ge0,
\]
so $X$ is Lipschitz on $[0,\infty)$, and $\|(\phi,\psi,\vartheta)(t,\cdot)\|_{L^\infty}\to0$ gives \eqref{est:shift}. This completes the proof of Theorem~\ref{thm:main}.

\section{Estimate on the $L^2$ perturbation}\label{sec:L2}
\setcounter{equation}{0}
In this section, we begin to prove Proposition \ref{prop:apriori}. As a first step, we present an $L^2$-estimate on the perturbation $(v-\bar{v},u-\bar{u},\theta-\bar{\theta},w-\bar{w})$. Specifically, we will derive the following estimate:
\begin{equation}
	\begin{aligned}\label{est:L2}
	&\sup_{t\in[0,T]}\|(v-\bar{v},u-\bar{u},\theta-\bar{\theta},w-\bar{w})\|_{L^2}^2
	+\delta_S\int_0^T |\dot{X}(t)|^2\, dt +\int_0^T (G^S+G^R+G_3+D)\, dt\\
	&\le C\|(v-\bar v,u-\bar u,\theta-\bar\theta,w-\bar w)(0,\cdot)
	\|_{L^2}^2 \\
    &\qquad+C (\varepsilon_1 + \delta_0^{1/2})\int_0^T \|(v-\bar{v},w-\bar{w})_x\|_{L^2}^2 \, dt +C\delta_0^{1/6},
	\end{aligned}
\end{equation}
where $G^S$ and $G^R$ are defined in \eqref{GSGR}, $G_3$ in \eqref{GSG3}, and $D$ in \eqref{GD}, which consists of $D_{u_1}$ and $D_{\theta_1}$ (see \eqref{Du1,Dtheta1-def} for details). To prove \eqref{est:L2}, we rely on estimating the relative entropy between the solution to \eqref{eq:NSFK-extended} and the superposition wave \eqref{eq:composite_extended}. Using the lower and upper bounds of the solution and superposition wave, the relative entropy estimate is reduced to the $L^2$-estimate for the perturbation. 
In Sections~\ref{sec:L2}--\ref{sec:H1} and Appendix~\ref{app:proof_heat}, $C$ denotes a positive constant that may change from line to line, depends only on $(v_-,u_-,\theta_-)$, $(v_*,u_*,\theta_*)$, $q$, $\gamma$, $R$, $\mu$, $\alpha$, and $\kappa$, and is independent of $\delta_0$, $\e_R$, $\varepsilon_1$, and $T$.

\subsection{Wave interaction estimates}\label{sec:interaction}
Before we present the relative entropy estimate, we first collect several estimates for interactions between the elementary waves. These will be used to control the terms when we estimate the relative entropy. We first prove \eqref{est:shift-bound}. Since $1<a<2$, $\bar v$ and $\bar\theta$ are bounded below by positive constants, and $|u^S_x|+|\theta^S_x|\le C|v^S_x|$ by Lemma~\ref{lem:shock-property}, \eqref{def:shift} gives
\[
|\dot X(t)|\le\frac{C}{\delta_S}\|(v-\bar v,u-\bar u,\theta-\bar\theta)(t,\cdot)\|_{L^\infty}\int_\R|v^S_x|\,dx=C\|(v-\bar v,u-\bar u,\theta-\bar\theta)(t,\cdot)\|_{L^\infty}.
\]
Hence, by \eqref{smallness} and the Sobolev embedding,
\[
|\dot X(t)|\le C\varepsilon_1 \qquad\text{and hence}\qquad |X(t)|\le C\varepsilon_1 t,\quad\mbox{for $0\le t\le T$}.
\]

We define the interaction terms $R_i$ for $i=1,2,3,4$ as
\begin{align*}
    &R_1:=|(v^R_x,\theta^R_x)|\,|(v^C-v_*,\theta^C-\theta_*,v^S-v^*,\theta^S-\theta^*)|,\\
    &R_2:=|(v^C_x,\theta^C_x)|\,|(v^R-v_*,\theta^R-\theta_*,v^S-v^*,\theta^S-\theta^*)|,\\
    &R_3:=|(v^S_x,\theta^S_x)|\,|(v^R-v_*,\theta^R-\theta_*,v^C-v^*,\theta^C-\theta^*)|,\\
    &R_4:=|v^R_x||v^C_x|+|v^R_x||v^S_x|+|v^C_x||v^S_x|.
\end{align*}

\begin{lemma}\label{lem: interaction estimates}
Let $X$ be the shift defined by \eqref{def:shift}. Under the hypotheses of Proposition \ref{prop:apriori}, there exists a constant $C>0$, independent of $T$, $\delta_S$, $\delta_C$, and $\e_R$, such that for all $t\in[0,T]$, the following estimates hold:
\begin{equation}\label{est-interaction}
\begin{aligned}
    \|R_1\|_{L^2}&\le C\Big(\delta_S e^{-C\delta_S t}+\delta_C e^{-Ct}\Big)
    \delta_R \e_R^{1/2}+C(\delta_S+\delta_C)\delta_R\varepsilon_R^{1/2} \bigl(1+\e_R^2t^2\bigr)^{-q/2},\\
    \|R_2\|_{L^2}&\le C\delta_R\delta_C e^{-Ct}+ C\delta_S\delta_Ce^{-C\delta_St}+C\delta_R\delta_C(1+t)^{-1/4}\bigl(1+\e_R^2(1+t)^2\bigr)^{-q/2},\\
    \|R_3\|_{L^2}&\le C \bigl( \delta_R +\delta_C\bigr)\delta_S^{3/2} e^{-C \delta_S t}+C\delta_S^{3/2}\delta_R(1+\e_R^2t^2)^{-q/4},\\
    \|R_4\|_{L^2}&\le C\delta_S^2\delta_R\e_R^{1/2}e^{-C\delta_St} + C\delta_S^2\delta_R\e_R^{1/2}(1+\e^2_Rt^2)^{-q/2}+ C\delta_S\delta_Ce^{-C\delta_St}\\
    &\quad+C\delta_R\delta_C\e_R^{1/2}(1+t)^{-1/4}\bigl(1+\e_R^2(1+t)^2\bigr)^{-q/2},\\
    \|Q_i^I\|_{L^2}
        &\le C\big( \delta_R \varepsilon_R^{1/2} \delta_S+\delta_R \delta_S^{3/2}+\delta_C \delta_S\big)e^{-C \delta_S t}
        + C\delta_R \delta_C e^{-Ct} \\
        &\quad+ C((\delta_S+\delta_C)\delta_R\varepsilon_R^{1/2}+\delta_S^{3/2}\delta_R)\bigl(1+\varepsilon_R^2 t^2\bigr)^{-q/4}\\
        &\quad +C\delta_R\delta_C(1+t)^{-1/4}\bigl(1+\e_R^2(1+t)^2\bigr)^{-q/4},
        \qquad i=1,2,\\
    \|Q_3\|_{L^2}&\le C\varepsilon_1\Big[(\delta_R\delta_S^{3/2}+\delta_C\delta_S)e^{-C\delta_St}+\delta_S^{3/2}\delta_R(1+\e_R^2t^2)^{-q/4}\\
    &\qquad \qquad +\delta_R\delta_C\e_R^{1/2}(1+t)^{-1/4}\bigl(1+\e_R^2(1+t)^2\bigr)^{-q/2}\Big].
\end{aligned}
\end{equation}
Moreover, since $\e_R=\delta_0/2$ and $\delta_C,\delta_S<\delta_0$, for any $T>0$, we have
\begin{equation}\label{est-time-int-interaction}
\begin{aligned}
    \int_0^T \|R_1\|_{L^2}\,dt
        &\le C (1+\delta_C)\,\delta_R\varepsilon_R^{1/2}
            +C(\delta_S+\delta_C)\,\delta_R\e_R^{-1/2}\le C\delta_0^{1/2},\\
    \int_0^T \|R_2\|_{L^2}\,dt
        &\le C\delta_C\delta_R\bigl(1+\e_R^{-3/4}\bigr) + C\delta_C \le C \delta_0^{1/4},\\
    \int_0^T \|R_3\|_{L^2}\,dt
        &\le C\delta_S^{1/2}(\delta_R+\delta_C)+C\delta_S^{3/2}\delta_R\e_R^{-1}\le C\delta_0^{1/2},\\
    \int_0^T \|R_4\|_{L^2}\,dt
        &\le C(\delta_S\delta_R\e^{1/2}_R+\delta_S^2\delta_R\e_R^{-1/2} + \delta_R\delta_C\e_R^{-1/4}+\delta_C)\le C\delta_0^{3/4},\\
    \int_0^T \|Q_i^I\|_{L^2}\,dt
        &\le C\Big(\delta_R \varepsilon_R^{1/2}+ \delta_R \delta_S^{1/2}+ \delta_C 
            + (\delta_S+\delta_C)\delta_R\e_R^{-1/2} +\delta_S^{3/2}\delta_R\e_R^{-1}+\delta_R\delta_C\e_R^{-3/4}\Big) \le C \delta_0^{1/4},\\
    \int_0^T \|Q_3\|_{L^2}\,dt
        &\le C \varepsilon_1 (\delta_R\delta_S^{1/2}+\delta_C+\delta_S^{3/2}\delta_R\e_R^{-1}+\delta_R\delta_C\e_R^{-1/4}) \le C \varepsilon_1 \delta_0^{1/2}.
\end{aligned}
\end{equation}
\end{lemma}
\begin{proof} 
\noindent $\bullet$ (Estimates of $R_i$): Since $\lambda_{1*}<0$, it follows from Lemma~\ref{lem:rarefaction_property} that for all $t\ge0$ and $x>\frac{\lambda_{1*}}{2}(1+t)$,
\begin{equation}\label{ptwise-R}
\begin{aligned}
&\big|(v^R-v_*,\theta^R-\theta_*)\big|
 \le C\delta_R \bigl(1+(\e_R|x-\lambda_{1*}(1+t)|)^2\bigr)^{-q/2},\\
&\big|(v^R_x,\theta^R_x)\big|
 \le C\delta_R \e_R\bigl(1+(\e_R|x-\lambda_{1*}(1+t)|)^2\bigr)^{-q},
\end{aligned}
\end{equation}
and
\begin{equation}\label{L2-Rx}
\|(v^R_x,\theta^R_x)(t,\cdot)\|_{L^2}
\le C\min\Big\{\delta_R\varepsilon_R^{1/2},\ \delta_R^{1/2}(1+t)^{-1/2}\Big\}.
\end{equation}

\medskip
\noindent $\diamond$ ($R_1$-estimate):
We split the estimate of $\|R_1\|_{L^2}$ as
\begin{equation}\label{R1-split}
\begin{aligned}
 \|R_1 \|_{L^2}
 &\le  \Big[\Big(\int_{\left\{x \le \frac{\lambda_{1*}}{2}(1+t)\right\}}
 +\int_{\left\{x > \frac{\lambda_{1*}}{2}(1+t)\right\}}\Big)
 \big|(v^R_x,\theta^R_x)\big|^2
 \big|(v^C-v_*, \theta^C-\theta_*, v^S-v^*, \theta^S-\theta^*) \big|^2\,dx\Big]^{\frac12}\\
& =: \sqrt{R_{11}+R_{12}}.
\end{aligned}
\end{equation}
For $R_{11}$, we first note that the smallness of $\e_1$ implies $|X(t)|\le C\varepsilon_1 t<\frac{\sigma t}{4}$. Then, for $x\le \frac{\lambda_{1*}}{2}(1+t)$,
\[
x-\sigma t- X(t)\le \frac{\lambda_{1*}}{2}(1+t)-\frac{3\sigma t}{4}<-\frac{3\sigma t}{4}<0,
\]
and hence $|x-\sigma t-X(t)|\ge \frac{3\sigma}{4}t$. Therefore, by Lemma~\ref{lem:viscous_contact} and Lemma~\ref{lem:shock-property}, for any $x\le \frac{\lambda_{1*}}{2}(1+t)<0$,
\begin{equation}\label{shock-contact-left}
\begin{aligned}
&\big|(v^S-v^*,\theta^S-\theta^*)\big|
\le C \delta_S e^{-C\delta_S|x-\sigma t -X(t)|} \le C \delta_S e^{-\frac{C\delta_S|x-\sigma t -X(t)|}{2}} e^{-C\delta_S t},\\
&\big|(v^S_x,\theta^S_x)\big|
\le C \delta_S^{2} e^{-C\delta_S|x-\sigma t -X(t)|} \le C \delta_S^2 e^{-\frac{C\delta_S|x-\sigma t -X(t)|}{2}} e^{-C\delta_S t},\\
&\big|(v^C-v_*, \theta^C-\theta_*)\big|+\big|(v^C_x, \theta^C_x)\big|
\le C\delta_Ce^{-\frac{C|x|^2}{1+t}}
\le C\delta_Ce^{-\frac{C|x|^2}{2(1+t)}}e^{-C t}.
\end{aligned}
\end{equation}
Consequently, using \eqref{L2-Rx} and \eqref{shock-contact-left},
\begin{equation}\label{R1L}
\begin{aligned}
R_{11}
&\le C\Big(\delta_S^2 e^{-2C\delta_S t}+\delta_C^2 e^{-2Ct}\Big)
\int_{\R} \big|(v^R_x,\theta^R_x)\big|^2\,dx\le C\Big(\delta_S^2 e^{-2C\delta_S t}+\delta_C^2 e^{-2Ct}\Big)
\delta_R^2\e_R.
\end{aligned}
\end{equation}
For $R_{12}$, we use \eqref{ptwise-R} and
$\big|(v^C-v_*,\theta^C-\theta_*, v^S-v^*, \theta^S-\theta^*) \big|\le C(\delta_S+\delta_C)$ to get
\begin{equation}\label{R1R-0}
\begin{aligned}
R_{12}&\le
C(\delta_S^2+\delta_C^2)\delta_R^2\varepsilon_R^2\int_{\left\{x > \frac{\lambda_{1*}}{2}(1+t)\right\}} (1+(\e_R|x-\lambda_{1*}(1+t)|)^2)^{-2q}\,d x.
\end{aligned}
\end{equation}
However, for $x>\frac{\lambda_{1*}}{2}(1+t)$, we have
\[x-\lambda_{1*}(1+t)=\underbrace{x-\frac{\lambda_{1*}}{2}(1+t)}_{>0}\underbrace{-\frac{\lambda_{1*}}{2}(1+t)}_{>0}>0,\]
which implies
\begin{equation}\label{split}
	(1+(\e_R|x-\lambda_{1*}(1+t)|)^2)^{-1}\le \left(1+\left(\e_R\left|x-\frac{\lambda_{1*}}{2}(1+t)\right|\right)^2\right)^{-1/2}\left(1+\left(\e_R\left|\frac{\lambda_{1*}}{2}(1+t)\right|\right)^2\right)^{-1/2}.
\end{equation}
We substitute \eqref{split} into \eqref{R1R-0} and integrate in $x$ to get
\begin{equation}\label{R1R}
	R_{12}\le C(\delta^2_S+\delta^2_C)\delta^2_R\e_R(1+\e_R^2t^2)^{-q}.
\end{equation}
Combining \eqref{R1-split}, \eqref{R1L}, and \eqref{R1R},
\begin{equation}\label{R1}
\|R_1\|_{L^2}
\le C\Big(\delta_S e^{-C\delta_S t}+\delta_C e^{-Ct}\Big)
\delta_R \e_R^{1/2}+C(\delta_S+\delta_C)\delta_R\varepsilon_R^{1/2}
\bigl(1+\e_R^2t^2\bigr)^{-q/2}.
\end{equation}

\medskip
\noindent $\diamond$ ($R_2$-estimate): We decompose $R_2$ as
\begin{equation*}
    \begin{aligned}
\|R_2 \|_{L^2}
&\le C \Bigg[\Bigg(\int_{\{x\le \frac{\lambda_{1*}}{2}(1+t)\}} + \int_{\{x> \frac{\lambda_{1*}}{2}(1+t)\}}\Bigg)
\big|(v^C_x,\theta^C_x)\big|^2 \big|(v^R-v_*, \theta^R-\theta_*) \big|^2 dx\\
 &\quad+ \Bigg(\int_{\{x\le \frac{\sigma t}{2}\}} + \int_{\{x> \frac{\sigma t}{2}\}}\Bigg)
\big|(v^C_x,\theta^C_x)\big|^2 \big|(v^S-v^*, \theta^S-\theta^*) \big|^2 dx\Bigg]^{\frac12}\\
& =: C\sqrt{R_{21}+R_{22}+R_{23}+R_{24}}.
\end{aligned}
\end{equation*}
For $R_{21}$, we use $|(v^R-v_*,\theta^R-\theta_*)|\le C\delta_R$ and \eqref{shock-contact-left} to get
\begin{equation*}
    \begin{aligned}
        R_{21}&\le C\delta_R^2\delta_C^2e^{-Ct}\int_{\{x \le \frac{\lambda_{1*}}{2}(1+t)\}} e^{-\frac{C|x|^2}{2(1+t)}}\,dx \le C\delta_R^2\delta_C^2e^{-Ct}.
    \end{aligned}
\end{equation*}
To estimate $R_{22}$, we recall \eqref{ptwise-R} for $x>\frac{\lambda_{1*}}{2}(1+t)$, which gives $|(v^R-v_*,\theta^R-\theta_*)|\le C\delta_R(1+\e_R^2(1+t)^2)^{-q/2}$ on this region. Combining this with $\|(v^C_x,\theta^C_x)\|_{L^2}\le C\delta_C(1+t)^{-1/4}$, which follows from Lemma \ref{lem:viscous_contact}, we obtain
\[R_{22}\le C\delta_R^2\delta_C^2(1+t)^{-1/2}\bigl(1+\e_R^2(1+t)^2\bigr)^{-q}.\]

For $R_{23}$, we use the smallness of $\e_1$ to get $x-\sigma t-X(t)\le -\frac{\sigma t}{4}<0$ for $x\le \frac{\sigma t}{2}$. This yields $|x-\sigma t-X(t)|\ge \frac{\sigma t}{4}$ and consequently,
\begin{equation}\label{shock_left_est}
	|(v^S-v^*,\theta^S-\theta^*)|\le C\delta_S e^{-\frac{C\delta_S|x-\sigma t-X(t)|}{2}}e^{-C\delta_S t},\quad |(v^S_x,\theta^S_x)|\le C\delta^2_S e^{-\frac{C\delta_S|x-\sigma t-X(t)|}{2}}e^{-C\delta_S t}.
\end{equation}
This and Lemma \ref{lem:viscous_contact} yield that
\[R_{23}\le C\delta_S^2\delta_C^2e^{-C\delta_St}.\]
Finally, for $R_{24}$, we use $x>\frac{\sigma t}{2}>0$ and Lemma~\ref{lem:viscous_contact} to get
\begin{equation} \label{contact-right}
    |(v^C-v^*,\theta^C-\theta^*)|+|(v^C_x,\theta^C_x)| \le C \delta_C e^{-\frac{C|x|^2}{1+t}} \le C \delta_C e^{-\frac{C|x|^2}{2(1+t)}}e^{-Ct},
\end{equation}
from which we obtain
\begin{equation*} 
    \begin{aligned}
        R_{24} \le C\delta_S^2\delta_C^2e^{-Ct}.
    \end{aligned}
\end{equation*}
Combining all the estimates of $R_{2i}$ for $i=1,2,3,4$, we get
\begin{equation}\label{R2}
\|R_2 \|_{L^2}\le C\delta_R\delta_C e^{-Ct}+ C\delta_S\delta_Ce^{-C\delta_St}+C\delta_R\delta_C(1+t)^{-1/4}\bigl(1+\e_R^2(1+t)^2\bigr)^{-q/2}.
\end{equation}

\medskip
\noindent $\diamond$ ($R_3$-estimate): We split $R_3$ at $x=\frac{\sigma t}{2}$ as
\begin{align*}
\begin{aligned}
\|R_3 \|_{L^2}
 &=\Bigg[\Bigg(\int_{\{x\le \frac{\sigma t}{2}\}} + \int_{\{x> \frac{\sigma t}{2}\}}\Bigg)
 \big|(v^S_x,\theta^S_x)\big|^2 
 \big|(v^R-v_*, \theta^R-\theta_*,v^C-v^*, \theta^C-\theta^*) \big|^2 dx\Bigg]^{\frac12}\\
& =: C\sqrt{R_{31}+R_{32}}.
\end{aligned}
\end{align*}
By using $|(v^R-v_*,\theta^R-\theta_*)|\le C\delta_R$ and \eqref{shock_left_est}, we estimate $R_{31}$ as
\begin{equation} \label{R3L}
    \begin{aligned}
        R_{31}&\le C\bigl( \delta_R^2+ \delta_C^2\bigr)\delta_S^4\int_{\{x \le \frac{\sigma}{2}t\}} e^{-\frac{C \delta_S|x-\sigma t -X(t)|}{2}} e^{-C \delta_S t}  \, dx \le C\bigl( \delta_R^2+ \delta_C^2\bigr) \delta_S^3 e^{-C \delta_St}.
    \end{aligned}
\end{equation}
For $R_{32}$, we use $x>\frac{\sigma t}{2}>0>\frac{\lambda_{1*}}{2}(1+t)$, \eqref{split} and \eqref{contact-right} to derive
\begin{equation} \label{R3R}
    \begin{aligned}
        R_{32} &\le C\left(\frac{\delta_R^2}{(1+\e_R^2t^2)^{q/2}}+ \delta_C^2e^{-Ct}\right)\delta_S^2\int_{\R}|(v^S_x,\theta^S_x)|\,dx\le C\delta_S^3\left(\frac{\delta_R^2}{(1+\e_R^2t^2)^{q/2}}+ \delta_C^2e^{-Ct}\right)
    \end{aligned}
\end{equation}

Combining \eqref{R3L} and \eqref{R3R}, we have
\begin{equation}\label{R3}
    \|R_3\|_{L^2}\le C \bigl( \delta_R +\delta_C\bigr)\delta_S^{3/2} e^{-C \delta_S t}+C\delta_S^{3/2}\delta_R(1+\e_R^2t^2)^{-q/4}.
\end{equation}

\noindent $\diamond$ ($R_4$-estimate): We use similar estimates to those of $R_1$ and $R_2$ to derive
\begin{equation}\label{R4}
    \begin{aligned}
        \|R_4\|_{L^2} &\le C\delta_S^2\delta_R\e_R^{1/2}e^{-C\delta_St} + C\delta_S^2\delta_R\e_R^{1/2}(1+\e^2_Rt^2)^{-q/2}\\
        &\quad+C\delta_R\delta_C\e_R^{1/2}(1+t)^{-1/4}\bigl(1+\e_R^2(1+t)^2\bigr)^{-q/2}+ C\delta_S\delta_Ce^{-C\delta_St}.
    \end{aligned}
\end{equation}
Here the rarefaction--shock interaction $|v^R_x||v^S_x|$ is estimated by splitting at $x=\frac{\sigma t}{2}$ as for $R_1$, while the rarefaction--contact interaction $|v^R_x||v^C_x|$ is estimated as for $R_{22}$, which produces the algebraic decay above.

\medskip
	\noindent $\bullet$ (Estimate of $Q_1^I$): Recall the definition of $Q_1^I$:
\begin{align*}
    Q_1^I&=\left(\bar{p}-p^R-p^C-p^S\right)_x- \mu\left(\frac{\bar{u}_x}{\bar{v}}-\frac{u^R_x}{v^R}-\frac{u^C_x}{v^C}-\frac{u^S_x}{v^S}\right)_x\\
    &\quad-\kappa\left(  \left(-\frac{\bar{v}_{xx}}{\bar{v}^5} + \frac{5(\bar{v}_x)^2}{2\bar{v}^6} \right)- \left(-\frac{v^R_{xx}}{(v^R)^5} + \frac{5(v^R_x)^2}{2(v^R)^6} \right)\right.\\
    &\hspace{4em}\left.-\left(-\frac{v^C_{xx}}{(v^C)^5} + \frac{5(v^C_x)^2}{2(v^C)^6} \right)- \left(-\frac{v^S_{xx}}{(v^S)^5} + \frac{5(v^S_x)^2}{2(v^S)^6} \right)\right)_x.
\end{align*}
The first two terms of $Q_1^I$ can be estimated as
    \begin{align*}
       \bigl| \left(\bar{p}-p^R-p^C-p^S\right)_x \bigr| &\le\Biggl| R\theta^R_x\left(\frac{1}{\bar{v}}-\frac{1}{v^R}\right)+R\theta^C_x\left(\frac{1}{\bar{v}}-\frac{1}{v^C}\right)+R\theta^S_x\left(\frac{1}{\bar{v}}-\frac{1}{v^S}\right)\\
       &\qquad -v^R_x\left( \frac{\bar{p}}{\bar{v}}-\frac{p^R}{v^R}\right)-v^C_x\left( \frac{\bar{p}}{\bar{v}}-\frac{p^C}{v^C}\right)-v^S_x\left( \frac{\bar{p}}{\bar{v}}-\frac{p^S}{v^S}\right)\Biggr|\\
       &\le C (R_1+R_2+R_3),
    \end{align*}
and
    \begin{align*}
        &\mu\Biggl|\left(\frac{\bar{u}_x}{\bar{v}}-\frac{u^R_x}{v^R}-\frac{u^C_x}{v^C}-\frac{u^S_x}{v^S}\right)_x \Biggr|
        \le \left| \left( u^R_x \left( \frac{\mu}{\bar{v}}-\frac{\mu}{v^R} \right) \right)_x \right|+\left| \left( u^C_x \left( \frac{\mu}{\bar{v}}-\frac{\mu}{v^C} \right) \right)_x \right| +\left|\left( u^S_x \left( \frac{\mu}{\bar{v}}- \frac{\mu}{v^S}\right) \right)_x \right|\\
        &\le C \Bigl[ |(u^R_{xx},u^R_x v^R_x)| |(v^C-v_*,v^S-v^*)|+|(u^C_{xx},u^C_x v^C_x)| |(v^R-v_*,v^S-v^*)| \\
        &\qquad + |(u^S_{xx},u^S_x v^S_x)| |(v^R-v_*,v^C-v^*)| + |u^R_x||(v^C_x,v^S_x)|+|u^C_x||(v^R_x,v^S_x)| +|u^S_x| |(v^R_x,v^C_x)|\Bigr].
    \end{align*}
Next, we estimate the third term of $Q^I_1$ as
    \begin{align*}
        &\kappa \left| \left(  \left(-\frac{\bar{v}_{xx}}{\bar{v}^5} + \frac{5(\bar{v}_x)^2}{2\bar{v}^6} \right)- \left(-\frac{v^R_{xx}}{(v^R)^5} + \frac{5(v^R_x)^2}{2(v^R)^6} \right)-\left(-\frac{v^C_{xx}}{(v^C)^5} + \frac{5(v^C_x)^2}{2(v^C)^6} \right)- \left(-\frac{v^S_{xx}}{(v^S)^5} + \frac{5(v^S_x)^2}{2(v^S)^6} \right) \right)_x \right| \\
        &\le \kappa \left| \left(  -  \frac{\bar{v}_{xx}}{\bar{v}^5} +  \frac{v^R_{xx}}{(v^R)^5} + \frac{v^C_{xx}}{(v^C)^5} +\frac{v^S_{xx}}{(v^S)^5} \right)_x \right| +\frac{5 \kappa}{2}  \left| \left(   \frac{(\bar{v}_x)^2}{\bar{v}^6} -  \frac{(v^R_x)^2}{(v^R)^6} -\frac{(v^C_x)^2}{(v^C)^6}-  \frac{(v^S_x)^2}{(v^S)^6} \right)_x \right|.
    \end{align*}
    However, since $\bar{v}_{xx}=v^R_{xx}+v^C_{xx}+v^S_{xx}$, we have
    \begin{align*}
        &\left| \left(  -  \frac{\bar{v}_{xx}}{\bar{v}^5} +  \frac{v^R_{xx}}{(v^R)^5} + \frac{v^C_{xx}}{(v^C)^5} + \frac{v^S_{xx}}{(v^S)^5} \right)_x \right|
        \\
        &=\left| \left( v^R_{xx} \left(\frac{1}{\bar{v}^5}-\frac{1}{(v^R)^5} \right)+ v^C_{xx} \left(\frac{1}{\bar{v}^5}-\frac{1}{(v^C)^5} \right)+ v^S_{xx} \left(\frac{1}{\bar{v}^5}-\frac{1}{(v^S)^5} \right) \right)_x \right|\\
         &\le C \Bigl[ |(v^R_{xxx},v^R_{xx} v^R_x)| |(v^C-v_*,v^S-v^*)|+|(v^C_{xxx},v^C_{xx} v^C_x)| |(v^R-v_*,v^S-v^*)| \\
        &\qquad + |(v^S_{xxx},v^S_{xx} v^S_x)| |(v^R-v_*,v^C-v^*)| + |v^R_{xx}||(v^C_x,v^S_x)|+|v^C_{xx}||(v^R_x,v^S_x)| +|v^S_{xx}| |(v^R_x,v^C_x)|\Bigr].
    \end{align*}
    Similarly, we use   $(\bar{v}_{x})^2=(v^R_{x})^2+(v^C_x)^2+(v^S_{x})^2+2 (v^R_x v^S_x+v^R_xv^C_x+v^C_xv^S_x)$ to obtain
    \begin{align*}
        &\left| \left(   \frac{(\bar{v}_x)^2}{\bar{v}^6} -  \frac{(v^R_x)^2}{(v^R)^6}-\frac{(v^C_x)^2}{(v^C)^6} -  \frac{(v^S_x)^2}{(v^S)^6} \right)_x \right|\\
        &\le \left| \left( (v^R_{x})^2 \left(\frac{1}{\bar{v}^6}-\frac{1}{(v^R)^6} \right)+(v^C_{x})^2 \left(\frac{1}{\bar{v}^6}-\frac{1}{(v^C)^6} \right)+ (v^S_{x})^2 \left(\frac{1}{\bar{v}^6}-\frac{1}{(v^S)^6} \right) \right)_x \right| \\
        &\quad+2 \left|   \left( \frac{1}{\bar{v}^6}v^R_xv^C_x  \right)_x \right|+2 \left|   \left( \frac{1}{\bar{v}^6}v^C_xv^S_x  \right)_x \right|+2 \left|   \left( \frac{1}{\bar{v}^6}v^R_xv^S_x  \right)_x \right| \\
        &\le C\Bigl[ | (v^R_xv^R_{xx},|v^R_x|^3 ) | |(v^C-v_*,v^S-v^*)|+| (v^C_xv^C_{xx},|v^C_x|^3 ) | |(v^R-v_*,v^S-v^*)| \\
        &\qquad + | (v^S_xv^S_{xx},|v^S_x|^3 ) | |(v^R-v_*,v^C-v^*)|+|( |v^R_x|^2,v^R_{xx})| |(v^C_x,v^S_x)|  \\
        &\qquad+|( |v^C_x|^2,v^C_{xx})| |(v^R_x,v^S_x)|+|( |v^S_x|^2,v^S_{xx})| |(v^R_x,v^C_x)|\Bigr].
    \end{align*}
    Combining all the estimates above, we obtain
    \begin{align*}
        |Q_1^I| &\le C\Big( |(v^R_x,u^R_{xx},u^R_xv^R_x,v^R_{xx}v^R_x,v^R_{xxx},|v^R_x|^3)| |(v^C-v_*,v^S-v^*)|\\
        &\qquad+|(v^C_x,u^C_{xx},u^C_xv^C_x,v^C_{xx}v^C_x,v^C_{xxx},|v^C_x|^3)| |(v^R-v_*,v^S-v^*)|\\
        &\qquad+ |(v^S_x,u^S_{xx},u^S_xv^S_x,v^S_{xx}v^S_x,v^S_{xxx},|v^S_x|^3)| | (v^R-v_*,v^C-v^*)|  \\
        &\qquad+|( u^R_x,|v^R_x|^2,v^R_{xx})| |(v^C_x,v^S_x)| +|( u^C_x,|v^C_x|^2,v^C_{xx})| |(v^R_x,v^S_x)|+|(u^S_x, |v^S_x|^2,v^S_{xx})| |(v^R_x,v^C_x)|\Big).
    \end{align*}
By Lemma \ref{lem:rarefaction_property}~(2), $|(v^R_{xx}, v^R_{xxx}, u^R_{xx})| \le C|(v^R_x, u^R_x)|$ and $\|v^R_x\|_{L^\infty(\mathbb{R})} \le C\delta_R\varepsilon_R$, and by Lemma~\ref{lem:shock-property}, $|(v^S_{xx},u^S_{xx},v^S_{xxx})|\le C|v^S_x|$. By Lemma~\ref{lem:viscous_contact}, the derivatives $\partial_x^n(v^C,u^C,\theta^C)$, $n\ge1$, satisfy the pointwise and $L^2$ bounds on $(v^C_x,\theta^C_x)$ used in the proofs of \eqref{R2} and \eqref{R4}. Therefore, the proofs of \eqref{R1}--\eqref{R4} apply to each term on the right-hand side, and
\begin{align*}
    \|Q_1^I\|_{L^2} &\le C\big( \delta_R \varepsilon_R^{1/2} \delta_S+\delta_R \delta_S^{3/2}+\delta_C \delta_S\big)e^{-C \delta_S t}
        + C\delta_R \delta_C e^{-Ct} \\
        &\quad+ C((\delta_S+\delta_C)\delta_R\varepsilon_R^{1/2}+\delta_S^{3/2}\delta_R)\bigl(1+\varepsilon_R^2 t^2\bigr)^{-q/4}+C\delta_R\delta_C(1+t)^{-1/4}\bigl(1+\e_R^2(1+t)^2\bigr)^{-q/4}. 
\end{align*}

\noindent $\bullet$ (Estimate of $Q_2^I$):
Similarly, $Q_2^I$ can be estimated as
\begin{align*}
    |Q_2^I|&\le |\bar{p}\bar{u}_x-p^Ru^R_x-p^Cu^C_x-p^Su^S_x|+\alpha\left|\left(\frac{\bar{\theta}_x}{\bar{v}}-\frac{\theta^R_x}{v^R}-\frac{\theta^C_x}{v^C}-\frac{\theta^S_x}{v^S}\right)_x\right| \\
    &\quad +\mu\left|\frac{\bar{u}^2_x}{\bar{v}}-\frac{(u^R_x)^2}{v^R}-\frac{(u^C_x)^2}{v^C}-\frac{(u^S_x)^2}{v^S}\right|\\
    &\le \bigl| u^R_x(\bar{p}-p^R)+u^C_x(\bar{p}-p^C)+u^S_x(\bar{p}-p^S)\bigr| \\
    &\quad + \alpha \Biggl| \left(\theta^R_{x}\left(\frac{1}{\bar{v}}-\frac{1}{v^R}\right)\right)_x+\left(\theta^C_{x}\left(\frac{1}{\bar{v}}-\frac{1}{v^C}\right)\right)_x+\left(\theta^S_{x}\left(\frac{1}{\bar{v}}-\frac{1}{v^S}\right)\right)_x\Biggr|\\
    &\quad + \mu\left|(u^R_x)^2\left(\frac{1}{\bar{v}}-\frac{1}{v^R}\right)+(u^C_x)^2\left(\frac{1}{\bar{v}}-\frac{1}{v^C}\right)+(u^S_x)^2\left(\frac{1}{\bar{v}}-\frac{1}{v^S}\right)\right|\\
    &\quad + \frac{2\mu}{\bar{v}}\left(|u^R_x||u^C_x|+|u^R_x||u^S_x|+|u^C_x||u^S_x|\right)\\
    &\le C \Bigl[ |(u^R_x,\theta^R_{xx},\theta^R_x v^R_x)| |(v^C-v_*,v^S-v^*)|+|(u^C_x,\theta^C_{xx},\theta^C_x v^C_x)| |(v^R-v_*,v^S-v^*)| \\
    &\qquad + |(u^S_x,\theta^S_{xx},\theta^S_x v^S_x)| |(v^R-v_*,v^C-v^*)| + |\theta^R_x||(v^C_x,v^S_x)|+|\theta^C_x||(v^R_x,v^S_x)| +|\theta^S_x| |(v^R_x,v^C_x)|\Bigr].
\end{align*}
Applying the pointwise bound $|\theta^R_{xx}| \le C|\theta^R_x|$ similarly as above, we can control the right-hand side by the interaction terms $R_i$. This yields
\begin{align*}
    \|Q_2^I\|_{L^2} &\le C\big( \delta_R \varepsilon_R^{1/2} \delta_S+\delta_R \delta_S^{3/2}+\delta_C \delta_S\big)e^{-C \delta_S t}
    + C\delta_R \delta_C e^{-Ct} \\
    &\quad+ C((\delta_S+\delta_C)\delta_R\varepsilon_R^{1/2}+\delta_S^{3/2}\delta_R)\bigl(1+\varepsilon_R^2 t^2\bigr)^{-q/4}+C\delta_R\delta_C(1+t)^{-1/4}\bigl(1+\e_R^2(1+t)^2\bigr)^{-q/4}.        
\end{align*}

\noindent $\bullet$ (Estimate of $Q_3$):
Again, we recall the definition of $Q_3$
\[ Q_3= - \sqrt{\kappa}\dot{X} \left(   v^S_x \left(\frac{ 1}{(v^S)^{5/2}} -  \frac{ 1}{(\bar{v})^{5/2}} \right)  \right)_x,\]
from which we directly obtain 
\begin{align*}
    |Q_3| &\le C |\dot{X}| \left[ (|v^S_{xx}|+|v^S_x|^2) |(v^R-v_*,v^C-v^*)| + |v^S_{x}|( |v^R_x|,|v^C_x|)\right].
\end{align*}
Therefore, 
\begin{align*}
	\|Q_3\|_{L^2} &\le C \varepsilon_1 (\|R_3\|_{L^2}+\|R_4\|_{L^2})\\
	&\le C\varepsilon_1\Bigl((\delta_R\delta_S^{3/2}+\delta_C\delta_S)e^{-C\delta_St}+\delta_S^{3/2}\delta_R(1+\e_R^2t^2)^{-q/4}\\
	&\qquad\qquad+\delta_R\delta_C\e_R^{1/2}(1+t)^{-1/4}\bigl(1+\e_R^2(1+t)^2\bigr)^{-q/2}\Bigr).
\end{align*}
Now, integrating \eqref{est-interaction} in time and using $\e_R=\delta_0/2$ together with $\delta_C,\delta_S<\delta_0$, we obtain \eqref{est-time-int-interaction}. Here all the algebraic factors are integrable in time since $q>3$.
\end{proof}

\subsection{Weighted relative entropy estimate}

To obtain the $L^2$-estimate on the perturbation between the solution and the superposition wave, we derive a weighted relative entropy estimate between them. Let $U=(v,u,E,w)^\top$ be a tuple of conserved quantities in \eqref{eq:NSFK-extended}. Recall that the physical entropy of the Navier--Stokes--Fourier system \cite{KVW25} is given by
\[s(U):=R\log v+\frac{R}{\gamma-1}\log \theta,\]
and the mathematical entropy is $\eta(U)=-s(U)$. Then, it follows from the formula \eqref{eq:total_energy} of the energy
\[E=e+\frac{u^2}{2}=\frac{R}{\gamma-1}\theta + \frac{w^2}{2}+\frac{u^2}{2}\]
that $dE = \frac{R}{\gamma-1}d\theta +udu+wdw$.
Therefore,
\[\theta ds = \frac{R\theta}{v}\,dv + \frac{R}{\gamma-1}d\theta=p\,dv + dE-udu-wdw,\]
which implies
\[\nabla_U s(U)=\left(\frac{p}{\theta},-\frac{u}{\theta},\frac{1}{\theta},-\frac{w}{\theta}\right).\]
Now, let $\bar{U}=(\bar{v},\bar{u},\bar{E},\bar{w})$ be a tuple of the superposition wave. Then the relative entropy $\eta(U|\bar{U})$ can be explicitly written as
\begin{align*}
	\eta(U|\bar{U}) &= -s(U|\bar{U})= -s(U) + s(\bar{U}) + \nabla_Us(\bar{U})\cdot(U-\bar{U})\\
	& = -R\log\frac{v}{\bar{v}}-\frac{R}{\gamma-1}\log\frac{\theta}{\bar{\theta}} + \frac{R}{\bar{v}}(v-\bar{v})-\frac{\bar{u}}{\bar{\theta}}(u-\bar{u})+\frac{1}{\bar{\theta}}(E-\bar{E})-\frac{\bar{w}}{\bar{\theta}}(w-\bar{w})\\
	& = R\left(\frac{v}{\bar{v}}-1-\log \frac{v}{\bar{v}}\right)+\frac{R}{\gamma-1}\left(\frac{\theta}{\bar{\theta}}-1-\log\frac{\theta}{\bar{\theta}}\right)+\frac{(u-\bar{u})^2}{2\bar{\theta}}+\frac{(w-\bar{w})^2}{2\bar{\theta}}\\
	&=R\Phi\left(\frac{v}{\bar{v}}\right)+\frac{R}{\gamma-1}\Phi\left(\frac{\theta}{\bar{\theta}}\right)+\frac{(u-\bar{u})^2}{2\bar{\theta}}+\frac{(w-\bar{w})^2}{2\bar{\theta}},
\end{align*}
where $\Phi(z) := z-1-\log z$ is a nonnegative function that is locally quadratic near $z=1$. Since $v$, $\bar v$, $\theta$, and $\bar\theta$ are bounded above and below by positive constants, $\eta(U|\bar U)$ is comparable to $|(v-\bar v,u-\bar u,\theta-\bar\theta,w-\bar w)|^2$. In the following, we begin with estimating the time derivative of the weighted relative entropy of the form
\[\frac{d}{dt}\int_{\R} a(t,x)\bar{\theta}(t,x) \eta(U(t,x)|\bar{U}(t,x))\,dx,\]
where $a$ is the weight function defined in \eqref{def:weight}.

\begin{lemma}\label{lem:rel_ent}
	Let $U$ and $\bar{U}$ be solutions to \eqref{eq:NSFK-extended} and \eqref{eq:composite_extended} respectively. Then, we have
	\begin{equation}\label{est:rel_ent}
		\frac{d}{dt}\int_{\R} a\bar{\theta}\eta(U|\bar{U})\,dx=\dot{X}Y+J^{\textup{bad}}-J^{\textup{good}},
	\end{equation}
	where
	\begin{align*}
		Y&=-\int_\R a_x\bar{\theta}\eta(U|\bar{U})\,dx+\int_\R a\left(u^S_x(u-\bar{u})+v^S_x \frac{\bar{p}}{\bar{v}}(v-\bar{v})+\frac{R}{\gamma-1}\frac{\theta^S_x}{\bar{\theta}}(\theta-\bar{\theta})+w^S_x(w-\bar{w})\right)\,dx\\
		&\quad - \int_\R a\theta^S_x\left(R\Phi\left(\frac{v}{\bar{v}}\right)+\frac{R}{\gamma-1}\Phi\left(\frac{\theta}{\bar{\theta}}\right)\right)\,dx,
	\end{align*}
	\begin{align*}
		J^{\textup{bad}}&=\int_\R a_x(u-\bar{u})(p-\bar{p})\,dx\\
		&\quad+\int_\R a\left(R\left(\theta^R_t+\theta^C_t-\sigma\theta^S_x\right)\Phi\left(\frac{v}{\bar{v}}\right)-\frac{\bar{p}{\bar{u}_x}}{v\bar{v}}(v-\bar{v})^2\right)\,dx\\
		&\quad +\int_\R a\left(\frac{R}{\gamma-1}\left(\theta^R_t+\theta^C_t-\sigma\theta^S_x\right)\Phi\left(\frac{\theta}{\bar{\theta}}\right)-\frac{\bar{u}_x}{\theta}(\theta-\bar{\theta})(p-\bar{p})+\bar{p}\bar{u}_x\frac{(\theta-\bar{\theta})^2}{\theta\bar{\theta}}\right)\,dx\\
		&\quad + \int_{\R}a\Bigg[\mu(u-\bar{u})\left(\frac{u_x}{v}-\frac{\bar{u}_x}{\bar{v}}\right)+\alpha\frac{\theta-\bar{\theta}}{\theta}\left(\frac{\theta_x}{v}-\frac{\bar{\theta}_x}{\bar{v}}\right)+\sqrt{\kappa}(u-\bar{u})\left(\frac{w_x}{v^{5/2}}-\frac{\bar{w}_x}{\bar{v}^{5/2}}\right)\\
		&\hspace{2cm}-\sqrt{\kappa}\left(\frac{u_x}{v^{5/2}}-\frac{\bar{u}_x}{\bar{v}^{5/2}}\right)(w-\bar{w})\Bigg]_x\,dx\\
		&\quad + \int_\R a\Bigg[-\mu \bar{u}_x\left(\frac{1}{v}-\frac{1}{\bar{v}}\right)(u-\bar{u})_x+\alpha\frac{\theta-\bar{\theta}}{\theta^2}\theta_x\left(\frac{\theta_x}{v}-\frac{\bar{\theta}_x}{\bar{v}}\right)-\alpha\frac{(\theta-\bar{\theta})_x}{\theta}\bar{\theta}_x\left(\frac{1}{v}-\frac{1}{\bar{v}}\right)\\
		&\hspace{2cm}+\mu\frac{\theta-\bar{\theta}}{\theta}\left(\frac{u_x^2}{v}-\frac{\bar{u}_x^2}{\bar{v}}\right)+\sqrt{\kappa}(\bar{u}_x(w-\bar{w})_x-(u-\bar{u})_x\bar{w}_x)\left(\frac{1}{v^{5/2}}-\frac{1}{\bar{v}^{5/2}}\right)\\
		&\hspace{2cm}-\frac{(\theta-\bar{\theta})^2}{\theta\bar{\theta}}\left(\alpha \left(\frac{\bar{\theta}_x}{\bar{v}}\right)_x+\mu \frac{\bar{u}_x^2}{\bar{v}}\right)-Q_1(u-\bar{u})-Q_2\left(\frac{\theta}{\bar{\theta}}-1\right)-Q_3(w-\bar{w})\Bigg]\,dx,
	\end{align*}
	and
	\begin{align*}
		J^{\textup{good}}&=\sigma\int_\R a_x \bar{\theta}\eta(U|\bar{U})\,dx+\int_\R a\left( \frac{\mu}{v}|(u-\bar{u})_x|^2+\frac{\alpha}{v\theta}|(\theta-\bar{\theta})_x|^2\right)\,dx.
	\end{align*}
\end{lemma}
\begin{proof}
	First of all, a direct computation yields
	\begin{align*}
		\frac{d}{dt}\int_\R a(t,x)\bar{\theta}(t,x)\eta(U|\bar{U})\,dx
		&=-\dot{X}(t)\int_\R a_x\bar{\theta}\eta(U|\bar{U})\,dx-\sigma\int_\R a_x\bar{\theta}\eta(U|\bar{U})\,dx \\
		&\quad + \int_\R a\pa_t\left(R\bar{\theta}\Phi\left(\frac{v}{\bar{v}}\right)+\frac{R\bar{\theta}}{\gamma-1}\Phi\left(\frac{\theta}{\bar{\theta}}\right)+\frac{(u-\bar{u})^2}{2}+\frac{(w-\bar{w})^2}{2}\right)\,dx.
	\end{align*}
	Using the equations of $v$ and $\bar{v}$ in \eqref{eq:NSFK-extended} and \eqref{eq:composite_extended}, we have
	
	\begin{align}
		\begin{aligned}\label{est:v_barv}
			\pa_t\left(R\bar{\theta}\Phi\left(\frac{v}{\bar{v}}\right)\right)&=R\bar{\theta}_t \Phi\left(\frac{v}{\bar{v}}\right)+R\bar{\theta}\left(\frac{1}{\bar{v}}-\frac{1}{v}\right)\left(\dot{X}(t)v^S_x+(u-\bar{u})_x\right)\\
			&\quad + R\bar{\theta}\left(\frac{1}{\bar{v}}-\frac{1}{v}\right)\left(1-\frac{v}{\bar{v}}\right)\left(-\dot{X}(t)v^S_x+\bar{u}_x\right)\\ 
			&=R\bar{\theta}_t \Phi\left(\frac{v}{\bar{v}}\right)+\dot{X}(t)v^S_x\frac{\bar{p}}{\bar{v}}(v-\bar{v})-\frac{\bar{p}\bar{u}_x}{v\bar{v}}(v-\bar{v})^2+R\bar{\theta}\left(\frac{1}{\bar{v}}-\frac{1}{v}\right)(u-\bar{u})_x.
		\end{aligned}
	\end{align}
	
	Similarly, we use the equations of $\theta$ and $\bar{\theta}$ to get
	\begin{align}
		\begin{aligned}\label{est:theta_bartheta}
			\pa_t&\left(\frac{R\bar{\theta}}{\gamma-1}\Phi\left(\frac{\theta}{\bar{\theta}}\right)\right)\\
			&=\frac{R\bar{\theta}_t}{\gamma-1}\Phi\left(\frac{\theta}{\bar{\theta}}\right)+\dot{X}(t)\frac{R}{\gamma-1}\frac{\theta^S_x}{\bar{\theta}}(\theta-\bar{\theta})-\frac{R}{v}(\theta-\bar{\theta})(u-\bar{u})_x-\frac{\bar{u}_x}{\theta}(\theta-\bar{\theta})(p-\bar{p})+\bar{p}\bar{u}_x\frac{(\theta-\bar{\theta})^2}{\theta\bar{\theta}}\\
			&\quad + \left(\alpha \frac{\theta-\bar{\theta}}{\theta}\left(\frac{\theta_x}{v}-\frac{\bar{\theta}_x}{\bar{v}}\right)\right)_x-\frac{\alpha}{v\theta}|(\theta-\bar{\theta})_x|^2+\alpha\frac{\theta-\bar{\theta}}{\theta^2}\theta_x\left(\frac{\theta_x}{v}-\frac{\bar{\theta}_x}{\bar{v}}\right)-\alpha\frac{(\theta-\bar{\theta})_x}{\theta}\bar{\theta}_x\left(\frac{1}{v}-\frac{1}{\bar{v}}\right)\\
			&\quad +\mu\frac{\theta-\bar{\theta}}{\theta}\left(\frac{u_x^2}{v}-\frac{\bar{u}_x^2}{\bar{v}}\right)-Q_2\left(\frac{\theta}{\bar{\theta}}-1\right)-\frac{(\theta-\bar{\theta})^2}{\theta\bar{\theta}}\left(\alpha\left(\frac{\bar{\theta}_x}{\bar{v}}\right)_x+\mu\frac{\bar{u}_x^2}{\bar{v}}\right).
		\end{aligned}
	\end{align}
	
	Next, using \eqref{eq:perturbation}$_2$ for $u-\bar{u}$, we get
	
	\begin{align}
		\begin{aligned}\label{est:u-baru}
			\frac{\pa_t(u-\bar{u})^2}{2}&=(u-\bar{u})\dot{X}(t)u^S_x -\left((p-\bar{p})(u-\bar{u})\right)_x + (p-\bar{p})(u-\bar{u})_x\\
			&\quad +\left(\mu\left(\frac{u_x}{v}-\frac{\bar{u}_x}{\bar{v}}\right)(u-\bar{u})\right)_x-\frac{\mu}{v}|(u-\bar{u})_x|^2-\mu(u-\bar{u})_x\bar{u}_x\left(\frac{1}{v}-\frac{1}{\bar{v}}\right)\\
			&\quad  +\sqrt{\kappa}\left((u-\bar{u})\left(\frac{w_x}{v^{5/2}}-\frac{\bar{w}_x}{\bar{v}^{5/2}}\right)\right)_x-\sqrt{\kappa}\underbrace{(u-\bar{u})_x\left(\frac{w_x}{v^{5/2}}-\frac{\bar{w}_x}{\bar{v}^{5/2}}\right)}_{=:\mathcal{A}_1}-Q_1(u-\bar{u}).
		\end{aligned}
	\end{align}
	
	Finally, we use \eqref{eq:perturbation}$_4$ to obtain
	
	\begin{align}
		\begin{aligned}\label{est:w-barw}
			\frac{\pa_t(w-\bar{w})^2}{2}&=(w-\bar{w})\dot{X}(t)w^S_x-\sqrt{\kappa}\left((w-\bar{w})\left(\frac{u_x}{v^{5/2}}-\frac{\bar{u}_x}{\bar{v}^{5/2}}\right)\right)_x \\
			&\quad+ \sqrt{\kappa}\underbrace{(w-\bar{w})_x\left(\frac{u_x}{v^{5/2}}-\frac{\bar{u}_x}{\bar{v}^{5/2}}\right)}_{=:\mathcal{A}_2} -Q_3(w-\bar{w}).
		\end{aligned}
	\end{align}
	
	Now, we observe that
	\begin{align*}
		\mathcal{A}_2-\mathcal{A}_1 &= (w-\bar{w})_x\left(\frac{u_x}{v^{5/2}}-\frac{\bar{u}_x}{\bar{v}^{5/2}}\right)-(u-\bar{u})_x\left(\frac{w_x}{v^{5/2}}-\frac{\bar{w}_x}{\bar{v}^{5/2}}\right)\\
		&=((w-\bar{w})_x\bar{u}_x-(u-\bar{u})_x\bar{w}_x)\left(\frac{1}{v^{5/2}}-\frac{1}{\bar{v}^{5/2}}\right).
	\end{align*}
	
	Thus, adding \eqref{est:u-baru} and \eqref{est:w-barw} yields
	
	\begin{align}
		\begin{aligned}\label{est:uw}
			&\pa_t\left(\frac{(u-\bar{u})^2}{2}+\frac{(w-\bar{w})^2}{2}\right)\\
			&=(u-\bar{u})\dot{X}(t)u^S_x + (w-\bar{w})\dot{X}(t)w^S_x -\left((p-\bar{p})(u-\bar{u})\right)_x + (p-\bar{p})(u-\bar{u})_x\\
			&\quad +\left(\mu\left(\frac{u_x}{v}-\frac{\bar{u}_x}{\bar{v}}\right)(u-\bar{u})\right)_x-\frac{\mu}{v}|(u-\bar{u})_x|^2-\mu(u-\bar{u})_x\bar{u}_x\left(\frac{1}{v}-\frac{1}{\bar{v}}\right)\\
			&\quad  +\sqrt{\kappa}\left((u-\bar{u})\left(\frac{w_x}{v^{5/2}}-\frac{\bar{w}_x}{\bar{v}^{5/2}}\right)\right)_x-Q_1(u-\bar{u})-Q_3(w-\bar{w})\\
			&\quad-\sqrt{\kappa}\left(\left(\frac{u_x}{v^{5/2}}-\frac{\bar{u}_x}{\bar{v}^{5/2}}\right)(w-\bar{w})\right)_x\\
			&\quad+\sqrt{\kappa}(\bar{u}_x(w-\bar{w})_x-(u-\bar{u})_x\bar{w}_x)\left(\frac{1}{v^{5/2}}-\frac{1}{\bar{v}^{5/2}}\right).
		\end{aligned}
	\end{align}
	Therefore, we combine \eqref{est:v_barv}, \eqref{est:theta_bartheta}, and \eqref{est:uw} to get
	
	\begin{align*}
		\begin{aligned}
			\frac{d}{dt}&\int_\R a(t,x)\bar{\theta}(t,x)\eta(U|\bar{U})\,dx\\
			&=-\dot{X}(t)\int_\R a_x\bar{\theta}\eta(U|\bar{U})\,dx -\sigma\int_\R a_x\bar{\theta}\eta(U|\bar{U})\,dx - \int_\R a\frac{\bar{p}\bar{u}_x}{v\bar{v}}(v-\bar{v})^2\,dx\\
			&\quad +\int_{\R}a_x(u-\bar{u})(p-\bar{p})\,dx +\int_\R a \left(-\frac{\bar{u}_x}{\theta}(\theta-\bar{\theta})(p-\bar{p})+\bar{p}\bar{u}_x\frac{(\theta-\bar{\theta})^2}{\theta\bar{\theta}}\right)\,dx\\
			&\quad -\int_\R a\frac{(\theta-\bar{\theta})^2}{\theta\bar{\theta}}\left(\alpha\left(\frac{\bar{\theta}_x}{\bar{v}}\right)_x+\mu\frac{\bar{u}_x^2}{\bar{v}}\right)\,dx\\
			&\quad + \int_{\R}a\Bigg(\mu(u-\bar{u})\left(\frac{u_x}{v}-\frac{\bar{u}_x}{\bar{v}}\right)+\alpha\frac{\theta-\bar{\theta}}{\theta}\left(\frac{\theta_x}{v}-\frac{\bar{\theta}_x}{\bar{v}}\right)+\sqrt{\kappa}(u-\bar{u})\left(\frac{w_x}{v^{5/2}}-\frac{\bar{w}_x}{\bar{v}^{5/2}}\right)\\
			&\hspace{2cm}-\sqrt{\kappa}\left(\frac{u_x}{v^{5/2}}-\frac{\bar{u}_x}{\bar{v}^{5/2}}\right)(w-\bar{w})\Bigg)_x\,dx-\int_\R a\left(\frac{\mu}{v}|(u-\bar{u})_x|^2+\frac{\alpha}{v\theta}|(\theta-\bar{\theta})_x|^2\right)\,dx\\
			&\quad + \int_\R a\Bigg(-\mu \bar{u}_x\left(\frac{1}{v}-\frac{1}{\bar{v}}\right)(u-\bar{u})_x+\alpha\frac{\theta-\bar{\theta}}{\theta^2}\theta_x\left(\frac{\theta_x}{v}-\frac{\bar{\theta}_x}{\bar{v}}\right)-\alpha\frac{(\theta-\bar{\theta})_x}{\theta}\bar{\theta}_x\left(\frac{1}{v}-\frac{1}{\bar{v}}\right)\\
			&\hspace{2cm}+\mu\frac{\theta-\bar{\theta}}{\theta}\left(\frac{u_x^2}{v}-\frac{\bar{u}_x^2}{\bar{v}}\right)+\sqrt{\kappa}\left(\bar{u}_x(w-\bar{w})_x-(u-\bar{u})_x\bar{w}_x\right)\left(\frac{1}{v^{5/2}}-\frac{1}{\bar{v}^{5/2}}\right)\\
			&\hspace{2cm}-Q_1(u-\bar{u})-Q_2\left(\frac{\theta}{\bar{\theta}}-1\right)-Q_3(w-\bar{w})\Bigg)\,dx\\
			&\quad +\dot{X}(t)\int_{\R}a\left(\frac{v^S_x\bar{p}}{\bar{v}}(v-\bar{v})+u^S_x(u-\bar{u})+w^S_x(w-\bar{w})+\frac{R}{\gamma-1}\frac{\theta_x^S}{\bar{\theta}}(\theta-\bar{\theta})\right)\,dx\\
			&\quad +\int_\R a\bar{\theta}_t\left(R\Phi\left(\frac{v}{\bar{v}}\right)+\frac{R}{\gamma-1}\Phi\left(\frac{\theta}{\bar{\theta}}\right)\right)\,dx.
		\end{aligned}
	\end{align*}
	Finally, since
	\[\bar{\theta}_t = \theta_t^R+\theta_t^C+(-\sigma-\dot{X})\theta^S_x,\]
	we obtain the desired estimate.
\end{proof}

\subsection{Decomposition of the bad and good terms}\label{sec:decompose}
Next, we decompose $J^{\textup{bad}}$ and $J^{\textup{good}}$ in \eqref{est:rel_ent} into the leading terms and the remaining terms. In particular, we control the second term in $J^{\textup{bad}}$:
\[\int_\R a\underbrace{\left(R(\theta^R_t+\theta^C_t-\sigma \theta^S_x)\Phi\left(\frac{v}{\bar{v}}\right)-\frac{\bar{p}\bar{u}_x}{v\bar{v}}(v-\bar{v})^2\right)}_{=:J_2}\,dx\]
and the third term in $J^{\textup{bad}}$:
\[\int_\R a\underbrace{\left(\frac{R}{\gamma-1}\left(\theta^R_t+\theta^C_t-\sigma\theta^S_x\right)\Phi\left(\frac{\theta}{\bar{\theta}}\right)-\frac{\bar{u}_x}{\theta}(\theta-\bar{\theta})(p-\bar{p})+\bar{p}\bar{u}_x\frac{(\theta-\bar{\theta})^2}{\theta\bar{\theta}}\right)}_{=:J_3}\,dx.\]
Since the estimates on the terms $J_2$ and $J_3$ can be obtained by using almost the same argument as in \cite[Section 4.4]{KVW25}, except that the rarefaction strength need not be small, we get
\begin{align*}
	J_2&\le -\frac{(\gamma+1)\bar{p}}{2\bar{v}^2}u^R_x |v-\bar{v}|^2+C(\delta_C+\delta_S+\varepsilon_1)|u^R_x||v-\bar{v}|^2\\
	&\quad + C\delta_C(1+t)^{-1}e^{-\frac{C_1|x|^2}{1+t}}|v-\bar{v}|^2\\
	&\quad + \frac{(\gamma+1)p^*\sigma^*}{2(v^*)^2}v^S_x|v-\bar{v}|^2+C(|v^R-v_*|+\delta_C+\delta_S+\varepsilon_1)|v^S_x||v-\bar{v}|^2, 
\end{align*}
and
\begin{align*}
	J_3&\le u^R_x\left(-\frac{R}{2\bar{v}\bar{\theta}}(\theta-\bar{\theta})+\frac{\bar{p}}{\bar{v}\bar{\theta}}(v-\bar{v})\right)(\theta-\bar{\theta})+C(\delta_C+\delta_S+\varepsilon_1)|u^R_x||(v-\bar{v},\theta-\bar{\theta})|^2\\
	&\quad +C\delta_C (1+t)^{-1}e^{-\frac{C_1|x|^2}{1+t}}|(v-\bar{v},\theta-\bar{\theta})|^2\\
	&\quad + \frac{\sigma^* v^S_x}{2v^*\theta^*}(R(\theta-\bar{\theta})-2p^*(v-\bar{v}))(\theta-\bar{\theta})+C(|v^R-v_*|+\delta_C+\delta_S+\varepsilon_1)|v^S_x||(v-\bar{v},\theta-\bar{\theta})|^2.
\end{align*}
Therefore, combining the estimates for $J_2$ and $J_3$ as in \cite{KVW25}, we obtain
\begin{align*}
	\int_\R a(J_2+J_3)\,dx&\le B_1-CG^R\\
	&\quad +C\delta_C(1+t)^{-1}\int_\R a e^{-\frac{C_1|x|^2}{1+t}}|(v-\bar{v},\theta-\bar{\theta})|^2\,dx\\
	&\quad +C(\delta_C+\delta_S+\varepsilon_1)\int_\R a(|u^R_x|+|v^S_x|)|(v-\bar{v},\theta-\bar{\theta})|^2\,dx\\
	&\quad +C\int_{\R} a|v^R-v_*||v^S_x||(v-\bar{v},\theta-\bar{\theta})|^2\,dx,
\end{align*}
where
\begin{align*}
	&B_1:=\int_\R a v^S_x\left(\frac{(\gamma+1)p^*\sigma^*}{2(v^*)^2}|v-\bar{v}|^2+\frac{R\sigma^*}{2v^*\theta^*}|\theta-\bar{\theta}|^2-\frac{p^*\sigma^*}{v^*\theta^*}(v-\bar{v})(\theta-\bar{\theta})\right)\,dx,
\end{align*}
and $G^R$ is defined in \eqref{GSGR}.

The difference from the corresponding estimate in \cite{KVW25} is that we cannot use the smallness of $|v^R-v_*|$. Therefore, the weighted relative entropy estimate \eqref{est:rel_ent} can be written as

\begin{equation}\label{est:rel_ent_2}
	\frac{d}{dt}\int_\R a\bar{\theta} \eta(U|\bar{U})\,dx\le \dot{X}Y +\sum_{i=1}^5 B_i + \sum_{i=1}^3S_i-CG^R-G-D,
\end{equation}
where
\begin{align*}
	&B_2:=\int_\R a_x(u-\bar{u})(p-\bar{p})\,dx,\\
	&B_3:=-\int_{\R}a_x\Bigg(\mu(u-\bar{u})\left(\frac{u_x}{v}-\frac{\bar{u}_x}{\bar{v}}\right)+\alpha\frac{\theta-\bar{\theta}}{\theta}\left(\frac{\theta_x}{v}-\frac{\bar{\theta}_x}{\bar{v}}\right)\\
	&\hspace{3cm}+\sqrt{\kappa}(u-\bar{u})\left(\frac{w_x}{v^{5/2}}-\frac{\bar{w}_x}{\bar{v}^{5/2}}\right)-\sqrt{\kappa}\left(\frac{u_x}{v^{5/2}}-\frac{\bar{u}_x}{\bar{v}^{5/2}}\right)(w-\bar{w})\Bigg)\,dx,\\
	&B_4:=\int_\R a\Bigg[-\mu \bar{u}_x\left(\frac{1}{v}-\frac{1}{\bar{v}}\right)(u-\bar{u})_x+\alpha\frac{\theta-\bar{\theta}}{\theta^2}\theta_x\left(\frac{\theta_x}{v}-\frac{\bar{\theta}_x}{\bar{v}}\right)-\alpha\frac{(\theta-\bar{\theta})_x}{\theta}\bar{\theta}_x\left(\frac{1}{v}-\frac{1}{\bar{v}}\right)\\
	&\hspace{2cm}+\mu\frac{\theta-\bar{\theta}}{\theta}\left(\frac{u_x^2}{v}-\frac{\bar{u}_x^2}{\bar{v}}\right)-\frac{(\theta-\bar{\theta})^2}{\theta\bar{\theta}}\left(\alpha\left(\frac{\bar{\theta}_x}{\bar{v}}\right)_x+\mu \frac{\bar{u}_x^2}{\bar{v}}\right)\\
	&\hspace{2cm}+\sqrt{\kappa}\left(\bar{u}_x(w-\bar{w})_x-(u-\bar{u})_x\bar{w}_x\right)\left(\frac{1}{v^{5/2}}-\frac{1}{\bar{v}^{5/2}}\right)\Bigg]\,dx,\\
	&B_5:=C\delta_C(1+t)^{-1}\int_\R a e^{-\frac{C_1|x|^2}{1+t}}|(v-\bar{v},\theta-\bar{\theta})|^2\,dx \\
	&\hspace{1cm}+C(\delta_C+\delta_S+\varepsilon_1)\int_\R a(|u^R_x|+|v^S_x|)|(v-\bar{v},\theta-\bar{\theta})|^2\,dx\\
	&\hspace{1cm}+C\int_{\R} a|v^R-v_*||v^S_x||(v-\bar{v},\theta-\bar{\theta})|^2\,dx,\\
	&S_1:=-\int_\R aQ_1(u-\bar{u})\,dx,\quad S_2:=-\int_\R aQ_2\left(\frac{\theta}{\bar{\theta}}-1\right)\,dx,\quad S_3:=-\int_{\R} aQ_3(w-\bar{w})\,dx,
\end{align*}
and
\begin{equation}\label{GD}
	G:=\sigma\int_\R a_x\bar{\theta}\eta(U|\bar{U})\,dx,\quad D:=\int_\R a\left(\frac{\mu}{v}|(u-\bar{u})_x|^2+\frac{\alpha}{v\theta}|(\theta-\bar{\theta})_x|^2\right)\,dx.
\end{equation}
We also decompose $Y$ as $Y=\sum_{i=1}^7 Y_i$, where
\begin{align}
	\begin{aligned}\label{def:Yi}
	&Y_1 :=\int_\R a u^S_x(u-\bar{u})\,dx,\quad Y_2 :=\int_\R a \frac{\bar{p}v^S_x}{\bar{v}}(v-\bar{v})\,dx,\\
	&Y_3 :=\int_\R a\frac{R}{\gamma-1}\frac{\theta^S_x}{\bar{\theta}}(\theta-\bar{\theta})\,dx,\quad Y_4:=-\int_\R a R\theta^S_x \Phi\left(\frac{v}{\bar{v}}\right)\,dx,\\
	&Y_5:=-\int_\R a\frac{R\theta^S_x}{\gamma-1}\Phi\left(\frac{\theta}{\bar{\theta}}\right)\,dx,\quad Y_6:=-\int_\R a_x\bar{\theta}\eta(U|\bar{U})\,dx,\quad Y_7:=\int_\R aw^S_x(w-\bar{w})\,dx.
	\end{aligned}
\end{align}
Now, from the definition of the shift in \eqref{def:shift}, we can express \eqref{def:shift} as
\[\dot{X}(t)=-\frac{M}{\delta_S}(Y_1+Y_2+Y_3),\quad X(0)=0,\]
where the constant $M$ was chosen as in \eqref{choice_M}. This yields
\[\dot{X}(t)Y=-\frac{\delta_S}{M}|\dot{X}(t)|^2+\dot{X}(t)\sum_{i=4}^7 Y_i.\]
Therefore, \eqref{est:rel_ent_2} can be further decomposed as
\begin{align*}
	\frac{d}{dt}\int_{\R} a\bar{\theta}\eta(U|\bar{U})\,dx&\le -\frac{\delta_S}{2M}|\dot{X}|^2+B_1+B_2-G-\frac{3}{4}D\\
	&\quad -\frac{\delta_S}{2M}|\dot{X}|^2+\dot{X}\sum_{i=4}^7Y_i+\sum_{i=3}^5B_i+\sum_{i=1}^3S_i-\frac{D}{4}-CG^R \\
	&\le \underbrace{-\frac{\delta_S}{2M}|\dot{X}|^2+B_1+B_2-G-\frac{3}{4}D}_{\mathcal{R}_1}\\
	&\quad \underbrace{-\frac{\delta_S}{4M}|\dot{X}|^2+\frac{C}{\delta_S}\sum_{i=4}^7|Y_i|^2+\sum_{i=3}^5B_i+\sum_{i=1}^3S_i-CG^R-\frac{D}{4}}_{\mathcal{R}_2},
\end{align*}
where we used Young's inequality in the last inequality. Note that the term $\mathcal{R}_1$ represents the leading terms that should be carefully estimated, while $\mathcal{R}_2$ represents the remaining terms that can be roughly bounded. 

\subsection{Estimates for the leading terms}
We first estimate the leading terms in $\mathcal{R}_1$.
\begin{lemma}\label{lem:leading}
	There exists a constant $C>0$ such that
	\begin{align*}
    \mathcal{R}_1&\le -CG^S-G_3+C\delta_S^{4/3}\left((\delta_R^{4/3}+\delta_C^{4/3})e^{-C\delta_St}+\frac{\delta_R^{4/3}}{(1+\e^2_Rt^2)^{q/3}}\right)\\
    &\quad +C\varepsilon_1^6\|(v-\bar{v},u-\bar{u},\theta-\bar{\theta})_x\|_{L^2}^2,
\end{align*}
	where
	\begin{equation}\label{GSG3}
		G^S:=\int_{\R} |v^S_x||(v-\bar{v},u-\bar{u},\theta-\bar{\theta})|^2\,dx,\quad G_3:=\frac{\sigma}{2}\int_{\R} a_x|w-\bar{w}|^2\,d x.
	\end{equation}
\end{lemma}

\begin{proof}
We mainly follow the argument of \cite[Lemma~4.4]{KVW25}, but pay attention to where the smallness of rarefaction strength cannot be used.

\medskip
\noindent $\bullet$ Step 1 (Estimate of $B_2-G$): We first handle the term $B_2-G$, which can be written as
\begin{align*}
	B_2-G&=\int_{\R} a_x\left[(u-\bar{u})(p-\bar{p})-\sigma\bar{\theta}\eta(U|\bar{U})\right]\,dx\\
	&=\int_{\R} a_x\biggl[(u-\bar{u})(p-\bar{p})-\sigma\left(R\bar{\theta}\Phi\left(\frac{v}{\bar{v}}\right)+\frac{R\bar{\theta}}{\gamma-1}\Phi\Bigl(\frac{\theta}{\bar{\theta}}\Bigl)+\frac{(u-\bar{u})^2}{2}+\frac{(w-\bar{w})^2}{2}\right)\biggr]\,dx.
\end{align*}
Following the estimates in \cite[Lemma 4.4]{KVW25}, there exists a constant $C>0$ such that
\begin{align*}
	&(u-\bar{u})(p-\bar{p})-\sigma\Bigl[R\bar{\theta}\Phi\left(\frac{v}{\bar{v}}\right)+\frac{R\bar{\theta}}{\gamma-1}\Phi\left(\frac{\theta}{\bar{\theta}}\right)+\frac{(u-\bar{u})^2}{2}\Bigr]\\
	&\le -\frac{R\sigma^*\theta^*}{2(v^*)^2}\left|(v-\bar{v})+\frac{u-\bar{u}}{\sigma^*}\right|^2-\frac{R\sigma^*}{2(\gamma-1)\theta^*}\left|(\theta-\bar{\theta})-\frac{(\gamma-1)\theta^*}{v^*\sigma^*}(u-\bar{u})\right|^2\\
	&\quad +C\left(|v-\bar{v}|+\delta_S+|(v^R-v_*,\theta^R-\theta_*)|+|(v^C-v^*,\theta^C-\theta^*)|\right)|(v-\bar{v},u-\bar{u},\theta-\bar{\theta})|^2,
\end{align*}
Thus, $B_2-G$ can be bounded as
\[B_2-G\le B_\textup{new}-G_1-G_2-G_3,\]
where the good terms are given by
\begin{align*}
	G_1&:=\frac{R\sigma^*\theta^*}{2(v^*)^2}\int_{\R}a_x\left|(v-\bar{v})+\frac{u-\bar{u}}{\sigma^*}\right|^2\,dx,\\
	G_2&:=\frac{R\sigma^*}{2(\gamma-1)\theta^*}\int_{\R} a_x\left|(\theta-\bar{\theta})-\frac{(\gamma-1)\theta^*}{v^*\sigma^*}(u-\bar{u})\right|^2\,dx,\quad G_3:=\frac{\sigma}{2}\int_{\R} a_x|w-\bar{w}|^2\,dx,
    \end{align*}
    and the new bad term is
    \begin{align*}
	B_{\textup{new}}&:=C\delta_S\int_{\R} a_x|(v-\bar{v},u-\bar{u},\theta-\bar{\theta})|^2\,dx +C\int_{\R} a_x|(v-\bar{v},u-\bar{u},\theta-\bar{\theta})|^3\,dx\\
	&\quad + C\int_{\R} a_x\left(|(v^R-v_*,\theta^R-\theta_*)|+|(v^C-v^*,\theta^C-\theta^*)|\right)|(v-\bar{v},u-\bar{u},\theta-\bar{\theta})|^2\,dx.
\end{align*}

Therefore, we use the good term $G$ to absorb one of the primary bad terms $B_2$, and get new good terms $G_1$, $G_2$, and $G_3$, together with the new bad term $B_{\textup{new}}$. We continue to follow the estimates in \cite[Lemma 4.4]{KVW25} to bound $B_{\textup{new}}$, but we need to be careful since we do not assume the smallness of the rarefaction strength $\delta_R$. The first and second terms in $B_{\textup{new}}$ can be estimated as in \cite{KVW25}, since the smallness of rarefaction strength does not play any role here. Thus, we have
\[C\delta_S\int_{\R} a_x|(v-\bar{v},u-\bar{u},\theta-\bar{\theta})|^2\,dx \le C\sqrt{\delta_S}G^S,\]
and
\begin{align*}
	&\int_{\R} a_x|(v-\bar{v},u-\bar{u},\theta-\bar{\theta})|^3\,dx\le C\varepsilon_1(G_1+G_2+D+G^S).
\end{align*}
Now, we estimate the third term of $B_{\textup{new}}$. We use the interpolation inequality and Lemma \ref{lem: interaction estimates} to obtain
\begin{align*}
	\int_{\R} &|a_x||(v^R-v_*,\theta^R-\theta_*, v^C-v^*,\theta^C-\theta^*)||(v-\bar{v},u-\bar{u},\theta-\bar{\theta})|^2\,dx\\
	&\le\frac{C}{\sqrt{\delta_S}}\int_{\R}|v^S_x||(v^R-v_*,\theta^R-\theta_*, v^C-v^*,\theta^C-\theta^*)| |(v-\bar{v},u-\bar{u},\theta-\bar{\theta})|^2\,dx\\
	&\le \frac{C}{\sqrt{\delta_S}}\|R_3\|_{L^2}\|(v-\bar{v},u-\bar{u},\theta-\bar{\theta})\|_{L^4}^2\\
	&\le \frac{C}{\sqrt{\delta_S}}\|R_3\|_{L^2}\|(v-\bar{v},u-\bar{u},\theta-\bar{\theta})\|_{L^2}^{3/2}\|(v-\bar{v},u-\bar{u},\theta-\bar{\theta})_x\|_{L^2}^{1/2}\\	
	&\le C\left(\bigl( \delta_R +\delta_C\bigr)\delta_S e^{-C \delta_S t}+\delta_S\frac{\delta_R}{(1+\e_R^2t^2)^{q/4}}\right)^{4/3}+C\varepsilon_1^6\|(v-\bar{v},u-\bar{u},\theta-\bar{\theta})_x\|_{L^2}^2.
\end{align*}
Thus, we have the following estimate on $B_{\textup{new}}$:
\begin{align*}
	B_{\textup{new}}&\le C(\sqrt{\delta_S}+\varepsilon_1)(G_1+G_2+D+G^S)+C\delta_S^{4/3}\left((\delta_R^{4/3}+\delta_C^{4/3})e^{-C\delta_St}+\frac{\delta_R^{4/3}}{(1+\e^2_Rt^2)^{q/3}}\right)\\
	&\quad + C\varepsilon_1^6\|(v-\bar{v},u-\bar{u},\theta-\bar{\theta})_x\|_{L^2}^2.
\end{align*}

\noindent $\bullet$ Step 2 (Estimate of $-\frac{\delta_S}{2M}|\dot X|^2$ and $B_1$):
For each fixed $t\ge 0$, let
\[y=y(x):=\frac{v^S(x-\sigma t-X(t))-v^*}{\delta_S},\quad \mbox{therefore}\quad \delta_S\,dy = v^S_x\,dx=-\frac{1}{\sigma}u^S_x\,dx.\]
Note that $y$ increases from 0 to 1 as $x$ increases from $-\infty$ to $+\infty$. Then, we define $f:[0,1]\to\R$ as
\[f:=\left(u(t,\cdot)-\bar{u}(t,\cdot)\right)\circ y^{-1},\quad \mbox{i.e.}\quad f(y) = u(t,x)-\bar{u}(t,x).\]

We estimate each of $Y_i$ for $i=1,2,3$ separately. First of all, it is straightforward to obtain
\[Y_1 = \int_{\R} au^S_x(u-\bar{u})\,dx = -\delta_S\sigma\int_0^1 a f\,dy.\]
Since $|a-1|\le \sqrt{\delta_S}$ and $|\sigma-\sigma^*|\le C\delta_S$, we get
\[\left|Y_1+\delta_S\sigma^*\int_0^1 f\,dy\right|\le C\delta_S^{3/2}\int_0^1 |f|\,dy.\]
Next, we split $Y_2$ as
\begin{align*}
	Y_2 = \int_{\R} a\frac{v^S_x\bar{p}}{\bar{v}}(v-\bar{v})\,dx =-\int_{\R} a\frac{v^S_x\bar{p}}{\bar{v}\sigma^*}(u-\bar{u})\,dx+\int_{\R} a\frac{v^S_x\bar{p}}{\bar{v}}\left((v-\bar{v})+\frac{u-\bar{u}}{\sigma^*}\right)\,dx.
\end{align*}
Using
\begin{align*}
	|\bar{p}-p^*|\le C(|\bar{v}-v^*|+|\bar{\theta}-\theta^*|)\le C(|(v^R-v_*,\theta^R-\theta_*)|+\delta_C+\delta_S),
\end{align*}
we get
\begin{align*}
	\left|Y_2+\frac{p^*\delta_S}{v^*\sigma^*}\int_0^1 f\,dy\right|&\le C\delta_S(\sqrt{\delta_S}+\delta_0)\int_0^1 |f|\,dy+ C\int_{\R} |v^S_x||(v^R-v_*,\theta^R-\theta_*)||(u-\bar{u})|\,dx\\
	&\quad + C\sqrt{\delta_S}\int_{\R} a_x \left|(v-\bar{v})+\frac{u-\bar{u}}{\sigma^*}\right|\,dx.
\end{align*}
Similar to $Y_2$, we estimate $Y_3$ as
\begin{align*}
	\left|Y_3 + \frac{(\gamma-1)p^*\delta_S}{v^*\sigma^*}\int_0^1 f\,dy\right|&\le C\delta_S(\sqrt{\delta_S}+\delta_0)\int_0^1 |f|\,dy+ C\int_{\R} |\theta^S_x||\theta^R-\theta_*||(u-\bar{u})|\,dx\\
	&\quad + C\sqrt{\delta_S}\int_{\R} a_x\left|(\theta-\bar{\theta})-\frac{(\gamma-1)\theta^*}{v^*\sigma^*}(u-\bar{u})\right|\,dx.
\end{align*}
Finally, we use Lemma \ref{lem: interaction estimates} to obtain
\begin{align}
	\begin{aligned}\label{est:Yi_remain}
	&\int_{\R} (|v^S_x|+|\theta^S_x|)|(v^R-v_*,\theta^R-\theta_*)||u-\bar{u}|\,dx
    \\
    &\quad \le \|R_3\|_{L^2}\|u-\bar{u}\|_{L^2} \le C\varepsilon_1\delta_S^{3/2} \left(\bigl( \delta_R +\delta_C\bigr) e^{-C \delta_S t}+\frac{\delta_R}{(1+\e_R^2t^2)^{q/4}}\right).
	\end{aligned}
\end{align}
Since $\dot{X} = -\frac{M}{\delta_S}(Y_1+Y_2+Y_3)$, we collect the estimates of $Y_i$ for $i=1,2,3$ and \eqref{est:Yi_remain} to get
\begin{align*}
	\left|\dot{X}-2\sigma^*M\int_0^1f\,dy\right|&\le C(\sqrt{\delta_S}+\delta_0)	\int_0^1|f|\,dy+\frac{C}{\sqrt{\delta_S}}\int_{\R} a_x\left|(v-\bar{v})+\frac{u-\bar{u}}{\sigma^*}\right|\,dx\\
	&\quad +\frac{C}{\sqrt{\delta_S}}\int_{\R}a_x\left|(\theta-\bar{\theta})-\frac{(\gamma-1)\theta^*}{v^*\sigma^*}(u-\bar{u})\right|\,dx \\
	&\quad+ C\varepsilon_1\delta_S^{1/2} \left(\bigl( \delta_R +\delta_C\bigr) e^{-C \delta_S t}+\frac{\delta_R}{(1+\e_R^2t^2)^{q/4}}\right).
\end{align*}
Using inequality $\frac{p^2}{2}-q^2 \le (q-p)^2$ with $p=2\sigma^*M \int_0^1 f \, dy$ and $q=\dot{X}$, and Cauchy--Schwarz inequality, we obtain
\begin{align*}
	2(\sigma^*)^2M^2\left(\int_0^1f\,dy\right)^2-|\dot{X}|^2&\le C(\sqrt{\delta_S}+\delta_0)^2\int_0^1|f|^2\,dy+\frac{C}{\sqrt{\delta_S}}(G_1+G_2)\\
	&\quad+C\varepsilon_1^2\delta_S\left(\bigl( \delta^2_R +\delta^2_C\bigr) e^{-C \delta_S t}+\frac{\delta_R^2}{(1+\e_R^2t^2)^{q/2}}\right).
\end{align*}
Multiplying by $\tfrac{\delta_S}{2M}$ and rearranging the terms, we get
\begin{align*}
    -\frac{\delta_S}{2M}|\dot{X}|^2&\le -(\sigma^*)^2 M\delta_S\left(\int_0^1 f\,dy\right)^2+C\delta_S(\sqrt{\delta_S}+\delta_0)^2\int_0^1f^2\,dy\\
    &\quad+C\sqrt{\delta_S}(G_1+G_2)+C\varepsilon_1^2\delta_S^2\left(\bigl( \delta^2_R +\delta^2_C\bigr) e^{-C \delta_S t}+\frac{\delta_R^2}{(1+\e_R^2t^2)^{q/2}}\right).
\end{align*}

Finally, since the estimate of $B_1$ in \cite[Lemma 4.4]{KVW25} does not use the smallness of the rarefaction strength, we can directly adopt its estimate as
\[B_1\le A^*\left(1+\sqrt{\delta_S}+\delta_S^{1/4}\right)\delta_S\int_0^1 f^2\,dy+C\delta_S^{1/4}(G_1+G_2),\]
where $A^*$ is the constant defined in Lemma~\ref{lem:shock-property}.
\medskip

\noindent $\bullet$ Step 3 (Estimate of $D$):
Next, we split $D$ as
\begin{equation} \label{Du1,Dtheta1-def}
D=\int_{\R} a\frac{\mu}{v}|(u-\bar u)_x|^2\,dx
  +\int_{\R} a\frac{\alpha}{v\theta}|(\theta-\bar\theta)_x|^2\,dx
=:D_{u_1}+D_{\theta_1} .
\end{equation}
We first estimate $D_{u_1}$. Since $a\ge 1$, we have
\begin{equation*}
    \begin{aligned}
D_{u_1}
&\ge
\mu\int_{\R}\Big(\frac{1}{v^S}
+\Big(\frac{1}{v}-\frac{1}{\bar v}\Big)\Big)|(u-\bar u)_x|^2\,dx
+\mu\int_{\R}\Big(\frac{1}{\bar v}-\frac{1}{v^S}\Big)|(u-\bar u)_x|^2\,dx .
\end{aligned}
\end{equation*}
Using $v^R-v_* \le 0$ and $|v^C-v^*| \le C\delta_C$,
\begin{equation*} 
    \frac{1}{\bar v}-\frac{1}{v^S}
=\frac{v^S-\bar v}{\bar vv^S}
=-\frac{(v^R-v_*)+(v^C-v^*)}{\bar vv^S}\ge -C\delta_C.
\end{equation*}
With $\|v-\bar v\|_{L^\infty}\le C\varepsilon_1$ and the change of variables $x\mapsto y$, we obtain
\begin{equation}\label{eq:D1-lb-new}
D_{u_1}\ge
\bigl(1-C(\delta_C+\varepsilon_1)\bigr)
\int_0^1 |f_y|^2\,\Big(\frac{\mu}{v^S}\frac{dy}{dx}\Big)\,dy .
\end{equation}
Similarly, we estimate $D_{\theta_1}$ as
\[
D_{\theta_1}\ge\alpha\int_{\R}\Bigl(\frac{1}{v^S\theta^S}
+\Bigl(\frac{1}{v\theta}-\frac{1}{\bar v\,\bar\theta}\Bigr)\Bigr)\,|(\theta-\bar\theta)_x|^2\,dx
+\alpha\int_{\R}\Bigl(\frac{1}{\bar v\,\bar\theta}-\frac{1}{v^S\theta^S}\Bigr)\,|(\theta-\bar\theta)_x|^2\,dx.
\]
Note that
\begin{align*}
    \frac{1}{\bar{v} \bar{\theta}} - \frac{1}{v^S \theta^S} &=  -\frac{(v^R+((v^C-v_*)+(v^S-v^*)))(\theta^R+((\theta^C-\theta_*)+(\theta^S-\theta^*)))}-v^S\theta^S{\bar{v} \bar{\theta} v^S \theta^S}\\
    &= -\frac{v^R\theta^R-v_*\theta_*}{\bar{v} \bar{\theta} v^S \theta^S} -\frac{v^R((\theta^C-\theta_*)+(\theta^S-\theta^*))+\theta^R((v^C-v_*)+(v^S-v^*))}{\bar{v} \bar{\theta} v^S \theta^S}\\
    &\quad -\frac{(v^C-v_*)(\theta^C-\theta_*)+(v^S-v^*)(\theta^C-\theta_*)+(\theta^S-\theta^*)(v^C-v_*)}{\bar{v} \bar{\theta} v^S \theta^S} \\
    &\quad -\frac{-\theta^*(v^S-v^*)-v^*(\theta^S-\theta^*)-(v^*\theta^*-v_*\theta_*)}{\bar{v} \bar{\theta} v^S \theta^S}.
\end{align*}
Since the approximate rarefaction wave satisfies \eqref{smooth rarefaction}$_3$, which implies
\begin{equation}\label{thetaRvR}
	\theta^R (v^R)^{\gamma-1} = \theta_*(v_*)^{\gamma-1},\quad \mbox{that is}\quad v^R\theta^R=C (v^R)^{2-\gamma},
\end{equation}
the function $v^R\theta^R=C(v^R)^{2-\gamma}$ is nondecreasing in $v^R$ for $1<\gamma\le2$, and $v^R\le v_*$ gives $v^R\theta^R-v_*\theta_*\le0$. With $|v^C-v^*|+|\theta^C-\theta^*|\le C\delta_C$ and $|v^S-v^*|+|\theta^S-\theta^*|\le C\delta_S$, we obtain
\begin{equation*}
\frac{1}{\bar v\bar\theta}-\frac{1}{v^S\theta^S}\ge-C(\delta_C+\delta_S).
\end{equation*}
Then using $\|(\theta-\bar\theta,v-\bar v)\|_{L^\infty}\le C\varepsilon_1$ and the change of variables $x\mapsto y$, we get
\begin{equation}\label{eq:D2-lb-new}
D_{\theta_1}\ge\bigl(1-C(\delta_C+\delta_S+\varepsilon_1)\bigr)
\int_0^1 \bigl|(\theta-\bar\theta)_y\bigr|^2
\frac{\alpha}{\mu\theta^*}\Bigl(\frac{\mu}{v^S}\frac{dy}{dx}\Bigr)\,dy.
\end{equation}

To obtain a sharp estimate on the lower bounds in \eqref{eq:D1-lb-new} and \eqref{eq:D2-lb-new}, we invoke the following estimate for the Jacobian:
\begin{equation}\label{Diffusion-Jacobian}
\left|
\frac{1}{y(1-y)}\frac{\mu}{v^S}\frac{dy}{dx}
-\delta_SA^*
\frac{\mu R\gamma}{\mu R\gamma+\alpha(\gamma-1)^2}
\right|\le C\delta_S^{2}.
\end{equation}
To prove \eqref{Diffusion-Jacobian}, we use the shock equation \eqref{eq:viscous_shock}$_{1,2}$ to derive
\[
\mu\frac{v_x^S}{v^S}
-\frac{\kappa}{\sigma}\!\left(-\frac{v_{xx}^S}{(v^S)^5}
+\frac{5(v_x^S)^2}{2(v^S)^6}\right)
=-\sigma(v^S-v^*)-\frac{p^S-p^*}{\sigma}.
\]
Since $w^S=-\sqrt{\kappa}v_x^S/(v^S)^{5/2}$ and $y=(v^S-v^*)/\delta_S$, we rewrite the above equation as
\begin{equation}\label{dydx-1}
\delta_S\frac{\mu}{v^S}\frac{dy}{dx}
=+\frac{\sqrt{\kappa}}{\sigma}\frac{w_x^S}{(v^S)^{5/2}}
-\frac{1}{\sigma}\Bigl(\sigma^2(v^S-v^*)+p^S-p^*\Bigr).
\end{equation}
Using $|v_{xx}^S|\le C\delta_S|v_x^S|$ and $|v_x^S|\le C\delta_S^2$,
\begin{equation}\label{w_x-dydx}
\left|\frac{\sqrt{\kappa}}{\sigma}\frac{w_x^S}{(v^S)^{5/2}}\right|
\le C\delta_S^2\frac{\mu}{v^S}\frac{dy}{dx}.
\end{equation}
We substitute the estimate \eqref{w_x-dydx} into \eqref{dydx-1} to get
\[
\delta_S(1-C\delta_S)\frac{\mu}{v^S}\frac{dy}{dx}
\le
-\frac{1}{\sigma}\Bigl(\sigma^2(v^S-v^*)+p^S-p^*\Bigr)
\le
\delta_S(1+C\delta_S)\frac{\mu}{v^S}\frac{dy}{dx}.
\]
Since $y=\dfrac{v^S-v^*}{\delta_S}$, $1-y=\dfrac{v_+-v^S}{\delta_S}$, and
\[
-\frac{1}{\sigma}\Bigl(\sigma^2(v^S-v^*)+p^S-p^*\Bigr)
=-\frac{1}{\sigma(v_+-v^*)}\Bigl((p^S-p_+)(v^S-v^*)+(v_+-v^S)(p^S-p^*)\Bigr),
\]
we deduce that
\begin{equation}\label{eq:two-sided}
\frac{1-C\delta_S}{y(1-y)}\frac{\mu}{v^S}\frac{dy}{dx}
\le
\frac{1}{\sigma}\!\left(\frac{p^S-p_+}{v^S-v_+}-\frac{p^S-p^*}{v^S-v^*}\right)
\le
\frac{1+C\delta_S}{y(1-y)}\frac{\mu}{v^S}\frac{dy}{dx}.
\end{equation}
On the other hand, {dividing \eqref{secant-est} by $\sigma^*$ gives}
\begin{equation}\label{eq:KV-sigma-star}
\left|
\frac{1}{\sigma^*}\!\left(\frac{p^S-p_+}{v^S-v_+}-\frac{p^S-p^*}{v^S-v^*}\right)
-\delta_SA^*
\frac{\mu R\gamma}{\mu R\gamma+\alpha(\gamma-1)^2}
\right|
\le C\delta_S^2.
\end{equation}
Using $|\sigma-\sigma^*|\le C\delta_S$, and $|(v^S-v^*,\theta^S-\theta^*)|,|(v^S-v_+,\theta^S-\theta_+)|\le C\delta_S$, we have
\[
\left|
\frac{1}{\sigma}\!\left(\frac{p^S-p_+}{v^S-v_+}-\frac{p^S-p^*}{v^S-v^*}\right)
-\frac{1}{\sigma^*}\!\left(\frac{p^S-p_+}{v^S-v_+}-\frac{p^S-p^*}{v^S-v^*}\right)
\right|
\le C\delta_S^2.
\]
Hence \eqref{Diffusion-Jacobian} follows from \eqref{eq:two-sided} and \eqref{eq:KV-sigma-star}.

Now, the lower bounds \eqref{eq:D1-lb-new} and \eqref{eq:D2-lb-new}, and the Jacobian estimate \eqref{Diffusion-Jacobian} give
\[
\begin{aligned}
D=D_{u_1}+D_{\theta_1}
&\ge\delta_S\,A^*\,
\frac{\mu R\gamma}{\mu R\gamma+\alpha(\gamma-1)^2}
\Bigl(1-C(\delta_C+\delta_S+\varepsilon_1)\Bigr)
\\
&\qquad\times\Biggl(
\int_0^1 y(1-y)\,|f_y|^2\,dy
+\frac{\alpha}{\mu\theta^*}\int_0^1 y(1-y)\,|(\theta-\bar\theta)_y|^2\,dy
\Biggr).
\end{aligned}
\]
Then, we follow the same argument as in \cite{KVW25}, using the decomposition
\[
    \theta-\bar\theta=\Bigl[(\theta-\bar\theta)-\frac{(\gamma-1)\theta^*}{v^*\sigma^*}(u-\bar u)\Bigr]+\frac{(\gamma-1)\theta^*}{v^*\sigma^*}(u-\bar u),
\]
and the weighted Poincar\'e inequality \cite[Lemma~2.9]{KV21-JEMS} for $f$ and $\theta-\bar\theta$, to derive
\begin{align*}
    D &\ge 2A^*\frac{\mu R\gamma}{\mu R\gamma+\alpha(\gamma-1)^2} \delta_S (1-C(\delta_C+\delta_S +\varepsilon_1)) \\
    &\quad \times \Bigg[  \int_0^1 f^2 dy - \left(\int_0^1 f dy\right)^2+\frac{\alpha}{\mu \theta^*}\int_0^1 (\theta-\bar{\theta})^2dy-\frac{\alpha}{\mu \theta^*}\left(\int_0^1 (\theta-\bar{\theta})dy\right)^2 \Bigg]\\
     &\ge2A^*\left(1-C(\delta_0+\varepsilon_1)-\delta_S^{1/4}\right)\delta_S\int_0^1 f^2\,dy \\
     &\quad-2A^*\delta_S\frac{\mu R \gamma}{\mu R \gamma+\alpha(\gamma-1)^2}\left(1+\frac{2\alpha(\gamma-1)^2}{\mu R\gamma}\right)\left(\int_0^1 f\,dy\right)^2-C\delta_S^{1/4}G_2.
\end{align*}

\noindent $\bullet$ Step 4 (Conclusion): Combining all the estimates in Step 1--Step 3, and using the smallness of $\delta_S$, $\delta_C$, and $\varepsilon_1$, we estimate the leading terms as

\begin{align*}
	\mathcal{R}_1&\le C(\sqrt{\delta_S}+\varepsilon_1)G^S-\frac{A^*\delta_S}{4}\int_0^1 f^2\,dy -\frac{1}{2}(G_1+G_2)-G_3\\
	&\quad -(\sigma^*)^2M\delta_S\left(\int_0^1 f\,dy\right)^2+\frac{3}{2}A^*\delta_S\frac{\mu R \gamma}{\mu R \gamma+\alpha(\gamma-1)^2}\left(1+\frac{2\alpha(\gamma-1)^2}{\mu R\gamma}\right)\left(\int_0^1 f\,dy\right)^2\\
    &\quad +C\delta_S^{4/3}\left((\delta_R^{4/3}+\delta_C^{4/3})e^{-C\delta_St}+\frac{\delta_R^{4/3}}{(1+\e^2_Rt^2)^{q/3}}\right) +C\varepsilon_1^6\|(v-\bar{v},u-\bar{u},\theta-\bar{\theta})_x\|_{L^2}^2.
\end{align*}

Now, the choice of the constant $M$ as in \eqref{choice_M} implies that
\begin{align*}
    \mathcal{R}_1&\le C(\sqrt{\delta_S}+\varepsilon_1)G^S -\frac{A^*\delta_S}{4}\int_0^1 f^2\,dy -\frac{1}{2}(G_1+G_2)-G_3 \\
    &\quad +C\delta_S^{4/3}\left((\delta_R^{4/3}+\delta_C^{4/3})e^{-C\delta_St}+\frac{\delta_R^{4/3}}{(1+\e^2_Rt^2)^{q/3}}\right)+C\varepsilon_1^6\|(v-\bar{v},u-\bar{u},\theta-\bar{\theta})_x\|_{L^2}^2.
\end{align*}
Finally, using Young's inequality, $a_x=\frac{v^S_x}{\sqrt{\delta_S}}$, and
\[\delta_S\int_0^1 f^2 \, dy=\int_\mathbb{R}|v^S_x||u-\bar{u}|^2 \, dx,\]
we have
\begin{align*}
    \mathcal{R}_1&\le -C\int_{\R} |v^S_x||(v-\bar{v},u-\bar{u},\theta-\bar{\theta})|^2\,dx-G_3\\
    &\quad+C\delta_S^{4/3}\left((\delta_R^{4/3}+\delta_C^{4/3})e^{-C\delta_St}+\frac{\delta_R^{4/3}}{(1+\e^2_Rt^2)^{q/3}}\right)+C\varepsilon_1^6\|(v-\bar{v},u-\bar{u},\theta-\bar{\theta})_x\|_{L^2}^2.
\end{align*}
This completes the proof of Lemma \ref{lem:leading}.
\end{proof}

\subsection{Estimates for the remaining terms}
Next, we control the terms in $\mathcal{R}_2$. Since the estimates for the terms in $\mathcal{R}_2$ are similar to those in \cite{KVW25}, we omit the details and refer the reader to \cite{KVW25}. Instead, we provide detailed estimates for the additional terms arising from the capillarity and the large strength of the rarefaction wave.

\noindent $\bullet$ (Estimate of $Y_i$ for $i=4,\ldots,7$): Due to the locally quadratic nature of $\Phi$ and \eqref{shock_est}, we have
\[|(Y_4,Y_5)|\le C\int_{\R} |v^S_x||(v-\bar{v},\theta-\bar{\theta})|^2\,dx\le CG^S,\]
and 
\[|(Y_4,Y_5)|\le C\delta_S^2\int_{\R} |(v-\bar{v},\theta-\bar{\theta})|^2\,dx\le C\delta_S^2\varepsilon_1^2.\]
Therefore, 
\[\frac{C}{\delta_S}|(Y_4,Y_5)|^2\le C\delta_S\varepsilon_1^2 G^S.\]
Next, we estimate $Y_6$ as
\begin{align*}
	\frac{C}{\delta_S}|Y_6|^2&\le \frac{C}{\delta_S}\left(\int_{\R} |a_x||(v-\bar{v},u-\bar{u},\theta-\bar{\theta})|^2\,dx\right)^2+\frac{C}{\delta_S}\left(\int_{\R}|a_x||w-\bar{w}|^2\,dx\right)^2\\
	&\le \frac{C}{\delta_S^2}\left(\int_{\R}|v^S_x||(v-\bar{v},u-\bar{u},\theta-\bar{\theta})|^2\,dx\right)^2+\frac{C}{\delta_S}\|a_x\|_{L^\infty}\|w-\bar{w}\|_{L^2}^2G_3\\
	&\le C\varepsilon_1^2\int_{\R}|v^S_x||(v-\bar{v},u-\bar{u},\theta-\bar{\theta})|^2\,dx+C\sqrt{\delta_S}\e^2_1G_3\\
	&\le C\varepsilon_1^2G^S+C\sqrt{\delta_S}\varepsilon_1^2G_3.
\end{align*}
Finally, since \eqref{shock_est} implies
\[|w_x^S|= \sqrt{\kappa}\left|\left(\frac{v^S_x}{(v^S)^{5/2}}\right)_x\right|\le C(|v_{xx}^S|+|v_x^S|^2)\le C \delta_S |v_x^S|,\]
we use Cauchy--Schwarz inequality to obtain
\begin{align*}
	\frac{C}{\delta_S}|Y_7|^2&\le \frac{C}{\delta_S}\left(\int_\R |w_x^S||w-\bar{w}|\,dx\right)^2\le \frac{C}{\delta_S}\left(\int_\R|w_x^S|\,dx\right)\left(\int_\R |w_x^S||w-\bar{w}|^2\,dx\right)\\
	&\le C\delta_S^2\int_\R |v_x^S||w-\bar{w}|^2\,dx\le C\delta_S^{3/2}\int_\R |a_x||w-\bar{w}|^2\,dx\le C\delta_S^{3/2}G_3.
\end{align*}
Combining all the estimates of $Y_i$ for $i=4,\ldots, 7$, we have
\[\frac{C}{\delta_S}\sum_{i=4}^7|Y_i|^2\le C\varepsilon_1^2G^S+C\sqrt{\delta_S}G_3.\]

\noindent $\bullet$ (Estimate of $B_i$ for $i=3,4,5$).
We first split $B_3$ as
\[
B_3 = B_{31} + B_{32},
\]
where
\begin{align*}
B_{31} &= -\int_\R a_x\left(\mu(u-\bar{u})\Bigl(\frac{u_x}{v}-\frac{\bar{u}_x}{\bar{v}}\Bigr)
        +\alpha \frac{\theta-\bar{\theta}}{\theta}\Bigl(\frac{\theta_x}{v}-\frac{\bar{\theta}_x}{\bar{v}}\Bigr)\right)\,dx,\\
B_{32} &= -\int_\R a_x\left(\sqrt{\kappa}(u-\bar{u})\Bigl(\frac{w_x}{v^{5/2}}-\frac{\bar{w}_x}{\bar{v}^{5/2}}\Bigr)
        -\sqrt{\kappa}(w-\bar{w})\Bigl(\frac{u_x}{v^{5/2}}-\frac{\bar{u}_x}{\bar{v}^{5/2}}\Bigr)\right)\,dx.
\end{align*}
Using $|a_x|=\tfrac{|v^S_x|}{\sqrt{\delta_S}}$, Young’s inequalities and \eqref{shock_est}, we obtain
\begin{equation}\label{est-B31}
\begin{aligned}
    B_{31} 
    &\le \frac{C}{\sqrt{\delta_S}}  \int_\mathbb{R} |v^S_x||u-\bar{u}| \bigl( |(u-\bar{u})_x|+|\bar{u}_x||v-\bar{v}| \bigr) \, dx \\
    &\quad + \frac{C}{\sqrt{\delta_S}} \int_\mathbb{R} |v^S_x| |\theta-\bar{\theta}|  \bigl( |(\theta-\bar{\theta})_x|+|\bar{\theta}_x||v-\bar{v}| \bigr) \, dx   \\
    &\le C\sqrt{\delta_S}  D  + C \sqrt{\delta_S} G^S + C \int_\mathbb{R} \bigl( |\bar{u}_x|^2 +|\bar{\theta}_x|^2  \bigr) | v-\bar{v}|^2 \, dx.
\end{aligned}
\end{equation}
Similarly, we estimate $B_{32}$ as
\begin{equation}\label{est-B32}
\begin{aligned}
       B_{32} 
    &\le  \frac{C}{\sqrt{\delta_S}} \int_\mathbb{R} |v^S_x| |u-\bar{u}|  \bigl( |(w-\bar{w})_x|+|\bar{w}_x||v-\bar{v}| \bigr) \, dx \\
    &\quad + C \int_\mathbb{R} |a_x| |w-\bar{w}|  \bigl( |(u-\bar{u})_x|+|\bar{u}_x||v-\bar{v}| \bigr) \, dx \\
    &\le C\sqrt{\delta_S} \bigl( D+ \|(w-\bar{w})_x\|_{L^2}^2 \bigr) + C \sqrt{\delta_S} G^S + C \delta_S G_3 + C \int_\mathbb{R} \bigl( |\bar{u}_x|^2 + |\bar{w}_x|^2 \bigr) | v-\bar{v}|^2 \, dx.
\end{aligned}
\end{equation}
Moreover, by Lemma~\ref{lem:superposition-derivatives} and $\|(v^R_x,v^S_x)\|_{L^\infty}\le C(\delta_R\e_R+\delta_S^2)$, we get
\begin{equation} \label{est-v}
    \begin{aligned}
        &\int_{\R} \bigl(|\bar u_x|^2 + |\bar\theta_x|^2 + |\bar w_x|^2\bigr)|v-\bar v|^2\,dx
\le C\int_{\R}\bigl(|v^R_x|^2+|v^S_x|^2\bigr)|v-\bar v|^2\,dx\\ &\quad+C\delta_C^2(1+t)^{-1}\int_{\R}e^{-\frac{2C_1|x|^2}{1+t}}|v-\bar v|^2\,dx \\
&\le C\delta_R \varepsilon_R G^R + C\delta_S G^S
   + C\delta_C^2 (1+t)^{-1}\int_{\R} e^{-\frac{2C_1|x|^2}{1+t}} |v-\bar v|^2\,dx .
    \end{aligned}
\end{equation}
Combining \eqref{est-B31}--\eqref{est-v} and using the smallness
of $\delta_S$, we conclude that for a sufficiently small constant $\tau$ that will be chosen later,
\begin{align*}
B_3
&\le \tau D
   + C\sqrt{\delta_S}\bigl(G^S+\|(w-\bar w)_x\|_{L^2}^2+G_3\bigr)
   + C\delta_R \varepsilon_R G^R \\
&\quad + C\delta_C^2 (1+t)^{-1}\int_{\R} e^{-\frac{2C_1|x|^2}{1+t}} |v-\bar v|^2\,dx .
\end{align*}
For $B_4$, we again decompose
\begin{align*}
B_4
&=\int_\R a\biggl[-\mu \bar{u}_x\Bigl(\frac{1}{v}-\frac{1}{\bar{v}}\Bigr)(u-\bar{u})_x
+\alpha\frac{\theta-\bar{\theta}}{\theta^2}\theta_x\Bigl(\frac{\theta_x}{v}-\frac{\bar{\theta}_x}{\bar{v}}\Bigr)
-\alpha\frac{(\theta-\bar{\theta})_x}{\theta}\bar{\theta}_x\Bigl(\frac{1}{v}-\frac{1}{\bar{v}}\Bigr)\\
&\hspace{2cm}
+\mu\frac{\theta-\bar{\theta}}{\theta}\Bigl(\frac{u_x^2}{v}-\frac{\bar{u}_x^2}{\bar{v}}\Bigr)
-\frac{(\theta-\bar{\theta})^2}{\theta \bar{\theta}}\Bigl(
    \alpha\Bigl(\frac{\bar{\theta}_x}{\bar{v}}\Bigr)_x
   +\mu \frac{\bar{u}_x^2}{\bar{v}} \Bigr)\biggr]\,dx\\
&\quad + \int_\R a\biggl[-\sqrt{\kappa}\Bigl(\bar{u}_x(w-\bar{w})_x\Bigl(\frac{1}{\bar{v}^{5/2}}-\frac{1}{v^{5/2}}\Bigr)
+(u-\bar{u})_x\bar{w}_x\Bigl(\frac{1}{v^{5/2}}-\frac{1}{\bar{v}^{5/2}}\Bigr)\Bigr)\biggr]\,dx\\
&=:B_{41}+B_{42}.
\end{align*}
We estimate $B_{41}$ by following a similar argument to that in \cite{KVW25}:
\begin{equation*}
    \begin{aligned}
B_{41} &\le C \int_\mathbb{R} \Bigl[ |\bar{u}_x||v-\bar{v}||(u-\bar{u})_x|+|\theta-\bar{\theta}| \bigl(|(\theta-\bar{\theta})_x|+|\bar{\theta}_x|\bigr)\bigl(|(\theta-\bar{\theta})_x|+|\bar{\theta}_x||v-\bar{v}|\bigr)\Bigr] dx\\
&+C \int_\mathbb{R} \Bigl[ |(\theta-\bar{\theta})_x||\bar{\theta}_x||v-\bar{v}|+|\theta-\bar{\theta}| \bigl(|(u-\bar{u})_x|^2+|\bar{u}_x||(u-\bar{u})_x|+|\bar{u}_x|^2|v-\bar{v}|\bigr)\Bigr] dx\\
&+C \int_\mathbb{R} |\theta-\bar{\theta}|^2 \bigl[ |\bar{\theta}_{xx}| +|(\bar{v}_x,\bar{\theta}_x,\bar{u}_x)|^2 \bigr] dx\\
&\le \tau D +C \int_\mathbb{R} \bigl(|\bar{\theta}_{xx}|+|(\bar{v}_x,\bar{\theta}_x,\bar{u}_x)|^2 \bigr)|(v-\bar{v},\theta-\bar{\theta})|^2 dx.
    \end{aligned}
\end{equation*}
For $B_{42}$, using the boundedness of $v$ and $\bar{v}$, we have
\begin{align*}
B_{42}
&\le C\int_\R |\bar{u}_x|\,|v-\bar{v}|\,|(w-\bar{w})_x|\,dx
   +C\int_\R |(u-\bar{u})_x|\,|\bar{w}_x|\,|v-\bar{v}|\,dx\\
&\le \tau D +C\int_\R |\bar{u}_x|^{1/2}|(w-\bar{w})_x|^2\,dx+C\int_\R \bigl(|\bar{u}_x|^{3/2}+|\bar{w}_x|^2\bigr)|v-\bar{v}|^2\,dx.
\end{align*}

Again, using Lemma~\ref{lem:rarefaction_property}, Lemma~\ref{lem:viscous_contact}, and Lemma~\ref{lem:shock-property}, we deduce
\begin{align*}
&\int_\mathbb{R} \bigl(|\bar{\theta}_{xx}|+|(\bar{v}_x,\bar{\theta}_x,\bar{u}_x)|^2 +|\bar{u}_x|^{3/2}+|\bar{w}_x|^2\bigr)|(v-\bar{v},\theta-\bar{\theta})|^2 dx\\
&\le C(\delta_R^{1/2} \varepsilon_R^{1/2}G^R+\delta_SG^S)
   +C \delta_C (1+t)^{-1}\int_\R e^{-\frac{C_1x^2}{1+t}}|(v-\bar{v},\theta-\bar{\theta})|^2\,dx,
\end{align*}
where the rarefaction part of $|\bar\theta_{xx}|$ is bounded by $|\theta^R_{xx}|\le C\e_R|v^R_x|$ by Lemma~\ref{lem:rarefaction_property}~(1),(2).
Therefore,
\begin{align*}
B_4
&\le 3\tau D
  + C(\delta_R^{1/2}\varepsilon_R^{1/2}G^R+\delta_SG^S)
  +C \delta_C (1+t)^{-1}
     \int_\R e^{-\frac{C_1x^2}{1+t}}|(v-\bar{v},\theta-\bar{\theta})|^2\,dx\\
&\quad + C(\delta_R^{1/2}\varepsilon_R^{1/2}+\delta_S^{1/2}+\delta_C^{1/2})
         \|(w-\bar{w})_x\|_{L^2}^2.
\end{align*}
Finally, to estimate $B_5$, we use the H\"older inequality, the Sobolev interpolation inequality, and Young's inequality to obtain
\begin{align*}
    &\int_\R |v^R-v_*||v^S_x||(v-\bar{v},\theta-\bar{\theta})|^2\,dx \le \| |v^S_x| |v^R-v_*|\|_{L^2} \| (v-\bar{v},\theta-\bar{\theta})\|_{L^4}^2 \\
    &\le C \| |v^S_x| |v^R-v_*|\|_{L^2} \| (v-\bar{v},\theta-\bar{\theta})\|_{L^2}^{3/2}\| (v-\bar{v},\theta-\bar{\theta})_x\|_{L^2}^{1/2} \\
    &\le C \| |v^S_x| |v^R-v_*|\|_{L^2}^{4/3}+C \| (v-\bar{v},\theta-\bar{\theta})\|_{L^2}^6 \| (v-\bar{v},\theta-\bar{\theta})_x\|_{L^2}^2.
\end{align*}
Thus, by the smallness assumption and Lemma \ref{lem: interaction estimates}, we have
\begin{align*}
    B_5&\le C\delta_C(1+t)^{-1}\int_\R e^{-\frac{C_1|x|^2}{1+t}}
    |(v-\bar{v},\theta-\bar{\theta})|^2\,dx+C(\delta_C+\delta_S+\varepsilon_1)(G^R+G^S)\\
      &\quad + C \bigl( \delta_R +\delta_C\bigr)^{4/3}\delta_S^{2} e^{-C \delta_S t}+C\delta_S^{2}\frac{\delta_R^{4/3}}{(1+\e_R^2t^2)^{q/3}}+ C \varepsilon_1^6 \| (v-\bar{v},\theta-\bar{\theta})_x\|_{L^2}^2.
\end{align*}
Collecting the estimates for $B_3$, $B_4$ and $B_5$, we obtain
\begin{align*}
\sum_{i=3}^5B_i
&\le 4\tau D
  + C \bigl(\delta_R^{1/2}\varepsilon_R^{1/2} + \delta_C+\delta_S^{1/2}+\varepsilon_1\bigr) \bigl(G^R+G^S+G_3\bigr)\\
  & \quad +C \varepsilon_1^6 \|(v-\bar{v},\theta-\bar{\theta})_x\|_{L^2}^2+C \bigl( \delta_R^{1/2}\varepsilon_R^{1/2} + \delta_C^{1/2}+\delta_S^{1/2}\bigr) \|(w-\bar{w})_x\|_{L^2}^2 \\
&\quad +C\delta_C(1+t)^{-1}\int_\R e^{-\frac{C_1|x|^2}{1+t}}
     |(v-\bar{v},\theta-\bar{\theta})|^2\,dx\\
     &\quad +C\delta_S^{2}\Bigl((\delta_R+\delta_C)^{4/3}e^{-C\delta_S t}+\frac{\delta_R^{4/3}}{(1+\e_R^2t^2)^{q/3}}\Bigr).
\end{align*}

\noindent $\bullet$ (Estimate of $S_i$ for $i=1,2,3$): We split the estimate of $S_1+S_2$ and $S_3$. For the terms in $S_1+S_2$, we recall that $Q_i=Q_i^I+Q_i^R+Q_i^C$ for $i=1,2$ to derive
\begin{align}
	\begin{aligned}\label{est:S1S2}
	S_1+S_2&=-\int_{\R}aQ_1(u-\bar{u})\,dx-\int_{\R}aQ_2\Bigl(\frac{\theta}{\bar{\theta}}-1\Bigr)dx\\
	&\le C\int_{\R} (|Q_1^I|+|Q_1^R|+|Q_1^C|)|u-\bar{u}|\,dx +C\int_{\R}(|Q_2^I|+|Q_2^R|+|Q_2^C|)|\theta-\bar{\theta}|\,dx\\
    &\le C \|(Q_1^I,Q_1^C,Q_2^I,Q_2^C)\|_{L^2}\|(u-\bar{u},\theta-\bar{\theta})\|_{L^2} + C \|(Q_1^R,Q_2^R)\|_{L^1} \|(u-\bar{u},\theta-\bar{\theta})\|_{L^\infty}\\
    &\le C \varepsilon_1 \|(Q_1^I,Q_1^C,Q_2^I,Q_2^C)\|_{L^2}  + C \|(Q_1^R,Q_2^R)\|_{L^1} \|(u-\bar{u},\theta-\bar{\theta})\|_{L^2}^{1/2} \|(u-\bar{u},\theta-\bar{\theta})_x\|_{L^2}^{1/2}\\
    &\le C \varepsilon_1 \|(Q_1^I,Q_1^C,Q_2^I,Q_2^C)\|_{L^2}  + C \varepsilon_1^{2/3} \|(Q_1^R,Q_2^R)\|_{L^1}^{4/3} + \tau D .
    \end{aligned}
\end{align}
Since the $L^2$-estimate of $Q_i^I$ and $Q_i^C$ can be found in Lemma \ref{lem: interaction estimates} and \eqref{error_contact} respectively, we focus on estimating the $L^1$-norm of $Q_i^R$. Recall that 
\[Q_1^R=-\mu\left(\frac{u_x^R}{v^R}\right)_x +\kappa\left(\frac{v^R_{xx}}{(v^R)^5}-\frac{5(v^R_x)^2}{2(v^R)^6}\right)_x,\quad Q^R_2 = -\alpha\left(\frac{\theta^R_x}{v^R}\right)_x-\mu\frac{(u^R_x)^2}{v^R}.\]
This implies 
\begin{align*}
	&|Q_1^R|\le C(|u_{xx}^R|+|u_x^R||v_x^R|+|v^R_{xxx}|+|v^R_{xx}||v^R_x|+|v^R_x|^3),\\
	&|Q_2^R|\le C(|\theta^R_{xx}|+|\theta^R_x||v^R_x|+|u^R_x|^2),
\end{align*}
and therefore,
\begin{align*}
	&\|Q_1^R\|_{L^1}\le C (\|u^R_{xx}\|_{L^1}+\|(v^R_x,u^R_x)\|_{L^2}^2+\|v^R_{xxx}\|_{L^1}+\|v^R_x\|_{L^2}^2+\|v^R_x\|_{L^3}^3)\\
	&\|Q_2^R\|_{L^1}\le C (\|\theta^R_{xx}\|_{L^1}+\|(v^R_x,u^R_x,\theta^R_x)\|_{L^2}^2).
\end{align*}
However, it follows from Lemma \ref{lem:rarefaction_property} that
\begin{align*}
	&\|(u^R_{xx},\theta^R_{xx})\|_{L^1}\le C\min\left\{\delta_R \varepsilon_R, \frac{\delta_R+\delta_R^{1/q}}{(1+t)^{1-1/q}}\right\}, \quad
    \|(v^R_x,u^R_x,\theta^R_x)\|_{L^2}\le C\min\left\{\delta_R\varepsilon_R^{1/2},\frac{\delta_R^{1/2}}{(1+t)^{1/2}}\right\},\\
	&\|v^R_{xxx}\|_{L^1}\le C\min\left\{\delta_R \varepsilon_R^2, \frac{\delta_R+\delta_R^{2/q}}{(1+t)^{1-2/q}}\right\},\quad
    \|v^R_x\|_{L^3}\le C\min\left\{\delta_R\varepsilon_R^{2/3},\frac{\delta_R^{1/3}}{(1+t)^{2/3}}\right\},
\end{align*}
and consequently,
\[\|Q^R_1\|_{L^1},\|Q^R_2\|_{L^1}\le C\min\left\{\delta_R \varepsilon_R, \frac{\delta_R+\delta_R^{2/q}}{(1+t)^{1-2/q}}\right\}.\]
Then, thanks to Lemma \ref{lem: interaction estimates} and \eqref{error_contact}, and the above estimate, we obtain
\begin{align*}
	S_1+S_2
	&\le \tau D + C\varepsilon_1^{2/3} \min\left\{\delta_R\varepsilon_R, \frac{\delta_R+\delta_R^{2/q}}{(1+t)^{1-2/q}}\right\}^{4/3}+C \varepsilon_1 \|(Q_1^I,Q_1^C,Q_2^I,Q_2^C)\|_{L^2}.
\end{align*}

Finally, we estimate $S_3$ as
\begin{align*}
	S_3 &= -\int_{\R} aQ_3(w-\bar{w})\,dx\le C\|Q_3\|_{L^2}\|w-\bar{w}\|_{L^2}\le C \varepsilon_1\|Q_3\|_{L^2}.
\end{align*}

\subsection{Proof of the estimate \eqref{est:L2}}

We combine all the estimates on the leading terms and the remaining terms. By choosing $\tau<\frac{1}{40}$, we obtain
\begin{align*}
	&\frac{d}{dt}\int_{\R} a\bar{\theta}\eta(U|\bar{U})\,dx \le \mathcal{R}_1+\mathcal{R}_2\\
	&\le -\frac{\delta_S}{4M}|\dot{X}|^2-\frac{C}{2}(G^S+G^R+G_3)-\frac{D}{8} +C\delta_S^{4/3}\left((\delta_R^{4/3}+\delta_C^{4/3})e^{-C\delta_St}+\frac{\delta_R^{4/3}}{(1+\e_R^2t^2)^{q/3}}\right)\\
	&\quad +C\delta_C(1+t)^{-1}\int_{\R}e^{-\frac{C_1|x|^2}{1+t}}|(v-\bar{v},\theta-\bar{\theta})|^2\,dx \\
	&\quad +C\varepsilon_1\|(Q^I_1,Q^C_1,Q^I_2,Q^C_2,Q_3)\|_{L^2}+ C\varepsilon_1^{2/3}\min\left\{\delta_R\varepsilon_R, \frac{\delta_R+\delta_R^{2/q}}{(1+t)^{1-2/q}}\right\}^{4/3} \\
	&\quad+ C \varepsilon_1^6 \|(v-\bar{v},u-\bar{u},\theta-\bar{\theta})_x\|_{L^2}^2+C(\delta_R^{1/2}\varepsilon_R^{1/2}+\delta_S^{1/2}+\delta_C^{1/2})\|(w-\bar{w})_x\|_{L^2}^2.
\end{align*}

Integrating this inequality over $[0,t]$ for any $0\le t\le T$ and using
$1 \le a \le \frac{3}{2}$, \eqref{error_contact}, Lemma \ref{lem: interaction estimates} and the uniform boundedness of $\bar{\theta}$, we obtain
\begin{align*}
    &\sup_{t\in[0,T]}\int_{\R}\eta(U|\bar{U})\,dx+\delta_S\int_0^T |\dot{X}|^2\, dt+\int_0^T (G^S+G^R+G_3+D)\, dt\\
	&\le C\int_{\R}\eta(U_0|\bar{U}(0,x))\,dx + C \varepsilon_1 \int_0^T \|(v-\bar{v},u-\bar{u},\theta-\bar{\theta})_x\|_{L^2}^2 \, dt\\
    &\quad +C(\delta_R^{1/2}\varepsilon_R^{1/2}+\delta_C^{1/2} + \delta_S^{1/2})\int_0^T\|(w-\bar{w})_x\|_{L^2}^2\, dt\\
	&\quad +C\delta_C\int_0^T \frac{1}{1+t}\int_\R e^{-\frac{C_1|x|^2}{1+t}}|(v-\bar{v},\theta-\bar{\theta})|^2\,dx\, dt +C\delta_0^{1/4}+C \delta_0^{1/3}+C\delta_0^{4/3}\e_R^{-1}+C\e_R^{\frac{q-8}{3(q-2)}}.
\end{align*}
Here, we used $q>8$ for the integrability of $(1+t)^{-\frac{4}{3}+\frac{8}{3q}}$ in time, and the term $C\delta_0^{1/4}$ collects the contributions of $\varepsilon_1\int_0^T\|(Q_1^I,Q_2^I,Q_3)\|_{L^2}\,dt$ from Lemma \ref{lem: interaction estimates}. To close the relative entropy estimate, we use the following estimate similar to \cite[Lemma 4.5]{KVW25}:

\begin{align}\label{heat-int}
	\begin{aligned}
		&\int_0^t \frac{1}{1+s}\int_\R e^{-\frac{C_1|x|^2}{1+s}}|(v-\bar{v},u-\bar{u},\theta-\bar{\theta})|^2\,dxds\\
		&\le C\sup_{s\in[0,t]}\|(v-\bar v, u-\bar u, \theta-\bar \theta, w- \bar w)(s)\|^2_{L^2}+C\delta_S\int_0^t |\dot{X}(s)|^2\,ds\\ 
        &\quad+C\int_0^t (G^S+G^R+D)\,ds +C\int_0^t \|(v-\bar{v},w-\bar{w})_x\|^2_{L^2}\,ds+C\delta_0^{1/6}.
	\end{aligned}
\end{align}

Now we use $\e_R=\delta_0/2$: the terms $C\delta_0^{1/4}+C\delta_0^{1/3}+C\delta_0^{4/3}\e_R^{-1}+C\e_R^{\frac{q-8}{3(q-2)}}$ are bounded by $C\delta_0^{1/6}$, where we used that $\frac{q-8}{3(q-2)}\ge\frac{1}{6}$ holds if and only if $q\ge 14$. This is precisely the origin of the standing assumption $q\ge14$. Combining this with the smallness of $\delta_0,\varepsilon_1$ and \eqref{heat-int}, we deduce that

\begin{align*}
	&\sup_{t\in[0,T]}\int_{\R}\eta(U|\bar{U})\,dx+\delta_S\int_0^T |\dot{X}|^2\, dt+\int_0^T (G^S+G^R+G_3+D)\, dt\\
	&\le C\int_{\R}\eta(U_0|\bar{U}(0,x))\,dx +C (\varepsilon_1 + \delta_0^{1/2})\int_0^T \|(v-\bar{v},w-\bar{w})_x\|_{L^2}^2 \, dt  +C\delta_0^{1/6}.
\end{align*}

Finally, by the strict convexity of the entropy and the uniform bounds on $U$ and $\bar U$, we have $\int_{\R}\eta(U|\bar{U})\,dx
\sim \|U-\bar{U}\|_{L^2}^2$. Hence, we obtain the desired estimate \eqref{est:L2}.

\section{Estimate on the $H^1$ perturbation}\label{sec:H1}
\setcounter{equation}{0}
In this section, we derive higher–order a priori estimates for the perturbation, thereby completing the proof of Proposition~\ref{prop:apriori}. Our objective is to establish the following estimate:
\begin{equation}\label{eq:H1-est}
\begin{aligned}
\sup_{0\le s\le t}&\|(v-\bar{v},u-\bar{u},\theta-\bar{\theta},w-\bar{w})(s,\cdot)\|_{H^1}^2+\delta_S\int_0^t |\dot{X}(s)|^2\,ds\\
&+\int_0^t(G^R+G^S)\,ds
+\int_0^t\bigl(\|(u-\bar{u},\theta-\bar{\theta})_x\|_{H^1}^2
+\|w-\bar{w}\|_{H^2}^2\bigr)\,ds\\
&\le C_0\|(v-\bar v,u-\bar u,\theta-\bar\theta,w-\bar w)(0,\cdot)
\|_{H^1}^2
+ C_0\delta_0^{1/6},
\end{aligned}
\end{equation}
for all $t\in[0,T]$. This estimate coincides with \eqref{est:apriori} in Proposition~\ref{prop:apriori}.

\medskip
We briefly outline the strategy used to obtain \eqref{eq:H1-est}. An $L^2$-energy estimate for the perturbation~\eqref{est:L2} has already been established in the previous section, but this estimate involves the derivative terms $\|(v-\bar{v},w-\bar{w})_x\|_{L^2}^2$, so it cannot be closed at the $L^2$-level.

To overcome this difficulty, we structure the proof sequentially. First, in Lemma~\ref{lem:est-0}, we differentiate the equations and derive $H^1$-estimates for the velocity and temperature perturbations $(u-\bar u,\theta-\bar\theta)$, obtaining $L^2$-bounds for $(u-\bar{u},\theta-\bar{\theta})_x$. At the same time, the estimate also involves the derivatives of $w-\bar{w}$. However, in contrast to the velocity and temperature, no dissipative terms for $w-\bar{w}$ arise at this stage; these terms appear only as error contributions that cannot be absorbed.

We introduce suitable mixed energy estimates in Lemmas~\ref{lem:est-1} and \ref{lem:est-2} that exploit the structure of the capillarity variable. These mixed estimates generate dissipation for $w-\bar{w}$ (via Lemma~\ref{lem:est-1}) and its higher-order derivatives (via Lemma~\ref{lem:est-2}), allowing us to control the $w-\bar{w}$ terms produced in Lemma~\ref{lem:est-0}. Finally, combining these estimates with the $L^2$-energy estimate~\eqref{est:L2} and the basic $H^1$-estimates yields the desired $H^1$-estimates.

\medskip
The a priori smallness assumption \eqref{smallness}, together with Sobolev embedding, implies the uniform bounds
\begin{equation}\label{eq:Linf-bounds}
\|(\phi,\phi_x,\psi,\vartheta)\|_{L^\infty((0,T)\times\R)} \le C \varepsilon_1,
\end{equation}
and, for the superposition waves~\eqref{def:superposition},
\begin{equation}\label{eq:Linf-bounds-2}
\|(\bar{v}_x,\bar{u}_x,\bar{\theta}_x)\|_{L^\infty((0,T)\times\R)} \le C \delta_0.
\end{equation}

\subsection{High-order estimates}
We first collect several identities for the pressure that will be used repeatedly in the sequel. Let
\[
p=p(v,\theta):=\frac{R\theta}{v},\qquad 
\bar p:=p(\bar v,\bar\theta)=\frac{R\bar\theta}{\bar v}.
\]
We denote
\[
p_v:=\partial_v p(v,\theta),\qquad p_\theta:=\partial_\theta p(v,\theta),
\]
and similarly define $\bar p_v,\bar p_\theta$ at $(\bar v,\bar\theta)$.
Then we have
\begin{equation}\label{eq:p_derivs}
p_v=-\frac{R\theta}{v^2},\qquad p_\theta=\frac{R}{v},\qquad
p_{vv}=\frac{2R\theta}{v^3},\qquad
p_{v\theta}=p_{\theta v}=-\frac{R}{v^2},\qquad
p_{\theta\theta}=0,
\end{equation}
and the corresponding identities for $\bar p$ follow by replacing $(v,\theta)$ with $(\bar v,\bar\theta)$. A direct application of the chain rule yields
\begin{equation}\label{eq:ppbar_x}
\begin{aligned}
(p-\bar p)_x&= p_v\phi_x+p_\theta\vartheta_x+\bar v_x(p_v-\bar p_v)+\bar\theta_x(p_\theta-\bar p_\theta)\\
&=-\frac{R\theta}{v^2}\phi_x+\frac{R}{v}\vartheta_x
-R\bar v_x\!\left(\frac{\theta}{v^2}-\frac{\bar\theta}{\bar v^2}\right)
+R\bar\theta_x\!\left(\frac{1}{v}-\frac{1}{\bar v}\right).
\end{aligned}
\end{equation}
Differentiating \eqref{eq:ppbar_x} once more, we obtain
\begin{equation}\label{eq:ppbar_xx}
\begin{aligned}
(p-\bar p)_{xx}
&=p_v\phi_{xx}+p_\theta\vartheta_{xx} +\bar{v}_{xx}(p_v-\bar{p}_v)+\bar{\theta}_{xx}(p_\theta-\bar{p}_\theta)\\
&\quad+p_{vv} \phi_x (v_x+\bar{v}_x) + (\bar{v}_x)^2 (p_{vv}-\bar{p}_{vv})\\
&\quad +2p_{v\theta} (\phi_x \vartheta_x +\bar{\theta}_x\phi_x + \bar{v}_x \vartheta_x)+2 \bar{v}_x \bar{\theta}_x(p_{v\theta}-\bar{p}_{v \theta}) \\
&\quad +p_{\theta \theta} \vartheta_x(\theta_x+\bar{\theta}_x)+(\bar{\theta}_x)^2(p_{\theta \theta}-\bar{p}_{\theta \theta})\\
&=-\frac{R \theta}{v^2}\phi_{xx}+\frac{R}{v}\vartheta_{xx} -R\bar{v}_{xx}\Bigl(\frac{ \theta}{v^2}-\frac{\bar{\theta}}{\bar{v}^2}\Bigr)+ R\bar{\theta}_{xx}\Bigl(\frac{1}{v}-\frac{1}{\bar v}\Bigr)
\\
&\quad +\frac{2R\theta}{v^3} \phi_x (v_x+\bar{v}_x) +2R (\bar{v}_x)^2 \Bigl(\frac{\theta}{v^3}-\frac{\bar{\theta}}{\bar{v}^3} \Bigr)\\
&\quad -2\frac{R}{v^2} (\phi_x \vartheta_x +\bar{\theta}_x\phi_x + \bar{v}_x \vartheta_x)
-2R \bar{v}_x \bar{\theta}_x \Bigl(\frac{1}{v^2}-\frac{1}{\bar{v}^2}\Bigr) .
\end{aligned}
\end{equation}

\medskip
\begin{lemma}\label{lem:est-0}
Under the hypotheses of Proposition~\ref{prop:apriori}, there exists a constant $C>0$, independent of $\delta_0$ and $\varepsilon_1$, such that for all $t\in[0,T]$,
\begin{equation}\label{eq: est-0}
\begin{aligned}
&\sup_{0\le s\le t} \|(\psi_x,\vartheta_x,\omega_x)(s,\cdot)\|_{L^2}^2+\int_0^t \bigl(D_{u_2}+D_{\theta_2}\bigr)(s)\,ds \\
&\quad\le C \|(\psi_x,\vartheta_x,\omega_x)(0,\cdot)\|_{L^2}^2 + C \delta_S \int_0^t |\dot{X}(s)|^2\, ds+ C\int_0^t \bigl(\|\phi_x(s,\cdot)\|_{L^2}^2+D_{u_1}(s)+D_{\theta_1}(s)\bigr)\, ds\\
&\qquad +C (\varepsilon_1 + \delta_0^{1/2}) \int_0^t  \bigl( \|\omega_x(s,\cdot)\|_{L^2}^2 +\| \omega_{xx}(s,\cdot)\|_{L^2}^2 \bigr)\, ds+ C \delta_0^{1/2} \int_0^t \bigl(G^R+G^S\bigr)(s)\, ds \\
&\qquad +C \delta_0 \int_0^t \frac{1}{1+s} \int_\mathbb{R} e^{-\frac{3 C_1 |x|^2}{2 (1+s)}} |(\phi,\vartheta)(s,x)|^2 \, dx\, ds + C \delta_0,
\end{aligned}
\end{equation}
where 
\[
D_{u_2}:=\mu \int_{\R} \frac{|\psi_{xx}|^2}{v}\,dx, \qquad
D_{\theta_2}:=\alpha \int_{\R} \frac{|\vartheta_{xx}|^2}{v}\,dx.
\]
\end{lemma}

\begin{proof}
Multiplying \eqref{eq:perturbation}$_2$, \eqref{eq:perturbation}$_3$, and \eqref{eq:perturbation}$_4$
by $-\psi_{xx}$, $-\vartheta_{xx}$, and $-\omega_{xx}$, respectively, we obtain
\begin{align*}
\frac{d}{dt} \int_{\R} \left(
   \frac{|\psi_x|^2}{2}
 + \frac{R |\vartheta_x|^2}{2(\gamma-1)}
 + \frac{|\omega_x|^2}{2}
\right)\,dx
= \sum_{i=1}^6 K_i,
\end{align*}
where
\begin{align*}
    K_1&:=- \dot{X}(t) \int_{\R} \left( u^S_x \psi_{xx} + \frac{R}{\gamma-1} \theta^S_x \vartheta_{xx} + w^S_x \omega_{xx} \right)\,dx, \\
    K_2&:= \int_{\R} \bigl((p-\bar{p})_x \psi_{xx} + (p u_x - \bar{p} \bar{u}_x)\vartheta_{xx}\bigr)\,dx,\\
    K_3&:=- \mu \int_{\R} \Bigl( \frac{u_x}{v} - \frac{\bar{u}_x}{\bar{v}} \Bigr)_x \psi_{xx}\,dx - \alpha \int_{\R} \Bigl( \frac{\theta_x}{v} - \frac{\bar{\theta}_x}{\bar{v}} \Bigr)_x \vartheta_{xx}\,dx, \\ 
    K_4&:= \sqrt{\kappa} \int_{\R} \Bigl[\Bigl( \frac{u_x}{v^{5/2}} - \frac{\bar{u}_x}{\bar{v}^{5/2}} \Bigr)_x \omega_{xx}- \Bigl( \frac{w_x}{v^{5/2}} - \frac{\bar{w}_x}{\bar{v}^{5/2}} \Bigr)_x \psi_{xx}\Bigr]\,dx, \\
    K_5&:=- \mu \int_{\R} \Bigl( \frac{u_x^2}{v} - \frac{\bar{u}_x^{\,2}}{\bar{v}} \Bigr)\vartheta_{xx}\,dx, \quad
    K_6:=\int_{\R} \bigl( Q_1 \psi_{xx} + Q_2 \vartheta_{xx} + Q_3 \omega_{xx} \bigr)\,dx.
\end{align*}

\noindent $\bullet$ (Estimate of $K_1$):
Using $|u^S_x|,|\theta^S_x|\le C|v^S_x|$ and $|w^S_x|\le C(|v^S_{xx}|+|v^S_x|^2)\le C|v^S_x|$, we have
\begin{align*}
K_1
&\le C\delta_S^{3/2} |\dot{X}(t)|(\|\psi_{xx}\|_{L^2}+\|\vartheta_{xx}\|_{L^2}+\|\omega_{xx}\|_{L^2})\\
&\le \frac{1}{8}(D_{u_2}+D_{\theta_2})+\delta_S^2|\dot{X}(t)|^2+C\delta_S\|\omega_{xx}\|_{L^2}^2.
\end{align*}

\noindent $\bullet$ (Estimate of $K_2$):
Using \eqref{eq:ppbar_x} and
\[
pu_x-\bar{p}\bar{u}_x
= p\psi_x + (p-\bar{p})\bar{u}_x
= p\psi_x + \left(\frac{R\vartheta}{v}+R\bar{\theta}\Bigl(\frac{1}{v}-\frac{1}{\bar v}\Bigr)\right)\bar{u}_x,
\]
we obtain
\begin{align*}
K_2 
&\le \frac{1}{8} \|(\psi_{xx},\vartheta_{xx})\|_{L^2}^2+C\int_{\R} \bigl(|(p-\bar{p})_x|^2 + |(pu_x-\bar{p} \bar{u}_x)|^2 \bigr)\,dx\\
&\le \frac{1}{8} (D_{u_2}+D_{\theta_2})
 + C\Big(\|(\phi_x,\psi_x,\vartheta_x)\|_{L^2}^2
 +\int_{\R}|(\bar{v}_x,\bar{u}_x,\bar{\theta}_x)|^2|(\phi,\vartheta)|^2\,dx\Big)\\
&\le \frac{1}{8} (D_{u_2}+D_{\theta_2})+C \|(\phi_x,\psi_x,\vartheta_x)\|^2_{L^2}+C \delta_0 (G^R+G^S) + C\delta_0^2 (1+t)^{-1}\int_{\R} e^{-\frac{2C_1|x|^2}{1+t}} |(\phi,\vartheta)|^2\,dx .
\end{align*}

\noindent $\bullet$ (Estimate of $K_3$):
We extract the good terms $D_{u_2}$ and $D_{\theta_2}$ as
\begin{align*}
K_3
=& -D_{u_2}
-\mu\int_{\R}\Bigl(\frac{1}{v}\Bigr)_x\psi_x\psi_{xx}\,dx
-\mu\int_{\R}\Bigl(\bar u_x\Bigl(\frac{1}{v}-\frac{1}{\bar v}\Bigr)\Bigr)_x\psi_{xx}\,dx\\
 &-D_{\theta_2}
-\alpha\int_{\R}\Bigl(\frac{1}{v}\Bigr)_x\vartheta_x\vartheta_{xx}\,dx
-\alpha\int_{\R}\Bigl(\bar\theta_x\Bigl(\frac{1}{v}-\frac{1}{\bar v}\Bigr)\Bigr)_x\vartheta_{xx}\,dx.
\end{align*}

Using $|v_x|\le |\phi_x|+|\bar v_x|$ and Lemma~\ref{lem:superposition-derivatives}, we have
\begin{align*}
K_3
&\le -\bigl(D_{u_2}+D_{\theta_2}\bigr)
 +C\int_{\R}(|\phi_x|+|\bar v_x|)|(\psi_x,\vartheta_x)|\,|(\psi_{xx},\vartheta_{xx})|\,dx \\
&\quad +C\int_{\R}|(\bar u_{xx},\bar\theta_{xx})||\phi|\,|(\psi_{xx},\vartheta_{xx})|\,dx
 +C\int_{\R}|(\bar u_x,\bar\theta_x)|(|\phi_x|+|\bar v_x||\phi|)\,|(\psi_{xx},\vartheta_{xx})|\,dx \\
&\le -\frac{7}{8}\bigl(D_{u_2}+D_{\theta_2}\bigr)
 +C(\varepsilon_1+\delta_0)\bigl(D_{u_1}+D_{\theta_1}\bigr)
 +C\delta_0^2\|\phi_x\|_{L^2}^2 \\
&\quad +C\delta_0(G^R+G^S)
 +C\delta_0^2(1+t)^{-1}\int_{\R}e^{-\frac{2C_1|x|^2}{1+t}}|\phi|^2\,dx .
\end{align*}
Here we used the Sobolev interpolation inequality and the smallness assumption \eqref{smallness},
\begin{equation*} 
    \|\phi_x\|_{L^\infty}\le C\|\phi_x\|_{L^2}^{1/2}\|\phi_{xx}\|_{L^2}^{1/2}\le C\varepsilon_1.
\end{equation*}

\noindent $\bullet$ (Estimate of $K_4$):
We directly calculate $K_4$ as
\begin{align*} 
    K_4 &=  \sqrt{\kappa}\int_{\mathbb{R}}\psi_x \Bigl( \frac{1}{v^{5/2}}\Bigr)_x \omega_{xx} \, dx +\sqrt{\kappa }\int_{\mathbb{R}} \Bigl[ \bar{u}_{x} \Bigl( \frac{1}{v^{5/2}}-\frac{1}{\bar{v}^{5/2}} \Bigr) \Bigr]_x \omega_{xx}dx\\
    &\quad  -\sqrt{\kappa}\int_{\mathbb{R}}\omega_x \Bigl( \frac{1}{v^{5/2}}\Bigr)_x \psi_{xx} \, dx- \sqrt{\kappa}\int_{\mathbb{R}} \Bigl[ \bar{w}_{x} \Bigl( \frac{1}{v^{5/2}}-\frac{1}{\bar{v}^{5/2}} \Bigr)\Bigr]_x \psi_{xx} \, dx .
\end{align*}
Using $|v_x| \le |\phi_x| +|\bar{v}_x|$, $| \bar{w}_x|\le C (|\bar{v}_{xx}|+|\bar{v}_x|^2)$ and $| \bar{w}_{xx}|\le C (|\bar{v}_{xxx}|+|\bar{v}_x||\bar{v}_{xx}|+|\bar{v}_x|^3)$, we have
\begin{align*}
    K_4
    &\le C\int_{\R}|v_x|(  |\psi_x| |\omega_{xx}| + |\omega_x| |\psi_{xx}|)\,dx  +C\int_{\R}|\bar{u}_{xx}||\phi||\omega_{xx}|\,dx  +C\int_{\R}|\bar{u}_x|(|\phi_x|+|\bar{v}_x| |\phi|)|\omega_{xx}|\,dx\\
    &\quad +C\int_{\R}|\bar{w}_{xx}||\phi||\psi_{xx}|\,dx+C\int_{\R}|\bar{w}_x|(|\phi_x|+|\bar{v}_x| |\phi|)|\psi_{xx}|\,dx\\
&\le \frac{1}{8}D_{u_2}+ C (\varepsilon_1 + \delta_0 ) \bigl( D_{u_1}+\|\omega_x\|_{L^2}^2+\|\omega_{xx}\|_{L^2}^2\bigr) +C\delta_0^{1/2} \bigl( \|\phi_x\|_{L^2}^2+\|\omega_{xx}\|_{L^2}^2 \bigr) \\
&\qquad+ C\int_\mathbb{R} \bigl( |\bar{u}_{xx}|^{3/2} + |\bar{v}_x|^2+|\bar{w}_{xx}|^2+|\bar{w}_x|^2|\bar{v}_x|^2\bigr) |\phi|^2 \, dx\\
&\le \frac{1}{8}D_{u_2}+ C (\varepsilon_1 + \delta_0 ) \bigl( D_{u_1}+\|\omega_x\|_{L^2}^2+\|\omega_{xx}\|_{L^2}^2\bigr) +C\delta_0^{1/2} \bigl( G^R+G^S+\|\phi_x\|_{L^2}^2+\|\omega_{xx}\|_{L^2}^2 \bigr)\\
&\quad+ C  \delta_0 (1+t)^{-1} \int_\mathbb{R} e^{-\frac{3C_1x^2}{2(1+t)}} |\phi|^2\, dx.
\end{align*}
Here the term $|\bar u_{xx}|^{3/2}$ is estimated by Lemma~\ref{lem:rarefaction_property}~(2), Lemma~\ref{lem:viscous_contact}, and \eqref{shock_est}, and the remaining weights by Lemma~\ref{lem:superposition-derivatives}.

\noindent $\bullet$ (Estimate of $K_5$):
Using
\[
\frac{u_x^2}{v}-\frac{\bar u_x^2}{\bar v}
=\frac{\psi_x(\psi_x+2\bar u_x)}{v}+\bar u_x^2\Bigl(\frac1v-\frac1{\bar v}\Bigr),
\]
we obtain
\begin{align*}
K_5
&\le C\int_{\R}\Big(|\psi_x|^2+|\psi_x||\bar u_x|+|\bar u_x|^2|\phi|\Big)|\vartheta_{xx}|\,dx \\
&\le \frac{1}{16} D_{\theta_2} +C \int_\mathbb{R} |\psi_x|^2 |\vartheta_{xx}| \, dx +C \delta_0^2(G^R+G^S+D_{u_1})+C \delta_0^4 (1+t)^{-2}\int_\mathbb{R} e^{-\frac{2C_1|x|^2}{1+t}}|\phi|^2\,dx.
\end{align*}
Moreover, the Gagliardo--Nirenberg inequality and a priori smallness imply
\[
\int_{\R}|\psi_x|^2|\vartheta_{xx}|\,dx
\le C\|\psi_x\|_{L^2}\|\psi_x\|_{L^\infty}\|\vartheta_{xx}\|_{L^2}
\le C\varepsilon_1\big(\|\psi_x\|_{L^2}^2+\|\psi_{xx}\|_{L^2}^2+\|\vartheta_{xx}\|_{L^2}^2\big),
\]
which yields 
\begin{align*}
    K_5 \le \frac{1}{8} (D_{u_2}+
    D_{\theta_2} )+ C (\varepsilon_1+\delta_0) D_{u_1} + C \delta_0
    ^2(G^R+G^S) +C \delta_0^4 (1+t)^{-2}\int_\mathbb{R} e^{-\frac{2C_1|x|^2}{1+t}}|\phi|^2\,dx.
\end{align*}

\noindent $\bullet$ (Estimate of $K_6$):
Finally, we estimate $K_6$ as
\begin{align*}
K_6 &\le \frac{1}{8} (D_{u_2}+D_{\theta_2})+\varepsilon_1\|\omega_{xx}\|_{L^2}^2+C \|(Q_1,Q_2)\|_{L^2}^2+\frac{C}{\varepsilon_1}\|Q_3\|_{L^2}^2.
\end{align*}

Therefore, combining all estimates for $K_i$, we obtain
\begin{align}
\begin{aligned}\label{est:psi_x}
&\frac{d}{dt}\int_{\mathbb{R}} \left(\frac{|\psi_x|^2}{2} + \frac{R|\vartheta_x|^2}{2(\gamma-1)}+\frac{|\omega_x|^2}{2} \right) dx + \frac{1}{4} (D_{u_2}+D_{\theta_2}) \\ 
&\le \delta_S^2 |\dot{X}|^2 +C\| \phi_x\|_{L^2}^2 +C (D_{u_1}+D_{\theta_1}) +C (\varepsilon_1+\delta_0^{1/2})  \|\omega_x\|_{H^1}^2 \\
&\quad + C\delta_0^{1/2} (G^R+G^S) + C \delta_0(1+t)^{-1} \int_\mathbb{R} e^{-\frac{3C_1|x|^2}{2(1+t)}}|(\phi,\vartheta)|^2 \, dx \\
&\quad + C\left(\|(Q_1^I,Q_1^C,Q_2^I,Q_2^C)\|_{L^2}^2 + \|(Q_1^R,Q_2^R)\|_{L^2}^2+ \frac{1}{\varepsilon_1}\|Q_3\|_{L^2}^2\right).
\end{aligned}
\end{align}

We note that
\begin{align*}
    \|(Q_1^R,Q_2^R)\|_{L^2}^2 
    &\le C \big( \|u^R_{xx}\|_{L^2}^2 + \|u^R_x v^R_x\|_{L^2}^2+\|v^R_{xxx}\|_{L^2}^2+\|v^R_xv^R_{xx}\|_{L^2}^2+ \|(v^R_x)^3\|_{L^2}^2\\
    &\hspace{3em}+\|\theta^R_{xx}\|_{L^2}^2+\|\theta^R_xv^R_x\|_{L^2}^2+\|(u^R_x)^2\|_{L^2}^2 \bigr)\\
    & \le C \bigl( \|u^R_{xx}\|_{L^2}^2+\|u^R_{xxx}\|_{L^2}^2+\|u^R_x\|_{L^4}^4+\|u^R_x\|_{L^6}^6 \bigr).
\end{align*}

Since the wave strength $\delta_R = O(1)$ and the smoothing parameter is chosen as $\e_R = \delta_0/2$, we can absorb $\delta_R$ into a generic constant $C>0$ and represent the estimates solely in terms of $\delta_0$:
\begin{align*}
&\|u^R_{xx}\|_{L^2} \le C \min \left\{ \delta_0^{3/2}, \frac{1}{(1+t)^{1-1/q}} \right\}, \quad  \|u^R_x\|_{L^4} \le C \min \left\{ \delta_0^{3/4}, \frac{1}{(1+t)^{3/4}} \right\}, \\
&\|u^R_{xxx}\|_{L^2} \le C \min \left\{ \delta_0^{5/2}, \frac{1}{(1+t)^{1-2/q}} \right\}, \quad  \|u^R_{x}\|_{L^6} \le C \min \left\{ \delta_0^{5/6}, \frac{1}{(1+t)^{5/6}} \right\}.
\end{align*}

Integrating these bounds over $[0, t]$ yields
\begin{equation*}
\int_0^t \|u^R_{xx}\|_{L^2}^2 \,ds \le C \delta_0^{\frac{3(q-2)}{2(q-1)}}, \quad \int_0^t \|u^R_{xxx}\|_{L^2}^2 \,ds \le C \delta_0^{\frac{5(q-4)}{2(q-2)}},
\end{equation*}
\begin{equation*}
\int_0^t \|u^R_x\|_{L^4}^4 \,ds \le C \delta_0^2, \quad \int_0^t \|u^R_x\|_{L^6}^6 \,ds \le C \delta_0^4.
\end{equation*}
For $q \ge 14$, the minimum exponent among these is $\frac{3(q-2)}{2(q-1)} \ge \frac{18}{13} > 1$. Hence, we obtain the uniform bound
\begin{equation} \label{Q1R-L2est}
\int_0^t \|(Q_1^R,Q_2^R)\|_{L^2}^2\,ds \le C  \delta_0.
\end{equation}    

Integrating \eqref{est:psi_x} over $(0,t)$ for any $t\in[0,T]$ and using
\eqref{est-time-int-interaction}, \eqref{error_contact}, and \eqref{Q1R-L2est}, we conclude the desired estimate.
\end{proof}

\begin{lemma}\label{lem:est-1}
Under the hypotheses of Proposition~\ref{prop:apriori}, there exists a constant $C>0$, independent of $\delta_0$ and $\varepsilon_1$, such that for all $t\in[0,T]$,
\begin{equation}\label{eq: est-1}
\begin{aligned}
&  \int_\mathbb{R} \psi(t,x)\,\omega(t,x) \, dx+  \frac12\int_0^t \left(G_{w_1} +D_{w_1} \right)(s)\, ds \\
&\le   \int_\mathbb{R} \psi_0 \omega_0 \, dx+ \frac{\delta_S}{2}\int_0^t| \dot{X}(s)|^2\,ds + C \int_0^t \bigl( D_{u_1} + D_{\theta_1} \bigr)(s)\, ds \\
&\quad +C \delta_0 \int_0^t \left(G^R + G^S  \right)(s)\, ds+ C\delta_0^{2} \int_0^t \frac{1}{1+s}\int_{\R} e^{-\frac{2C_1|x|^2}{1+s}} |(\phi,\vartheta)(s,x)|^2\,dx\, ds  + C \delta_0^{1/2},
\end{aligned}
\end{equation}
where
\[
G_{w_1}:=\int_{\R}\frac{R\theta v^{1/2}}{\sqrt{\kappa}}|\omega|^2\,dx, \qquad
D_{w_1}:=\int_{\R}\frac{\sqrt{\kappa}}{v^{5/2}}|\omega_x|^2\,dx .
\]
\end{lemma}

\begin{proof}
Multiplying \eqref{eq:perturbation}$_2$ and \eqref{eq:perturbation}$_4$ by $\omega$ and $\psi$, respectively, and integrating over $\R$, we obtain
\[\frac{d}{dt}\int_{\R}\psi\,\omega\,dx=\sum_{i=1}^6 L_i ,\]
where
\begin{align*}
L_1&:=\dot X\int_{\R}\bigl(u_x^S\omega+w_x^S\psi\bigr)\,dx, &
L_2&:=-\int_{\R}(p-\bar p)_x\omega\,dx,\\
L_3&:=\mu\int_{\R}\Bigl(\frac{u_x}{v}-\frac{\bar u_x}{\bar v}\Bigr)_x\omega\,dx,&
L_4&:=\sqrt{\kappa}\int_{\R}\Bigl(\frac{w_x}{v^{5/2}}-\frac{\bar w_x}{\bar v^{5/2}}\Bigr)_x\omega\,dx,\\
L_5&:=-\sqrt{\kappa}\int_{\R}\Bigl(\frac{u_x}{v^{5/2}}-\frac{\bar u_x}{\bar v^{5/2}}\Bigr)_x\psi\,dx,&
L_6&:=-\int_{\R}\bigl(Q_1\omega+Q_3\psi\bigr)\,dx .
\end{align*}

\noindent$\bullet$ (Estimate of $L_1$):
Since $|u_x^S|\le C|v_x^S|$ and $|w_x^S|\le C\delta_S|v_x^S|$, we have
\begin{align*}
L_1
&\le C  |\dot{X}| \Biggl( \|v^S_x\|_{L^2}\|\omega\|_{L^2} +\delta_S\sqrt{ \int_\mathbb{R} |v^S_x| \, dx}\sqrt{\int_\mathbb{R} |v^S_x| |\psi|^2 \, dx } \; \Biggr) \le \frac{\delta_S}{2} |\dot{X}|^2 + C\delta_S\bigl(G^S+G_{w_1}\bigr).
\end{align*}

\noindent $\bullet$ (Estimate of $L_2$):
By the definition of $w$ and $\bar{w}$, we note that
\begin{equation}\label{eq:phi_x_w}
\sqrt{\kappa} \, \phi_x=-v^{5/2}\omega-\bar w( v^{5/2}-\bar{v}^{5/2} ).
\end{equation}
Using the expansion of $(p-\bar p)_x$ in \eqref{eq:ppbar_x} and \eqref{eq:phi_x_w}, we extract the good term $G_{w_1}$ and obtain
\begin{align*}
L_2
&=-G_{w_1} - \int_\mathbb{R} \frac{R \theta}{ v^2\sqrt{\kappa}} \bar{w}\Bigl( v^{5/2}-\bar{v}^{5/2} \Bigr) \omega \, dx -\int_\mathbb{R}\frac{R}{v} \vartheta_x \omega \, dx\\
&\quad+ \int_\mathbb{R} R\bar v_x\Big(\frac{\theta}{v^2}-\frac{\bar\theta}{\bar v^2}\Big) \omega \, dx -\int_\mathbb{R} R\bar{\theta}_x \Bigl(\frac{1}{v}-\frac{1}{\bar{v}} \Bigr) \omega \, dx \\
&\le -G_{w_1} +C \int_\mathbb{R} |\bar{w}||\phi||\omega|\, dx +C\int_\mathbb{R} |\vartheta_x||\omega|\, dx  +C\int_\mathbb{R} |(\bar{v}_x,\bar{\theta}_x)| |(\phi,\vartheta)| |\omega| \, dx \\ 
&\le -\frac{7}{8}G_{w_1} + C D_{\theta_1} + C \int_\mathbb{R} |\bar{w}|^2 |\phi|^2 \, dx + C \int_\mathbb{R} |(\bar{v}_x,\bar{\theta}_x)|^2 | (\phi,\vartheta)|^2 \,dx \\
&\le -\frac{7}{8}G_{w_1} + C D_{\theta_1} +C \delta_0(G^R+G^S)+ C\delta_0^2 (1+t)^{-1}\int_{\R} e^{-\frac{2C_1|x|^2}{1+t}} |(\phi,\vartheta)|^2\,dx .
\end{align*}

\noindent $\bullet$ (Estimate of $L_3$): We use integration by parts to obtain
\begin{align*}
    L_3
    &= -\mu \int_{\R}  \frac{\psi_x}{v} \omega_x \, dx  -\mu \int_\mathbb{R} \bar{u}_x \Bigl( \frac{1}{v}-\frac{1}{\bar{v}} \Bigr) \omega_x \, dx \le \frac{1}{8} D_{w_1} + C D_{u_1} + C \int_\mathbb{R} |\bar{u}_x|^2 |\phi|^2 \,dx \\
    &\le \frac{1}{8} D_{w_1} + C D_{u_1} +C \delta_0(G^R+G^S)+ C\delta_0^2 (1+t)^{-1}\int_{\R} e^{-\frac{2C_1|x|^2}{1+t}} |\phi|^2\,dx .
\end{align*}

\noindent $\bullet$ (Estimate of $L_4$): Extracting the good term $D_{w_1}$ and using Lemma~\ref{lem:superposition-derivatives}, we obtain
\begin{align*}
    L_4
    &=-D_{w_1} - \sqrt{\kappa} \int_{\R}  \bar{w}_x\Bigl( \frac{1}{v^{5/2}}- \frac{1}{\bar{v}^{5/2}} \Bigr)  \omega_x \, dx \le -\frac{7}{8} D_{w_1} + C\int_\mathbb{R} |\bar{w}_x|^2 |\phi|^2 \,dx \\
    &\le -\frac{7}{8} D_{w_1} +C \delta_0(G^R+G^S)+ C\delta_0^2 (1+t)^{-1}\int_{\R} e^{-\frac{2C_1|x|^2}{1+t}} |\phi|^2\,dx .
\end{align*}

\noindent $\bullet$ (Estimate of $L_5$): Similarly,
\begin{align*}
    L_5
    &= \sqrt{\kappa} \int_{\R} \frac{|\psi_x|^2}{v^{5/2}}\,dx+ \sqrt{\kappa} \int_\mathbb{R} \bar{u}_x \Bigl( \frac{1}{v^{5/2}}- \frac{1}{\bar{v}^{5/2}} \Bigr) \psi_x \, dx \le C D_{u_1} + C \int_\mathbb{R} |\bar{u}_x|^2 |\phi|^2 \, dx \\
    &\le C D_{u_1} +C \delta_0(G^R+G^S)+ C\delta_0^2 (1+t)^{-1}\int_{\R} e^{-\frac{2C_1|x|^2}{1+t}} |\phi|^2\,dx.
\end{align*}

\noindent $\bullet$ (Estimate of $L_6$): Finally, we use the Cauchy--Schwarz inequality with the smallness assumption~\eqref{smallness} to obtain
\begin{align*}
     L_6
     &\le \frac{1}{8} G_{w_1}+ C \bigl( \|Q_1\|_{L^2}^2 + \varepsilon_1 \|Q_3\|_{L^2} \bigr).
\end{align*}

Combining all the estimates for $L_i$ and using the smallness of $\delta_0$ and $\varepsilon_1$, we obtain
\begin{align*}
&\frac{d}{dt} \int_\mathbb{R} \psi \omega \, dx + \frac{1}{2} \left( G_{w_1}+D_{w_1}\right) \\ 
&\le \frac{\delta_S}{2} | \dot{X}|^2 + C \bigl(D_{u_1} +D_{\theta_1} \bigr) + C  \delta_0 \bigl(G^R+G^S  \bigr) \\
&\quad + C\delta_0^2 (1+t)^{-1}\int_{\R} e^{-\frac{2C_1|x|^2}{1+t}} |(\phi,\vartheta)|^2\,dx + C \left(\|Q_1\|_{L^2}^2+ \varepsilon_1 \|Q_3\|_{L^2} \right).
\end{align*}
After integrating the above estimate over $[0,t]$ and using \eqref{est-time-int-interaction} and \eqref{Q1R-L2est}, we obtain
\begin{align*}
&  \int_\mathbb{R} \psi(t,x) \omega(t,x) \, dx + \frac{1}{2} \int_0^t \left(G_{w_1} +D_{w_1} \right)(s) \, ds \\ 
&\le   \int_\mathbb{R} \psi_0 \omega_0 \, dx + \frac{\delta_S}{2} \int_0^t| \dot{X}(s)|^2\,d s + C \int_0^t \bigl( D_{u_1} + D_{\theta_1} \bigr)(s) \, d s \\
&\quad +C \delta_0 \int_0^t \left(G^R + G^S  \right)(s)\, d s+ C\delta_0^2 \int_0^t (1+s)^{-1}\int_{\R} e^{-\frac{2C_1|x|^2}{1+s}} |(\phi,\vartheta)(s,x)|^2\,dx\, ds+ C \delta_0^{1/2},
\end{align*}
which completes the proof of Lemma~\ref{lem:est-1}.

\end{proof}

\begin{lemma}\label{lem:est-2}
Under the hypotheses of Proposition~\ref{prop:apriori}, there exists a constant $C>0$, independent of $\delta_0$ and $\varepsilon_1$, such that for all $t\in[0,T]$,
\begin{equation}\label{eq: est-2}
\begin{aligned}
\int_{\mathbb{R}} &\psi_x(t,x)\,\omega_x(t,x) \, dx + \frac 12\int_{0}^t \left( G_{w_2} + D_{w_2} \right)(s)\, ds \\
&\le \int_{\mathbb{R}} \psi_{0x}\omega_{0x} \, dx + \delta_S \int_0^t |\dot{X}(s)|^2 \, ds + C \int_{0}^{t} ( D_{u_2}+D_{\theta_2})(s)\, ds \\ 
&\quad +C(\varepsilon_1+\delta_0) \int_0^t \bigl(\|\phi_x(s,\cdot)\|_{L^2}^2+G_{w_1}(s)+D_{u_1}(s)+D_{\theta_1}(s) \bigr)\, ds+ C \delta_0\int_0^t (G^R+G^S)(s)\,ds\\
&\quad+ C \delta_0^2\int_0^t \frac{1}{1+s} \int_\mathbb{R} e^{-\frac{2C_1|x|^2}{1+s}} |(\phi,\vartheta)(s,x)|^2 \, dx\, ds +C \delta_0^{1/2},
\end{aligned}
\end{equation}
where
\[
G_{w_2}:=\int_\mathbb{R} \frac{R \theta v^{1/2}}{\sqrt{\kappa}} |\omega_x|^2 \, dx, \qquad
D_{w_2}:=  \int_{\R} \frac{\sqrt{\kappa}}{v^{5/2}} |\omega_{xx}|^2\,dx.
\]
\end{lemma}

\begin{proof}
We differentiate \eqref{eq:perturbation}$_2$ and \eqref{eq:perturbation}$_4$ with respect to $x$, multiply the resulting equations by $\omega_x$ and $\psi_x$, respectively, and integrate over $\R$ to obtain
\begin{align*}
  &\frac{d}{dt} \int_{\R}  \psi_x \omega_x \, dx=\sum_{i=1}^6 N_i,
\end{align*}
where
\begin{align*}
    N_1&:=\dot{X} \int_{\R} \bigl( u^S_{xx} \omega_x +  w^S_{xx} \psi_x \bigr) \, dx, &  N_2&:= - \int_{\R}  (p-\bar{p})_{xx} \omega_x \, dx, \\
      N_3&:= \mu \int_{\R} \Bigl( \frac{u_x}{v}-\frac{\bar{u}_x}{\bar{v}} \Bigr)_{xx} \omega_x \, dx, &
      N_4&:= \sqrt{\kappa} \int_{\R} \Bigl( \frac{w_x}{v^{5/2}}- \frac{\bar{w}_x}{\bar{v}^{5/2}} \Bigr)_{xx}  \omega_x \, dx,\\
      N_5&:= - \sqrt{\kappa} \int_{\R}  \Bigl( \frac{u_x}{v^{5/2}}- \frac{\bar{u}_x}{\bar{v}^{5/2}} \Bigr)_{xx} \psi_x \, dx, &
      N_6&:=- \int_{\R} \bigl(  (Q_1)_x \omega_x + (Q_3)_x \psi_x \bigr) \, dx .
\end{align*}

\medskip
\noindent $\bullet$ (Estimate of $N_1$):  Since $|u^S_{xx}|,|w^S_{xx}| \le C \delta_S |v^S_x| $, we estimate $N_1$ as
\begin{align*}
N_1 
&\le C \delta_S |\dot{X}| (\|v^S_x\|_{L^2}\|\omega_x\|_{L^2}+\|v^S_x\|_{L^2}\|\psi_x\|_{L^2})\\ 
&\le \frac{1}{8} G_{w_2} +C \delta_S^2 |\dot{X}|^2 + C \delta_S^3 D_{u_1}.
\end{align*}

\medskip
\noindent $\bullet$ (Estimate of $N_2$):
Differentiating \eqref{eq:phi_x_w} with respect to $x$, we obtain
\begin{align}\label{eq:phi_xx_w}
\sqrt{\kappa}\, \phi_{xx}
&= -v^{5/2}\omega_x
-(v^{5/2})_x\omega
-\bar w_x(v^{5/2}-\bar v^{5/2})
-\bar w ( v^{5/2}-\bar v^{5/2} )_x.
\end{align}
Using \eqref{eq:ppbar_xx} and \eqref{eq:phi_xx_w}, and Young's inequality, we estimate $N_2$ as
\begin{align*}
N_2
&\le -G_{w_2}
+ C\int_{\mathbb{R}} N_{21}|\omega_x|dx\le -\frac{7}{8}G_{w_2}+C\int_{\R}|N_{21}|^2\,dx,
\end{align*}
where $N_{21}$ is defined as
\begin{align*}
N_{21}
&:= |v_x||\omega|+|\bar w_x||\phi|+|\bar w|\bigl(|\phi_x|+|\bar v_x||\phi|\bigr)+|\vartheta_{xx}|+|\bar v_{xx}|(|\phi|+|\vartheta|)+|\bar\theta_{xx}||\phi| \nonumber\\
&\quad+|\phi_x|(|v_x|+|\bar v_x|)+|\bar v_x|^2(|\phi|+|\vartheta|)+|\phi_x||\vartheta_x|+|\bar v_x||\vartheta_x|+|\bar\theta_x||\phi_x|+|\bar v_x||\bar\theta_x||\phi|.
\end{align*}
Moreover, using $|v_x| \le |\phi_x| + |\bar{v}_x|$ with \eqref{eq:Linf-bounds}, $|\bar w|\le C|\bar v_x|$, and Lemma~\ref{lem:superposition-derivatives}, we obtain
\begin{align*} 
&\int_{\mathbb{R}} |N_{21}|^2\,dx \\
&\le C D_{\theta_2} +C \int_\mathbb{R}  |(v_x,\bar{v}_x,\bar{\theta}_x)|^2 |\phi_x|^2 \, dx+C\int_\mathbb{R} |\bar{v}_x|^2 |\vartheta_x|^2 \, dx + \int_\mathbb{R} |\phi_x|^2 |\vartheta_x|^2\,dx  + C \int_\mathbb{R} |v_x|^2 |\omega|^2 \, dx\\
&\quad + C\int_{\mathbb{R}} \bigl(|(\bar{w}_x,  \bar{v}_{xx},\bar{\theta}_{xx})|^2 +|(\bar{w},\bar{v}_x,\bar{\theta}_x)|^2|\bar{v}_x|^2 \bigr)|\phi|^2\,dx+ C \int_\mathbb{R} \bigl(|\bar{v}_{xx}|^2 + |\bar{v}_x|^4 \bigr) |\vartheta|^2 \, dx \\
&\le C D_{\theta_2} +C (\varepsilon_1+\delta_0) \bigl( \|\phi_x\|_{L^2}^2 + D_{\theta_1}+G_{w_1} \bigr) + C \delta_0(G^R+G^S)\\
&\quad + C \delta_0^2 (1+t)^{-1} \int_\mathbb{R} e^{-\frac{2C_1|x|^2}{1+t}} |(\phi,\vartheta)|^2 \, dx .
\end{align*}
Consequently,
\begin{align*}
N_2
&\le -\frac{7}{8}G_{w_2} + C D_{\theta_2} +C (\varepsilon_1+\delta_0) \bigl( \|\phi_x\|_{L^2}^2 + D_{\theta_1}+G_{w_1} \bigr) \\
&\quad + C \delta_0(G^R+G^S) + C \delta_0^2 (1+t)^{-1} \int_\mathbb{R} e^{-\frac{2C_1|x|^2}{1+t}} |(\phi,\vartheta)|^2 \, dx .
\end{align*}

\medskip
\noindent $\bullet$ (Estimate of $N_3$): We integrate by parts to obtain
\begin{align*}
    N_3 &\le C \int_\mathbb{R} |\psi_{xx}| |\omega_{xx}| \, dx + C \int_\mathbb{R} |v_x| |\psi_x| |\omega_{xx}|\,dx + C \int_\mathbb{R} |\bar{u}_{xx}| |\phi| |\omega_{xx}|\,dx\\
    &\quad + C \int_\mathbb{R} |\bar{u}_x| \bigl(|\phi_x| + |\bar{v}_x| |\phi| \bigr)  |\omega_{xx}|\,dx\\
    &\le \frac{1}{8} D_{w_2} + C D_{u_2}+ C\int_\mathbb{R} |v_x|^2 |\psi_x|^2 \,dx+ C\int_\mathbb{R} |\bar{u}_x|^2 |\phi_x|^2 \, dx + C \int_\mathbb{R} \bigl( |\bar{u}_{xx}|^2 + |\bar{u}_x|^2 |\bar{v}_x|^2 \bigr) |\phi|^2 \, dx \\
    &\le \frac{1}{8} D_{w_2} + C D_{u_2}+ C (\varepsilon_1 + \delta_0) D_{u_1} + C \delta_0 \bigl( \|\phi_x\|_{L^2}^2 + G^R+G^S\bigr)\\
    &\quad + C \delta_0^2 (1+t)^{-1} \int_\mathbb{R} e^{-\frac{2C_1|x|^2}{1+t}} |\phi|^2 \, dx .
\end{align*}

\medskip
\noindent $\bullet$ (Estimate of $N_4$): We estimate $N_4$ as
\begin{align*}
    N_4 &\le -D_{w_2} + C \int_\mathbb{R} |v_x| |\omega_x||\omega_{xx}|\,dx + C \int_\mathbb{R} |\bar{w}_{xx}||\phi||\omega_{xx}|\, dx + C \int_\mathbb{R} |\bar{w}_x| \bigl( |\phi_x|  + |\bar{v}_x||\phi| \bigr) |\omega_{xx}| \, dx \\
    &\le -\frac{7}{8}D_{w_2} + C \int_\mathbb{R} |v_x|^2 |\omega_x|^2 \, dx +C \int_\mathbb{R} |\bar{w}_x|^2 |\phi_x|^2 \, dx + C \int_\mathbb{R} \bigl( |\bar{w}_{xx}|^2 + |\bar{w}_x|^2 |\bar{v}_x|^2 \bigr) |\phi|^2 \, dx \\
    &\le -\frac{7}{8}D_{w_2} + C (\varepsilon_1 +\delta_0) G_{w_2} +C \delta_0 ( \| \phi_x\|_{L^2}^2 +G^R+G^S) + C \delta_0^2 (1+t)^{-1} \int_\mathbb{R} e^{-\frac{2C_1|x|^2}{1+t}} |\phi|^2 \, dx .
\end{align*}

\medskip
\noindent $\bullet$ (Estimate of $N_5$): Similarly, we obtain
\begin{align*}
    N_5 &\le C \int_\mathbb{R} |\psi_{xx}|^2 \, dx + C \int_\mathbb{R} |v_x| |\psi_x| |\psi_{xx}| \, dx + C \int_\mathbb{R} |\bar{u}_{xx}| |\phi| |\psi_{xx}| \,dx\\
    &\quad + C \int_\mathbb{R} |\bar{u}_x| \bigl( |\phi_x| + |\bar{v}_x| |\phi| \bigr) |\psi_{xx} | \, dx \\
    &\le C D_{u_2} + C \int_\mathbb{R} |v_x|^2 |\psi_x|^2 \, dx + C \int_\mathbb{R}|\bar{u}_x|^2 |\phi_x|^2 \, dx + C \int_\mathbb{R} \bigl( |\bar{u}_{xx}|^2 + |\bar{u}_x|^2 |\bar{v}_x|^2 \bigr) |\phi|^2 \, dx\\
    &\le C D_{u_2} + C (\varepsilon_1 +\delta_0) D_{u_1} +C \delta_0 ( \| \phi_x\|_{L^2}^2 +G^R+G^S)\\
    &\quad + C \delta_0^2 (1+t)^{-1} \int_\mathbb{R} e^{-\frac{2C_1|x|^2}{1+t}} |\phi|^2 \, dx .
\end{align*}

\medskip
\noindent $\bullet$ (Estimate of $N_6$): Finally, we estimate $N_6$ as
\begin{align*}
     N_6 &\le \frac{1}{8} G_{w_2}+ C \varepsilon_1 D_{u_1} + C \left( \|(Q_1)_x\|_{L^2}^2 + \frac {1}{\varepsilon_1} \|(Q_3)_x\|_{L^2}^2 \right).
\end{align*}
We decompose $Q_1=Q_1^I+Q_1^R+Q_1^C$. By direct differentiation, we obtain the pointwise bounds
\begin{align*}
|(Q_1^I)_x|
\le C&\Big[
\bigl|\bigl(v^R_{xx},u^R_{xxx},u^R_{xx}v^R_x,u^R_xv^R_{xx},v^R_{xxx}v^R_x,(v^R_{xx})^2,v^R_{xxxx},(v^R_x)^2v^R_{xx} \bigr) \bigr||(v^C-v_*,v^S-v^*)|
\\
&+\bigl|\bigl(v^R_x,u^R_{xx},u^R_xv^R_x,v^R_{xx}v^R_x,v^R_{xxx},(v^R_x)^3\bigr) \bigr|\,|(v^C_x,v^S_x)|\\
&+\bigl|\bigl(v^C_{xx},u^C_{xxx},u^C_{xx}v^C_x,u^C_xv^C_{xx},v^C_{xxx}v^C_x,(v^C_{xx})^2,v^C_{xxxx},(v^C_x)^2v^C_{xx}\bigr) \bigr||(v^R-v_*,v^S-v^*)|\\
&+\bigl|\bigl(v^C_x,u^C_{xx},u^C_xv^C_x,v^C_{xx}v^C_x,v^C_{xxx},(v^C_x)^3\bigr)\bigr|\,|(v^R_x,v^S_x)|\\
&+\bigl|\bigl(v^S_{xx},u^S_{xxx},u^S_{xx}v^S_x,u^S_xv^S_{xx},v^S_{xxx}v^S_x,(v^S_{xx})^2,v^S_{xxxx},(v^S_x)^2v^S_{xx}\bigr) \bigr||(v^R-v_*,v^C-v^*)|\\
&+\bigl|\bigl(v^S_x,u^S_{xx},u^S_xv^S_x,v^S_{xx}v^S_x,v^S_{xxx},(v^S_x)^3\bigr)\bigr|\,|(v^R_x,v^C_x)|\\
&+|(u^R_{xx},v^R_x v^R_{xx},v^R_{xxx})|\,|(v^C_x,v^S_x)| + |(u^R_x,(v^R_x)^2,v^R_{xx})|\,|(v^C_{xx},v^S_{xx})|\\
&+|(u^C_{xx},v^C_x v^C_{xx},v^C_{xxx})|\,|(v^R_x,v^S_x)|
+|(u^C_x,(v^C_x)^2,v^C_{xx})|\,|(v^R_{xx},v^S_{xx})|\\
&+|(u^S_{xx},v^S_xv^S_{xx},v^S_{xxx})|\,|(v^R_x,v^C_x)|+|(u^S_x,(v^S_x)^2,v^S_{xx})|\,|(v^R_{xx},v^C_{xx})|\Big].
\end{align*}
and 
\begin{align*}
|(Q_1^R)_x|\le C\Big(
&|u^R_{xxx}|+|u^R_{xx}|\,|v^R_x|+|u^R_x|\,|v^R_{xx}|+|u^R_x|\,|v^R_x|^2\\
&\quad+|v^R_{xxxx}|+|v^R_x|\,|v^R_{xxx}|+|v^R_{xx}|^2+|v^R_x|^2\,|v^R_{xx}|+|v^R_x|^4\Big).
\end{align*}

By Lemma~\ref{lem: interaction estimates} and the argument for $Q_1^I$ and $Q_3$ in its proof, we get
\begin{align*}
&\int_0^t \|(Q_1^I)_x\|_{L^2}^2\,ds
\le C\delta_0^{1/2},\\
&\int_0^t \|(Q_3)_x\|_{L^2}^2\,ds
\le C\varepsilon_1\delta_0^{1/2}.
\end{align*}
Furthermore, using Lemma~\ref{lem:rarefaction_property}, we obtain
\begin{align*}
\int_0^t \|(Q_1^R)_x\|_{L^2}^2\,ds
&\le C\int_0^t\Bigl(\|v^R_{xxxx}\|_{L^2}^2+\|v^R_{xxx}\|_{L^4}^4+\|u^R_{xxx}\|_{L^2}^2+\|(v^R_{xx},u^R_{xx})\|_{L^4}^4 \nonumber\\
&\hspace{4.5em} +\|(v^R_x,u^R_x)\|_{L^4}^4+\|v^R_x\|_{L^8}^8\Bigr)\,ds\le C\delta_0.
\end{align*}
Finally, we deduce from \eqref{error_contact} that
\[\int_0^t\|(Q_1^C)_x\|_{L^2}^2\,ds \le C\delta_0.\]

Combining all the above estimates on $N_i$, integrating in time over $[0,t]$,
and using the smallness of $\delta_0$ and $\varepsilon_1$, we obtain
\begin{align*}
\int_{\mathbb{R}} &\psi_x\omega_x \, dx+\frac{1}{2} \int_{0}^t \left( G_{w_2} + D_{w_2} \right) \, ds\\
&\le \int_{\mathbb{R}} \psi_{0x}\omega_{0x} \, dx+ \delta_S \int_0^t |\dot{X}|^2 \, d s + C \int_{0}^{t} ( D_{u_2}+D_{\theta_2})\, ds\\ 
&\quad +C(\varepsilon_1+\delta_0) \int_0^t \bigl(\|\phi_x\|_{L^2}^2+G_{w_1}+D_{u_1}+D_{\theta_1} \bigr) \, d s +C \delta_0\int_0^t (G^R+G^S) \,ds\\
&\quad+ C \delta_0^2\int_0^t \frac{1}{1+s} \int_\mathbb{R} e^{-\frac{2C_1|x|^2}{1+s}} |(\phi,\vartheta)|^2 \, dx\,ds +C \delta_0^{1/2},
\end{align*}
which implies the desired estimate \eqref{eq: est-2}.
\end{proof}

\subsection{Proof of Proposition~\ref{prop:apriori}}
We now derive \eqref{eq:H1-est} by combining \eqref{eq: est-0}, \eqref{eq: est-1}, and \eqref{eq: est-2} with the $L^2$-estimate \eqref{est:L2}. For the reader's convenience, we decompose the argument into three steps.

\medskip
\noindent\textbf{Step 1.}
Applying Young's inequality to \eqref{eq: est-1}, we obtain, for some constant $C_2>0$,
\begin{equation}\label{eq: est-1,1}
\begin{aligned}
	&\int_0^t \bigl(G_{w_1} + D_{w_1}\bigr)\, ds \\
	&\le \|(\psi,\omega)(t)\|_{L^2}^2 +\|(\psi_{0},\omega_{0})\|_{L^2}^2 + C_2 \int_0^t  (\delta_S|\dot{X}(s)|^2+D_{u_1} +D_{\theta_1}) \,d s\\
	&\quad +C \delta_0\int_0^t (G^R+G^S) \,ds+ C \delta_0^2\int_0^t \frac{1}{1+s} \int_\mathbb{R} e^{-\frac{2C_1|x|^2}{1+s}} |(\phi,\vartheta)|^2 \, dx\,ds +C \delta_0^{1/2}.
\end{aligned}
\end{equation}

In order to bound the integral of $\|\phi_x\|_{L^2}^2$ appearing on the right-hand side of \eqref{est:L2}, we use \eqref{eq:phi_x_w} to deduce
\begin{equation}\label{eq:phi_x_omega}
	\|\phi_x\|_{L^2}^2 \le C G_{w_1} + C \delta_0 (G^R+G^S) + C \delta_0^2 (1+t)^{-1}\int_\mathbb{R} e^{-\frac{2C_1|x|^2}{1+t}} |(\phi,\vartheta)|^2 \, dx.
\end{equation}
Integrating \eqref{eq:phi_x_omega} in time over $[0, t]$ and utilizing the heat kernel estimate \eqref{heat-int}, we explicitly obtain
\begin{equation}\label{eq:phi_x_int}
\begin{aligned}
    \int_0^t \|\phi_x\|_{L^2}^2 \,ds &\le C \int_0^t G_{w_1} \,ds  +  C\delta_0 \int_0^t (G^R+G^S+D_{u_1}+D_{\theta_1}) \,ds \\
    &\quad + C\delta_0^2 \sup_{s\in[0,t]} \|(\phi,\psi,\vartheta,\omega)\|_{L^2}^2 +C\delta_0^2\delta_S \int_0^t |\dot{X}(s)|^2\,ds +C\delta_0^2\int_0^t D_{w_1}(s)\,ds +C\delta_0^2.
\end{aligned}
\end{equation}

Now, we multiply \eqref{eq: est-1,1} by $\frac{1}{2\max\{1,C_2\}}$ and add it to \eqref{est:L2}. Since $D=D_{u_1}+D_{\theta_1}$, the left-hand side of \eqref{est:L2} absorbs $\|(\psi,\omega)(t)\|_{L^2}^2$ and the $C_2$-term of \eqref{eq: est-1,1}. By \eqref{eq:phi_x_int} and \eqref{heat-int}, the remaining terms carry the small factor $\delta_0$ or $\varepsilon_1$, and we obtain
\begin{equation}\label{est-3}
\begin{aligned}
	&\sup_{t\in[0,T]}\|(\phi,\psi,\vartheta,\omega)\|_{L^2}^2
	+\delta_S\int_0^T |\dot{X}(t)|^2\, dt +\int_0^T \bigl(G^S+G^R+G_3+D_{u_1}+D_{\theta_1}+G_{w_1}+D_{w_1}\bigr)\, dt\\
	&\le C\|(\phi_0,\psi_0,\vartheta_0,\omega_0)\|_{L^2}^2 +C\delta_0^{1/6}.
\end{aligned}
\end{equation}

\medskip
\noindent\textbf{Step 2.}
Similarly, we apply Young's inequality to \eqref{eq: est-2} to observe that for some constant $C_3>0$,
\begin{equation}\label{eq: est-2,1}
\begin{aligned}
	&\int_{0}^t \bigl(G_{w_2} + D_{w_2} \bigr)\, ds \\
	&\le \|(\psi_x,\omega_x)(t)\|_{L^2}^2 +\|(\psi_{0x},\omega_{0x})\|_{L^2}^2 + C_3 \int_{0}^{t} (\delta_S|\dot{X}(s)|^2+D_{u_2}+D_{\theta_2})\, ds \\
	&\quad +C(\varepsilon_1+\delta_0) \int_0^t \bigl(\|\phi_x\|_{L^2}^2+G_{w_1}+D_{u_1}+D_{\theta_1} \bigr) \, d s +C \delta_0\int_0^t (G^R+G^S) \,ds\\
    &\quad+ C \delta_0^2\int_0^t \frac{1}{1+s} \int_\mathbb{R} e^{-\frac{2C_1|x|^2}{1+s}} |(\phi,\vartheta)|^2 \, dx\,ds +C \delta_0^{1/2}.
\end{aligned}
\end{equation}
Multiplying \eqref{eq: est-2,1} by $\frac{1}{2\max\{1,C_3\}}$ and adding it to \eqref{eq: est-0}, we apply \eqref{eq:phi_x_int} again and use $D_{w_1}\le C\,G_{w_2}$ to deduce that for some constant $C_4>0$,
\begin{equation}\label{eq: est-4}
	\begin{aligned}
		&\|(\psi_x, \vartheta_x,\omega_x)\|_{L^2}^2 + \int_{0}^t  \bigl(D_{u_2}+D_{\theta_2}+G_{w_2}+D_{w_2}\bigr)\, ds \\
		&\le C \|(\psi_{0x},\vartheta_{0x},\omega_{0x})\|_{L^2}^2 + C_4 \int_0^t \bigl(\delta_S|\dot{X}(s)|^2+D_{u_1}+D_{\theta_1}+G_{w_1}\bigr) \, ds \\
		&\quad + C\delta_0^2\sup_{s\in[0,t]}\|(\phi,\psi,\vartheta,\omega)(s,\cdot)\|_{L^2}^2+ C \delta_0^{1/2}\int_0^t (G^R+G^S) \,ds\\
		&\quad + C \delta_0\int_0^t \frac{1}{1+s} \int_\mathbb{R} e^{-\frac{3C_1|x|^2}{2(1+s)}} |(\phi,\vartheta)|^2 \, dx\,ds +C \delta_0^{1/2}.
	\end{aligned}
\end{equation}

\medskip
\noindent\textbf{Step 3.}
Finally, we multiply \eqref{eq: est-4} by $\frac{1}{2\max\{1,C_4\}}$ and add it to \eqref{est-3}. The left-hand side of \eqref{est-3} absorbs the $C_4$-term of \eqref{eq: est-4}, and \eqref{heat-int}, \eqref{eq:phi_x_int}, and the smallness of $\delta_0$ and $\varepsilon_1$ absorb the remaining terms. Since $\|\phi_x\|_{L^2}\le C\|\omega\|_{L^2}+C\delta_0\|\phi\|_{L^2}$ by \eqref{eq:phi_x_w} and \eqref{eq:Linf-bounds-2}, we arrive at
\begin{equation*}
\begin{aligned}
& \|(\phi,\psi,\vartheta,\omega)\|_{H^1}^2+ \delta_S\int_0^t |\dot{X}(s)|^2\,ds+\int_0^t \bigl(G^R+G^S+G_3\bigr)\,ds\\
&\quad +\int_0^t \bigl(D_{u_1}+D_{\theta_1}+D_{u_2}+D_{\theta_2}+G_{w_1}+G_{w_2}+D_{w_1}+D_{w_2}\bigr)\,ds \\
&\le C_0\|(\phi_0,\psi_0,\vartheta_0,\omega_0)\|_{H^1}^2+C\delta_0^{1/6},
\end{aligned}
\end{equation*}
which implies \eqref{eq:H1-est} and completes the proof of Proposition \ref{prop:apriori}.

\appendix

\section{Proof of \eqref{heat-int}}\label{app:proof_heat}
\setcounter{equation}{0}
	In this section, we provide the detailed proof of \eqref{heat-int}. We mainly follow the argument in \cite{KVW25}, together with additional control on the capillarity terms. With the perturbation \eqref{def:perturbation}, define
	\[W(t,x):=\frac{1}{\sqrt{1+t}}e^{-\frac{\beta|x|^2}{1+t}},\quad \beta:=\frac{C_1}{2},\quad h(t,x):=\int_{-\infty}^xW(t,y)\,dy,\quad H(t,x):=\int_{-\infty}^{x} W^2(t,y)\,dy.\]
	We note that $W$ is a solution to the heat equation $W_t=\frac{1}{4\beta}W_{xx}$ subject to the initial data $W(0,x)=e^{-\beta|x|^2}$. Therefore, $h$ satisfies the same heat equation $h_t = \frac{1}{4\beta}h_{xx}$. Let $\nu\in(0,1)$, and let $C_\nu$ denote a constant depending on $\nu$. Our goal is to show the following two estimates
	\begin{align}
		\begin{aligned}\label{est-heat-1}
		\int_0^t\int_{\R}&\left(\bar{p}\phi+\frac{R}{\gamma-1}\vartheta\right)^2W^2 \,dx\,ds\le C\sup_{s\in[0,t]}\|(\phi,\vartheta)(s,\cdot)\|_{L^2}^2\\
		&\quad +C\delta_S\int_0^t|\dot{X}|^2\,ds+(\nu+C\delta_0)\int_0^t\int_\R W^2|(\phi,\vartheta)|^2\,dx\,ds\\
		&\quad +C_\nu\int_0^t \|(\phi,\psi,\vartheta)_x\|_{L^2}^2\,ds+C\int_0^t(G^S+G^R)\,ds+C\delta_0^{1/6},
		\end{aligned}
	\end{align}
	and
	\begin{align}
		\begin{aligned}\label{est-heat-2}
		\int_0^t\int_{\R}&\left[\frac{(R\vartheta-\bar{p}\phi)^2}{2v}+\frac{\gamma\bar{p}\psi^2}{2}\right]W^2 \,dx\,ds\le C\sup_{s\in[0,t]}\|(\phi,\psi,\vartheta)(s,\cdot)\|_{L^2}^2\\
		&\quad + C\delta_S\int_0^t|\dot{X}|^2\,ds+(\nu+C\delta_0)\int_0^t\int_{\R} W^2|(\phi,\vartheta)|^2\,dx\,ds\\
		&\quad +C_\nu\int_0^t\|(\phi,\psi,\vartheta,\omega)_x\|_{L^2}^2\,ds+C(\delta_0+\varepsilon_1)\int_0^t(G^S+G^R)\,ds+C\delta_0^{1/6}.
		\end{aligned}
	\end{align}
	Once the estimates \eqref{est-heat-1} and \eqref{est-heat-2} are established, \eqref{heat-int} can be derived from \eqref{est-heat-1} and \eqref{est-heat-2} by following the same argument as in \cite{KVW25}. To show \eqref{est-heat-1}, we introduce $F:=\bar{p}\phi+ \frac{R}{\gamma-1}\vartheta$. Then, we follow the same lines as in \cite[Appendix D]{KVW25} to derive
	\begin{align}
		\begin{aligned}\label{est:F}
			&\frac{1}{4\beta} \int_0^t\int_\R F^2W^2 \,d x\,d s=\int_{\R} F(0,x)^2\frac{h_0^2}{2}\,dx\\
			&\quad - \int_{\R} F^2\frac{h^2}{2}\,dx +\int_0^t \dot{X}(s)\int_{\R} \left(v^S_x\bar{p}+\frac{R}{\gamma-1}\theta^S_x\right)Fh^2\,d x\,d s-\int_0^t \int_{\R} \frac{1}{4\beta}(F^2)_x hW\,dx\,ds\\
			&\quad-\int_0^t \int_\R (p-\bar{p})\psi_xFh^2 \,dx\,ds +\int_0^t\int_{\R} \left(\bar{p}_t \phi -\bar{u}_x (p-\bar{p})+\mu\left(\frac{u_x^2}{v}-\frac{\bar{u}_x^2}{\bar{v}}\right)-Q_2\right)Fh^2\,dx\,ds\\
			&\quad -\int_0^t\int_{\R} \alpha\left(\frac{\theta_x}{v}-\frac{\bar{\theta}_x}{\bar{v}}\right)(Fh^2)_x\,dx\,ds=:\sum_{i=1}^7 \mathcal{J}_i.
		\end{aligned}
	\end{align}
	Since the estimates of $\mathcal{J}_1$, $\mathcal{J}_2$, and $\mathcal{J}_3$ are the same as in \cite{KVW25}, we adopt the results as
	\begin{align*}
		|\mathcal{J}_1|+|\mathcal{J}_2|\le C\left(\|(\phi_0,\vartheta_0)\|_{L^2}^2+\|(\phi,\vartheta)(t,\cdot)\|_{L^2}^2\right),\quad |\mathcal{J}_3|\le \delta_S\int_0^t |\dot{X}(s)|^2\,ds +C\delta_S\int_0^t G^S\,ds.
	\end{align*}
	Next, we use the Cauchy-Schwarz inequality and $\e_R< \delta_0$ to estimate $\mathcal{J}_4$ as
	\begin{align*}
		\mathcal{J}_4&\le \frac{1}{16\beta}\int_0^t\int_{\R} F^2W^2\,dx\,ds+C\int_0^t\int_{\R} \left(|(\phi_x,\vartheta_x)|^2+|\bar{p}_x\phi|^2\right)\,dx\,ds\\
		&\le \frac{1}{16\beta}\int_0^t\int_{\R} F^2W^2\,dx\,ds+C\int_0^t\int_{\R} |(\phi_x,\vartheta_x)|^2\,dx\,ds\\
		&\quad +C\delta_0\int_0^t\int_{\R} W^2|\phi|^2\,dx\,ds+C\delta_0\int_0^t(G^R+G^S)\,ds.
	\end{align*}
	We manipulate $\mathcal{J}_5$ as
	\begin{align*}
		\mathcal{J}_5 &= \int_0^t\int_{\R}\left(\frac{\gamma\bar{p}}{2v}\phi^2Fh^2-\frac{\gamma-1}{v}\phi F^2h^2\right)_t \,dx\,ds- \int_0^t\int_{\R}\left(\frac{\gamma\bar{p}}{v}\phi^2 F-\frac{2(\gamma-1)}{v}\phi F^2\right)hh_t\,dx\,ds\\
		&-\int_0^t\int_{\R}\left(\frac{\gamma\bar{p}}{v}F\right)_t \frac{\phi^2}{2}h^2\,dx\,ds+\int_0^t\int_{\R}\left(\frac{\gamma-1}{v}F^2\right)_t\phi h^2\,dx\,ds-\int_0^t\dot{X}\int_{\R}v^S_x\frac{\bar{p}\phi-R\vartheta}{v}Fh^2\,dx\,ds\\
		&=:\sum_{i=1}^5 \mathcal{J}_{5i}.
	\end{align*}
	Again, we adopt the estimates of $\mathcal{J}_{51}$, $\mathcal{J}_{52}$, and $\mathcal{J}_{55}$ from \cite{KVW25} as
	\begin{align*}
		&\mathcal{J}_{51}\le C\sup_{s\in[0,t]}\|(\phi,\vartheta)(s,\cdot)\|_{L^2}^2,\quad \mathcal{J}_{52}\le C\int_0^t \|\phi_x\|_{L^2}^2\,ds + C\varepsilon_1^{4/3}\sup_{s\in[0,t]}\|(\phi,\vartheta)(s,\cdot)\|_{L^2}^2,\\
		&\mathcal{J}_{55}\le \delta_S\int_0^t |\dot{X}(s)|^2\,ds + C(\delta_0+\varepsilon_1)\int_0^t G^S\,ds.
	\end{align*}
	To estimate $\mathcal{J}_{53}$, we further split it as
	\begin{align*}
		\mathcal{J}_{53}&=-\int_0^t \int_{\R} F_t \frac{\gamma\bar{p}\phi^2h^2}{2v}\,dx\,ds-\int_0^t\int_{\R}\left(\frac{\bar{p}}{v}\right)_tF\frac{\gamma\phi^2h^2}{2}\,dx\,ds\\
		&\le-\int_0^t\int_{\R} \dot{X}(s)\left(v^S_x\bar{p}+\frac{R}{\gamma-1}\theta^S_x\right)\frac{\gamma \bar{p}\phi^2h^2}{2v}\,dx\,ds+\alpha\int_0^t\int_{\R} \left(\frac{\theta_x}{v}-\frac{\bar{\theta}_x}{\bar{v}}\right)\left(\frac{\gamma\bar{p}\phi^2h^2}{2v}\right)_x\,dx\,ds\\
		&\quad -\int_0^t\int_{\R} \left[-(p-\bar{p})\psi_x - \bar{u}_x (p-\bar{p})+\mu\left(\frac{u_x^2}{v}-\frac{\bar{u}_x^2}{\bar{v}}\right)-Q_2\right]\frac{\gamma\bar{p}\phi^2h^2}{2v}\,dx\,ds\\
		&\quad + C\int_0^t\int_{\R} \left(|\bar{p}_t|+|v_t|\right)|(\phi,\vartheta)||\phi|^2\,dx\,ds.
	\end{align*}
	The only different estimate on $\mathcal{J}_{53}$ compared to \cite{KVW25} is the term regarding $Q_2$. Since $h$ is bounded, we use a similar estimate as in \eqref{est:S1S2} to get
	\begin{align*}\int_\R \left|Q_2\frac{\gamma\bar{p}\phi^2 h^2}{2v}\right|\,dx&\le C\int_{\R} |Q_2||\phi|^2\,dx\le C\varepsilon_1\int_{\R} |Q_2||\phi|\,dx\\
	&\le C\varepsilon_1^2\|(Q_1^I,Q_1^C,Q_2^I,Q_2^C)\|_{L^2}+C\varepsilon_1^{5/3}\|(Q_1^R,Q_2^R)\|_{L^1}^{4/3}+C\e_1\|\phi_x\|_{L^2}^2.
	\end{align*}
	Thus, we use Lemma \ref{lem: interaction estimates}, \eqref{error_contact} and $\e_R<\delta_0$ to obtain
	\begin{align*}
		\mathcal{J}_{53}&\le C\delta_S \int_0^t |\dot{X}(s)|^2\,ds+C(\delta_0+\varepsilon_1)\int_0^t (G^R+G^S+D+\|\phi_x\|_{L^2}^2)\,ds\\
		&\quad + C\delta_C\int_0^t \int_{\R} W^2|(\phi,\vartheta)|^2\,dx\,ds+C(\delta_0^{1/4}+\delta_0^{\frac{q-8}{3(q-2)}}). 
	\end{align*}
	A similar estimate holds for $\mathcal{J}_{54}$ and $\mathcal{J}_6$. Finally, using $\e_R<\delta_0$, we estimate $\mathcal{J}_7$ as in \cite{KVW25}:
	\[\mathcal{J}_7\le (\nu+C\delta_0)\int_0^t\int_{\R} W^2|(\phi,\vartheta)|^2\,dx\,ds+C_\nu\int_0^t \|(\phi,\vartheta)_x\|_{L^2}^2\,ds+C\delta_0\int_0^t (G^R+G^S)\,ds.\]
	Collecting all the estimates on $\mathcal{J}_i$ with \eqref{est:F}, and using $q \ge 14$, we complete the proof of \eqref{est-heat-1}.
		
	For \eqref{est-heat-2}, we again follow the same estimate as in \cite{KVW25}, paying particular attention to the additional capillarity term, to obtain
	\begin{align*}
		&\int_0^t\int_\R W^2\Big[\frac{(R\vartheta-\bar p \phi)^2}{2v}+\frac{\gamma\bar p\psi^2}{2}\Big]\,dx\,ds=\int_\R H\psi(R\vartheta-\bar p \phi)|^{s=t}_{s=0}\,dx
		\\
		&- \int_0^t\int_\R  v_x\frac{H(R\vartheta-\bar p \phi)^2}{2v^2}\,dx\,ds  -\int_0^t\int_\R \psi(R\vartheta-\bar p \phi)H_t\,dx\,ds\\
		&+ \int_0^t\dot{X}(s)\int_\R H\big[\bar p \psi v^S_x-R\psi \theta^S_x-(R\vartheta-\bar p \phi) u^S_x\big]\,dx\,ds\\
		&-\int_0^t\int_\R \bar p_x H\frac{\gamma \psi^2}{2}\,dx\,ds+ \int_0^t\int_\R (\gamma-1)H\psi\big[(p-\bar p)\psi_x+\bar u_x(p-\bar p)\big]\,dx\,ds\\
		&+ \int_0^t\int_\R H\psi\phi\bar{p}_t\,dx\,ds  - \int_0^t\int_\R \mu (\gamma-1)H\psi \Bigl( \frac{(u_x)^2}{v} -\frac{ (\bar u_x)^2}{\bar v}
		\Bigr)\,dx\,ds \\ 
		& + \int_0^t\int_\R    \alpha(\gamma-1)(H\psi)_x \Bigl( \frac{\theta _x
		}{v}-\frac{\bar \theta_x}{\bar v} \Bigr)\,dx\,ds + \int_0^t\int_\R \mu \big(H(R\vartheta-\bar p \phi)\big)_x \Bigl(\frac{u_x}{v}-\frac{\bar u_x}{\bar v}\Bigr)\,dx\,ds  \\
		&+ \int_0^t\int_\R \sqrt{\kappa} \big(H(R\vartheta-\bar p \phi)\big)_x \Bigl(\frac{w_x}{v^{5/2}}-\frac{\bar w_x}{\bar v^{5/2}}\Bigr)\,dx\,ds + \int_0^t\int_\R \big[(\gamma-1)H\psi Q_2 +H(R\vartheta-\bar p \phi) Q_1\big]\,dx\,ds.
	\end{align*}
	Therefore, compared to \cite{KVW25}, the only additional term that we need to control is
	\[\int_0^t\int_{\R} \sqrt{\kappa}(H(R\vartheta-\bar{p}\phi))_x\left(\frac{w_x}{v^{5/2}}-\frac{\bar{w}_x}{\bar{v}^{5/2}}\right)\,dx\,ds=:\mathcal{I}\]
	We use $H_x = W^2$ and $|H|\le C(1+t)^{-1/2}$ to obtain
	\begin{align*}
		\mathcal{I}&\le C\int_0^t\int_{\R} \left(W^2|(\phi,\vartheta)|+|\vartheta_x|+|\bar{p}_x||\phi|+|\phi_x|\right)\left(|\omega_x|+|\bar{w}_x||\phi|\right)\,dx\,ds
	\end{align*}
	Now, using Lemma~\ref{lem:superposition-derivatives}, we have
	\begin{align*}
		\mathcal{I}\le (\nu+C\delta_0)\int_0^t\int_\R W^2|(\phi,\vartheta)|^2\,dx\,ds+C_\nu \int_0^t\|(\phi_x,\vartheta_x,\omega_x)\|_{L^2}^2\,ds+C\delta_0\int_0^t(G^S+G^R)\,ds.
	\end{align*}
	The other terms can be controlled by similar estimates to those for \eqref{est-heat-1}, and thus we obtain the desired estimate \eqref{est-heat-2}.
\section*{Statements and Declarations}

\subsection*{Acknowledgments}
The authors are grateful to Professor Moon-Jin Kang for his valuable advice.

\subsection*{Funding}
S. Han has been partially supported by the National Research Foundation of Korea (RS-2024-00361663 and NRF-2019R1A5A1028324) and by the Miwon Du-Myeong Fellowship. J. Kim has been supported by the National Research Foundation of Korea (NRF) grant funded by the Korea government (MSIT) (RS-2026-25590471).

\subsection*{Competing interests}
The authors have no competing interests to declare that are relevant to the content of this article.

\subsection*{Data availability}
Data sharing is not applicable to this article as no datasets were generated or analysed during the current study.

\bibliographystyle{acm}
\bibliography{reference}

\end{document}